\documentclass[a4paper,12pt]{amsart}

\usepackage{tikz-cd}

\usepackage{times}

\usepackage[latin1]{inputenc} 
\usepackage{graphicx}
\usepackage[matrix,arrow,tips,curve]{xy}
\usepackage[english]{babel}
\usepackage{amsmath}
\usepackage{amssymb}

\usepackage{mathrsfs}

\usepackage{enumerate,bm}

\newtheorem{thm}{Theorem}[section]

\newtheorem{lemma}[thm]{Lemma}
\newtheorem{thmdef}[thm]{Theorem - Definition}
\newtheorem{corollary}[thm]{Corollary}
\newtheorem{proposition}[thm]{Proposition}

\newtheorem{question}[thm]{Question}

\newtheorem*{thm*}{Theorem}

\theoremstyle{definition}
\newtheorem{definition}[thm]{Definition}

\newtheorem{example}[thm]{Example}
\newtheorem{remark}[thm]{Remark}
\newtheorem{prg}[thm]{}

\renewcommand{\theequation}{\thethm}
\numberwithin{figure}{section}
\numberwithin{table}{section}

\newcommand{\ph}{\varphi}
\newcommand{\w}{\widetilde}
\newcommand{\ma}{\mathcal}
\newcommand{\la}{\longrightarrow}
\newcommand{\ol}{\mathcal{O}}
\newcommand{\wi}{\widehat}
\newcommand{\pr}{\mathbb{P}}
\newcommand{\Q}{\mathbb{Q}}
\newcommand{\C}{\mathbb{C}}
\newcommand{\R}{\mathbb{R}}
\newcommand{\Z}{\mathbb{Z}}
\newcommand{\N}{\mathcal{N}_1}
\newcommand{\Nu}{\mathcal{N}^1}
\newcommand{\Gr}{\operatorname{Gr}}
\newcommand{\dom}{\operatorname{dom}}
\newcommand{\Tors}{\operatorname{Tors}}

\newcommand{\Sing}{\operatorname{Sing}}
\newcommand{\rk}{\operatorname{rk}}
\newcommand{\Pic}{\operatorname{Pic}}

\newcommand{\NE}{\operatorname{NE}}
\newcommand{\Exc}{\operatorname{Exc}}

\newcommand{\Lo}{\operatorname{Locus}}
\newcommand{\codim}{\operatorname{codim}}
\newcommand{\Eff}{\operatorname{Eff}}
\newcommand{\Nef}{\operatorname{Nef}}

\newcommand{\Mov}{\operatorname{Mov}}

\newcommand{\Bl}{\operatorname{Bl}}
\newcommand{\Br}{\operatorname{Br}}
\newcommand{\Bs}{\operatorname{Bs}}

\newcommand{\Cl}{\operatorname{Cl}}
\newcommand{\pts}{\text{\it pts}}
\newcommand{\sm}{\text{\it sm}}
\newcommand{\pt}{\text{\it pt}}
\newcommand{\reg}{\text{\it reg}}
\newcommand{\intr}{\text{\it intr}}
\newcommand{\s}{\scriptscriptstyle}

\usepackage{etoolbox}
\patchcmd{\section}{\normalfont}{\normalfont\large}{}{}
\patchcmd{\subsection}{\bfseries}{\scshape\centering}{}{}
\patchcmd{\subsection}{-.5em}{.5em}{}{}

\makeatletter
\@namedef{subjclassname@2020}{\textup{2020} Mathematics Subject Classification}
\makeatother

\calclayout

\title{Classifying Fano $4$-folds with a rational fibration onto a $3$-fold}
\author{C.~Casagrande}
\address{Universit\`a di Torino,
Dipartimento di Matematica,
via Carlo Alberto 10,
10123 Torino - Italy}
\email{cinzia.casagrande@unito.it}
\author{S.A.~Secci}
\address{Universit\`a di Milano,
Dipartimento di Matematica,
via Saldini 50,
20133 Milano - Italy}
\email{ saverio.secci@unimi.it}

\calclayout 

\begin{document}
 \maketitle
\setcounter{tocdepth}{1}
\renewcommand{\theequation}{\thethm}
\begin{abstract}
We study smooth, complex Fano $4$-folds $X$ with a rational contraction onto a $3$-fold, 
namely a rational map $X\dasharrow Y$ that factors as a sequence of flips $X\dasharrow \w{X}$ followed by a surjective morphism $\w{X}\to Y$ with connected fibers, where $Y$ is normal, projective, and $\dim Y=3$.

We show that if $X$ has a rational contraction onto a $3$-fold and $X$ is not a product of del Pezzo surfaces, then the Picard number $\rho_X$ of $X$ is at most $9$; this bound is sharp.

As an application, we show that every Fano $4$-fold $X$ with $\rho_X=12$ is isomorphic to a product of surfaces, thus improving the result by the first named author that shows the same for $\rho_X>12$.

We also give a classification result for Fano $4$-folds $X$,  not products of surfaces, having a ``special'' rational contraction $X\dasharrow Y$ with $\dim Y=3$, $\rho_X-\rho_Y=2$, and $\rho_X\geq 7$; we show that there are only three possible families. Then we prove that the first family exists if  $\rho_X=7$, and that the second family exists if and only if $\rho_X=7$. This provides the first examples of Fano $4$-folds with $\rho_X\geq 7$ different from
 products of del Pezzo surfaces and from the Fano models of $\Bl_{\pts}\pr^4$. We also construct three new families with $\rho_X=6$.

Finally we show that if a Fano $4$-fold $X$ has Lefschetz defect $\delta_X=2$, then $\rho_X\leq 6$;  this bound is again sharp.
\end{abstract}  
\section{Introduction}
\noindent 
Let $X$ be a smooth, complex Fano $4$-fold, and $\rho_X$ its Picard number. Recall that since $X$ is Fano, $\rho_X$ coincides with the second Betti number of $X$, therefore it is a topological invariant and is constant in smooth families. It is well known that there are finitely many families of Fano $4$-folds, and it is a recent result that when the Picard number is large, $X$ must be a product of surfaces: 
\begin{thm}[\cite{32}, Th.~1.1]\label{inv}
Let $X$ be a smooth Fano $4$-fold. If $\rho_X>12$, then $X\cong S_1\times S_2$, where $S_i$ are del Pezzo surfaces.
\end{thm}
Let us point out that all known examples of Fano $4$-folds that are not products have $\rho\leq 9$, hence it is not known whether for $\rho=10,11,12$ there are only products of surfaces. For the case $\rho=12$, there is the following partial result.
\begin{thm}[\cite{32}, Th.~4.3]\label{addendum}
  Let $X$ be a smooth Fano $4$-fold. If $\rho_X=12$ and $X$  is not isomorphic to a product of surfaces, then there is a rational contraction $X\dasharrow Y$ with $\dim Y=3$.
\end{thm}
A \emph{rational contraction} is a rational map $X\dasharrow Y$ that factors as $X\stackrel{\xi}{\dasharrow} \w{X}\stackrel{f}{\to} Y$ where $\xi$ is a sequence of flips, $f$ is surjective with connected fibers, and $Y$ is normal and projective (namely $f$ is a contraction).

In this paper we study Fano $4$-folds $X$ having a rational contraction $X\dasharrow Y$ with $\dim Y=3$, and such that $X$ is not isomorphic to a product of surfaces, as above.
We have different motivations and goals: first, to describe the geometry of $X$, especially when $\rho_X$ is large, and possibly to give classification results; then to determine a sharp bound on $\rho_X$, and to use it to show that Fano $4$-folds with Picard number 12 are products of surfaces; to construct new families of Fano $4$-folds with large Picard number (at least 6); and finally to give a sharp bound on the Picard number of Fano $4$-folds with Lefschetz defect $2$ (see below).
As an outcome we get results in different directions, let us explain them separately.

\medskip

\noindent{\bf Bounding the Picard number.} Our first goal is to bound the Picard number of Fano $4$-folds with a rational contraction onto a $3$-fold, and our result is the following.
\begin{thm}\label{general}
  Let $X$ be a smooth Fano $4$-fold  that is not isomorphic to a product of surfaces, and having a rational contraction
 $X\dasharrow Y$ with $\dim Y=3$. Then $\rho_X\leq 9$.
\end{thm}  
This improves the previous bound $\rho\leq 12$ obtained in \cite[Th.~1.2]{fibrations}, and it is sharp,
as the Fano model of $\Bl_{8 \pts}\pr^4$ is a smooth Fano $4$-fold with $\rho=9$ and a rational contraction onto $\Bl_{7 \pts}\pr^3$, see Ex.~\ref{Fanomodel}.

\medskip

\noindent{\bf Fano $4$-folds with $\rho=12$.}
As a straightforward consequence of Th.~\ref{addendum} and \ref{general} we improve the bound in Th.~\ref{inv}.
 \begin{corollary}
Let $X$ be a smooth Fano $4$-fold with $\rho_X=12$. Then $X\cong S_1\times S_2$ where $S_i$ are del Pezzo surfaces.
\end{corollary}

\noindent {\bf Fano $4$-folds with Lefschetz defect $2$.}
The Lefschetz defect of a Fano variety $X$ is an invariant defined as follows. Let $\N(X)$ be the real vector space of one-cycles in $X$, with real coefficients, up to numerical equivalence. For any prime divisor $\iota\colon D\hookrightarrow X$, the push-forward gives a linear map $\iota_*\colon\N(D)\to\N(X)$, and we consider the image
$$\N(D,X):=\iota_*(\N(D))\subseteq\N(X),$$
so that $\N(D,X)$ is the linear subspace of $\N(X)$ spanned by classes of curves in $D$. The \emph{Lefschetz defect} of $X$ is defined as
\stepcounter{thm}
\begin{equation}\label{defidelta}
  \delta_X:=\max\{\codim\N(D,X)\,|\,D\text{ a prime divisor in }X\}.
  \end{equation}
We refer the reader to \cite{rendiconti} for a survey on $\delta_X$, and recall the 
following result.
\begin{thm}[\cite{codim}, Th.~1.1]\label{deltageq4}
Let $X$ be a smooth Fano variety. If $\delta_X\geq 4$, then $X\cong S\times Z$, where $S$ is a del Pezzo surface with $\rho_S=\delta_X+1$.
\end{thm}
Therefore when $X$ is not a product, we have $\delta_X\leq 3$.
Fano varieties with $\delta_X=3$ do not need to be products, but  for them
 a structure theorem is given in \cite{delta3}, which in dimension $4$ yields the following.
\begin{thm}[\cite{delta3}, Prop.~1.5]\label{delta=3}
  Smooth Fano $4$-folds with $\delta_X=3$ are classified. 
  Either $X\cong S_1\times S_2$ where $S_i$ are del Pezzo surfaces, or $\rho_X\in\{5,6\}$. There are 6 families for $\rho_X=5$, and 11 for $\rho_X=6$.
\end{thm}
We believe that, to study Fano varieties with large $\rho_X$,  it is important to investigate the next case $\delta_X=2$, by looking for a special geometrical structure, as started in \cite{codimtwo}. 
In dimension $4$, as an application of our results, we prove the following.
\begin{thm}\label{delta2}
Let $X$ be a smooth Fano $4$-fold with $\delta_X=2$. Then $\rho_X\leq 6$.
  \end{thm}
  We note that this bound is sharp too, as it is achieved by $(\Bl_{2 \pts}\pr^2)^2$ (see Rem.~\ref{deltaproduct}). However we are not aware of other examples with $\rho=6$, while we provide several (known and new) examples of Fano $4$-folds $X$ with $\delta_X=2$ and $\rho_X\leq 5$, see \S\ref{ex_delta2} and \S\ref{other}.

  \medskip

\noindent{\bf New families of Fano $4$-folds with $\rho=7$.} 
We already recalled that all known examples of Fano $4$-folds that are not products of surfaces have $\rho\leq 9$. More precisely,
for $\rho=7,8,9$ there is only one known family, given by the Fano model of $\Bl_{r\,\pts}\pr^4$ for $r=6,7,8$, see Ex.~\ref{Fanomodel}. We point out that also for $\rho=6$ there are very few known families (again excluding products): $6$ toric \cite{bat2},  $2$ non toric with $\delta_X=3$ \cite[Prop.~7.1]{delta3}, and one in \cite{manivel4fold}.

We construct two new families with $\rho=7$,  which are the first examples of Fano $4$-folds with $\rho\geq 7$ different from the Fano model of $\Bl_{\pts}\pr^4$ and products of del Pezzo surfaces, as follows (see \S\ref{newa} and \S\ref{newb}).
\begin{proposition}\label{newexintro}
  Let $r\in\{0,\dotsc,4\}$ and let $W$ be the Fano model of $\Bl_{q_0,\dotsc,q_r}\pr^4$ (see Ex.~\ref{Fanomodel}), with $q_0,\dotsc,q_r\in\pr^4$ general points.
  Let $A\subset\pr^4$ be one of the following:
  \begin{enumerate}[$(i)$]
\item
  a general cubic rational normal scroll containing
  $q_0,\dotsc,q_r$;
  \item a general  sextic (singular) K3 surface with $\Sing(A)=\{q_0,\dotsc,q_r\}$,
having rational double points of type $A_1$ or $A_2$ in $q_i$ for every $i$, and contained in a smooth quadric hypersurface.
\end{enumerate}
Let $S\subset W$ be the transform of $A$, and $X\to W$ the blow-up of $S$.
Then $X$ is a smooth Fano $4$-fold
  with $\rho_X=r+3\in\{3,\dotsc,7\}$.  
\end{proposition}
We note that this gives also two new families with $\rho=6$, and besides these we construct an additional new family with $\rho=6$ (\S \ref{newel}). These new families all have a rational contraction onto a $3$-fold, and their
 construction has been suggested by
 our study of Fano $4$-folds with such a rational contraction, as we explain below.

There are a few cases in our constructions that we leave open, and could lead to more examples, see Questions \ref{Q1}, \ref{Q2}, and \ref{Q3}.
 
 \medskip

 \noindent{\bf Special rational contractions.}
 Let $X$ be a smooth Fano $4$-fold and $f_{\s X}\colon X\dasharrow Y$ a rational contraction with $\dim Y=3$. We say that $f_{\s X}$ is special if $Y$ is $\Q$-factorial and, when we factor $f_{\s X}$ as $X\stackrel{\xi}{\dasharrow} \w{X}\stackrel{f}{\to} Y$ where $\xi$ is a sequence of flips and $f$ a contraction, we have that  $f$ has  at most isolated $2$-dimensional fibers (and no $3$-dimensional fibers). Special rational contractions have especially good properties, see \cite[\S 3 and \S 6]{fibrations}.
 
 It follows from \cite{fibrations} that, when $X$ has a rational contraction onto a $3$-fold, it also has  a \emph{special} rational contraction onto a (possibly different) $3$-fold. In our setting we show the following.
 \begin{thm}\label{summary}
  Let $X$ be a smooth Fano $4$-fold with $\rho_X\geq 7$, not isomorphic to a product of surfaces, and having a rational contraction onto a $3$-dimensional variety.

  Then there exists a special rational contraction $X\dasharrow Y$, where $Y$ is a weak Fano\footnote{$Y$ is weak Fano if $-K_Y$ is nef and big.} $3$-fold with at most isolated, locally factorial, and canonical singularities, and $\rho_X-\rho_Y\in\{1,2\}$. 
\end{thm}
Therefore to study the geometry of $X$ as in Th.~\ref{summary} above we are reduced to study  separately the two cases where
$X$ has an elementary rational contraction onto a $3$-fold, and where 
there is a special rational contraction $X\dasharrow Y$  of relative Picard number two, with $\dim Y=3$.

\medskip

\noindent{\bf Classification results.}
In the case of relative Picard number two, in the range $\rho_X\geq 7$, we show that there are only three possibilities for $X$.
\begin{thm}\label{Bintro}
  Let $X$ be a smooth Fano $4$-fold with $\rho_X\geq 7$, not isomorphic to a product of surfaces, and having a special rational contraction 
$X\dasharrow Y$ with $\dim Y=3$ and $\rho_X-\rho_Y=2$.
Then $\rho_X\in\{7,8,9\}$ and
 $X$ is the blow-up of $W$ along a normal surface $S$,
where $W$ is the Fano model of $\Bl_{\pts}\pr^4$ (see Ex.~\ref{Fanomodel}), and $S\subset W$ is the transform of a surface $A\subset\pr^4$ containing the blown-up points, as in one of the following cases:
\begin{enumerate}[$(i)$]
\item $A$ is a cubic scroll;
\item $A$ is a sextic K3 surface, with rational double points of type $A_1$ or $A_2$ at the blown-up points, and $\rho_X=7$;
  \item $A$ is a cone over a twisted cubic.
  \end{enumerate}
  In cases $(i)$ and $(ii)$ the surface $S$ is smooth, while in $(iii)$ $S$ has one singular point, given by the vertex of the cone.
Moreover $Y$ is smooth, and up to flops $Y\cong\Bl_{\pts}\pr^3$.
\end{thm}
With Prop.~\ref{newexintro}
we show that cases $(i)$ and $(ii)$ do occur for $\rho_X\leq 7$, while we do not explore case $(iii)$.

\medskip

\noindent{\bf Techniques and strategy of proof.} Let $X$ be a  Fano $4$-fold and $f_{\s X}\colon X\dasharrow Y$ a rational contraction with $\dim Y=3$. First of all, by the results in \cite{fibrations}, we can assume that $f_{\s X}$ is special (see Prop.~\ref{factor}).

Recall that, when $\delta_X\geq 3$, either $X$ is a product of surfaces, or $\rho_X\in\{5,6\}$ and the possible $X$ are classified, by Th.~\ref{deltageq4} and \ref{delta=3}; therefore we can also suppose that
$\delta_X\leq 2$.

Let us consider a factorization of $f_{\s X}$ as $X\stackrel{\xi}{\dasharrow}\w{X}\stackrel{f}{\to}Y$; we can assume that $f$ is $K$-negative. Then $Y$ has at most isolated, locally factorial, and canonical singularities, and outside the (finitely many) $2$-dimensional fibers, $f$ is a conic bundle (see \ref{setting}). By studying the connected components of the discriminant divisor of $f$ in $Y$, 
we show that when $f$ is special and $\delta_X\leq 2$,
then either $\rho_X-\rho_Y\leq 2$, or $\delta_X=2$ (Lemma \ref{rhodelta}).

We first treat the case where $\rho_X-\rho_Y\leq 2$.
We assume that $\rho_Y\geq 5$ and follow the same approach as in \cite{eff} and \cite{fibrations}, that led to the bound $\rho_X\leq 12$. We show that $Y$ is weak Fano, and that (up to flops) there is a blow-up $Y\to Y_0$ of $r$ distinct smooth points, where $Y_0$ is weak Fano with $\rho_{Y_0}\leq 2$ (\ref{weakFano}, Lemma \ref{blowups}, proof of Th.~\ref{elem3fold}).

In order to bound $\rho_X$, we need to bound the number $r$ of blown-up points, and for this we use the anticanonical degree. Since $0<-K_Y^3=-K_{Y_0}^3-8r$, we give bounds on $-K_{Y_0}^3$ in our setting, using several results from the literature on singular Fano or weak Fano $3$-folds, see Section \ref{weakfano} and references therein.

In the elementary case (namely $\rho_X-\rho_Y=1$), this is enough to show that $\rho_X\leq 9$; if moreover $Y$ is smooth, then we show that (up to flops) there are only six possibilities
 for $Y_0$ and $Y$ (Cor.~\ref{tableelem}).

The case where $\rho_X-\rho_Y=2$ is studied in much more detail, as besides bounding the Picard number, we classify $X$ as stated in Th.~\ref{Bintro}. This
is done in  Section \ref{relrho2} which is the heart of the paper, and 
requires a long and articulated analysis that allows to prove first that $Y\cong\Bl_{\pts} \pr^3$, and then to identify $X$. We refer the reader to \ref{overview}
for an overview of the proof.

Finally let us consider
Fano $4$-folds $X$ with Lefschetz defect $\delta_X=2$. If $X$ is a product of surfaces, it is easy to see that $\rho_X\leq 6$. Otherwise, by \cite{codimtwo} either $X$ has a special rational contraction $X\dasharrow Y$ where $\dim Y=3$ and $\rho_X-\rho_Y=2$, or it has a flat contraction $X\to S$ onto a surface. In the second case we  get easily $\rho_X\leq 5$
by applying results from \cite{fanos}. In the first case we apply our previous results on the case of relative Picard number two; more precisely
we exclude that $\rho_X\geq 7$ thanks to  our classification (Thm.~\ref{Bintro}). 

\medskip

\noindent{\bf Outline of the paper.}
In Section \ref{prel} we set up the notation and give some preliminary results. In particular in \S\ref{2.2} we recall the properties of fixed prime divisors in Fano $4$-folds with $\rho\geq 7$, that will be crucial in the sequel.

Section \ref{secspecial} is about special rational contractions from smooth $4$-folds to $3$-folds. First we consider $K$-negative special contractions $f\colon X\to Y$ where $X$ is a smooth projective $4$-fold and $\dim Y=3$; then we introduce special rational contractions of Fano $4$-folds onto $3$-folds, and their properties.

Section \ref{weakfano} is an auxiliary section, where we present some results on Fano and weak Fano $3$-folds $Y$ with locally factorial and canonical singularities, in particular with respect to the 
anticanonical degree $-K_Y^3$; these will be needed in the rest of the paper.

In Section \ref{relrho2} we treat the case of relative Picard number $2$, and prove Th.~\ref{Bintro}.

Then in Section \ref{grenoble} we treat the elementary case and the case of Lefschetz defect $2$, proving Theorems \ref{general}, \ref{delta2}, and \ref{summary}.

Finally Section \ref{examples} is devoted to the construction of new families and examples, we prove Prop.~\ref{newexintro} and ask some open questions on possible further new examples.

{\small\tableofcontents}

   \section{Preliminaries}\label{prel}
   \subsection{Notations}
   \noindent We work over the field of complex numbers.
   We refer the reader to \cite{kollarmori} for the terminology and standard results in birational geometry, and to \cite{hukeel} for Mori dream spaces.

Let $X$ be a normal and $\Q$-factorial quasi-projective variety.
 A \emph{contraction} is a surjective, projective map $f\colon X\to Y$, with connected fibers, where $Y$ is normal and quasi-projective.
 We denote by $\N(X/Y)$ the vector space of one cycles in $X$, with coefficients in $\R$, contracted to points by $f$, up to numerical equivalence.
We say that $f$ is \emph{elementary} if $\dim\N(X/Y)=1$, and $f$ is \emph{$K$-negative} is $-K_X$ is $f$-ample.

Let $X$ be projective. As usual we denote by $\Nu(X)$ (respectively $\N(X)$)
the vector space of $\R$-divisors in $X$ (respectively  one cycles in $X$ with coefficients in $\R$) up to numerical equivalence. We denote by $[D]\in\Nu(X)$ the class of a divisor, and $\equiv$ stands for numerical equivalence.

An elementary contraction $f$ is \emph{of type $(a,b)$} if $\dim\Exc(f)=a$ and $\dim f(\Exc(f))=b$.
Additionally, when $\dim X=4$, we say that $f$ is of type
$(3,b)^{\sm}$, with $b\in\{0,1,2\}$, if $Y$ is smooth and $f$ is the blow-up of a smooth, irreducible subvariety of dimension $b$. Still when $\dim X=4$, we say that $f$ is of type $(3,0)^Q$ if $\Exc(f)\cong Q$ where $Q$ is a $3$-dimensional quadric, $f(Q)=\{pt\}$, and $\ma{N}_{\Exc(f)/Q}\cong \ol_Q(-1)$; then $Q$ is either smooth or the cone over a smooth quadric surface, see \cite[Lemma 2.19]{blowup}.

We denote by $\NE(f)$ the face of $\NE(X)$ generated by classes of curves contracted to points by $f$.

\medskip

Suppose that $X$ is a Mori dream space. An
extremal ray $R$ is  one dimensional face of $\NE(X)$.
Let $f\colon X\to Y$ be the associated elementary contraction (namely $R=\NE(f)$); we set $\Lo(R):=\Exc(f)$.
For a divisor  $D$ on $X$ we write
$D\cdot R>0, =0,<0$ if $D\cdot \gamma>0,=0,<0$ for $\gamma\in R$ non-zero. We also say that $R$ or $f$ are $D$-positive, $D$-trivial, $D$-negative respectively.

By a flip we mean the flip of a small extremal ray $R$, or equivalently of a small elementary contraction, in the sense of \cite[Def.~3.33]{kollarmori}. We say that the flip is
$D$-negative if $R$ is. In dimension $3$, as customary a flop is the flip of a $K$-trivial extremal ray.

We set $D^{\perp}:=\{\gamma\in\N(X)|D\cdot\gamma=0\}$.
In a real vector space $N$ of finite dimension, we denote by
$\langle \gamma_1,\dotsc,\gamma_r\rangle$ the convex cone spanned by $\gamma_1,\dotsc,\gamma_r$.

In $\Nu(X)$ we denote by $\Mov(X)$ (respectively $\Eff(X)$) the convex cone spanned by classes of movable (respectively effective) divisors. Since $X$ is a Mori dream space, both cones are polyhedral.

A \emph{fixed prime divisor} is a prime divisor $D$ such that $\Bs|mD|=D$ for every $m\in\Z_{>0}$.

 A \emph{small $\Q$-factorial modification (SQM)} of $X$ is a birational map $X\dasharrow X'$ that factors as a finite sequence of flips.
A \emph{rational contraction} is a rational map $f\colon X\dasharrow Y$ that factors as $X\stackrel{\ph}{\dasharrow} X'\stackrel{f'}{\to} Y$, where $\ph$ is a SQM and $f'$ is a contraction; $f$ is \emph{elementary} if $\rho_X-\rho_Y=1$.
An elementary rational contraction can be divisorial, small, or of fiber type, depending on the corresponding property of $f'$.

\medskip

Let $X$ be a normal and $\Q$-factorial projective variety. We say that $X$ is \emph{log Fano} if there exists an effective $\Q$-divisor $\Delta$ such that $(X,\Delta)$ has klt singularities and $-(K_X+\Delta)$ is ample; in particular $-K_X$ is big. If $X$ is log Fano, then $X$ is a Mori dream space, by \cite[Cor.~1.3.2]{BCHM}; moreover if $X\dasharrow Y$ is a rational contraction with $Y$ $\Q$-factorial, then $Y$ is still log Fano, by \cite[Lemma 2.8]{prokshok}.

   Let $X$ be a smooth projective $4$-fold. An \emph{exceptional plane} is a surface $L\subset X$ such that $L\cong\pr^2$ and $\ma{N}_{L/X}\cong\ol_{\pr^2}(-1)^{\oplus 2}$; we denote by $C_L\subset L$ a line in $L$, note that $-K_X\cdot C_L=1$.  An \emph{exceptional line} is a curve $\ell\subset X$ such that $\ell\cong\pr^1$ and $\ma{N}_{\ell/X}\cong\ol_{\pr^1}(-1)^{\oplus 3}$; note that $K_{X}\cdot\ell=1$.

A \emph{node} is an ordinary double point.

We denote by $\pr^2\bullet\pr^2$ the union of two planes in $\pr^4$ intersecting in one point.

If $\iota\colon Z\hookrightarrow X$ is a closed subset, we set $\N(Z,X):=\iota_*(\N(Z))\subseteq\N(X)$.

\medskip

We give a few preliminary results that are needed in the sequel; recall the definition of Lefschetz defect \eqref{defidelta}.
\begin{remark}[\cite{rendiconti}, Lemma 5]\label{deltaproduct}
Let $X=Y\times Z$ where $Y$ and $Z$ are smooth Fano varieties; then $\delta_X=\max\{\delta_Y,\delta_Z\}$. Moreover it follows from the definition of Lefschetz defect that a del Pezzo surface $S$ has $\delta_S=\rho_S-1$.
\end{remark}  
\begin{remark}\label{disjoint}
  Let $X$ be a normal and $\Q$-factorial projective variety, $Z\subset X$ a closed subset, and $D\subset X$ a prime divisor such that $Z\cap D=\emptyset$.
  Then $\N(Z,X)\subseteq D^{\perp}\subset\N(X)$.
  Indeed we have $D\cdot C=0$ for every curve $C\subset Z$.
\end{remark}
\begin{remark}\label{linalg}
  Let $X$ be a normal and $\Q$-factorial projective variety, $Z\subset X$ a closed subset, and $f\colon X\to Y$ a contraction.
  Then $\dim\N(Z,X)\leq\dim\N(f(Z),Y)+\rho_X-\rho_Y$.

  Indeed consider the pushforward $f_*\colon \N(X)\to\N(Y)$; we have $f_*(\N(Z,X))=\N(f(Z),Y)$ and $\dim\ker f_*=\rho_X-\rho_Y$.
\end{remark}  
\begin{lemma}\label{tautological}
  Let $\alpha\colon X\to W$ be the blow-up of a smooth projective $4$-fold along a smooth irreducible surface $S$, 
with exceptional divisor $E\subset X$.

  Then $\ol_X(-K_X)_{|E}$ is ample on $E$ if and only if the vector bundle $$\ma{N}_{S/W}^{\vee}\otimes\ol_W(-K_W)_{|S}\cong\ma{N}_{S/W}\otimes\ol_S(-K_S)$$
  is ample on $S$.
  Moreover we have the following:
  \begin{gather*}
  K_X^4
= K_W^4
- 3(K_{W|S})^2 - 2K_S \cdot K_{W|S} + c_2(\mathcal{N}_{S/W}) - K_S^2,\\
K_X^2\cdot c_2(X) = K_W^2\cdot c_2(W) - 12\chi(\ol_S) + 2K_S^2
- 2K_S \cdot K_{W|S}- 2c_2(\mathcal{N}_{W/Y}),\\
\chi(X,-K_X) = \chi(W,-K_W ) -\chi(\ol_S) -\frac{1}{2}\bigl(
(K_{W|S})^2 + K_S \cdot K_{W|S}\bigr).
\end{gather*}
\end{lemma}
\begin{proof}
  We have $E\cong\pr_S(\ma{N}_{S/W}^{\vee})$ with tautological class $\eta\cong\ol_X(-E)_{|E}$. Moreover
  $$\ol_E(K_E)=\alpha_{|E}^*\bigl(\ol_S(K_S)\otimes\det\ma{N}_{S/W}^{\vee}\bigr)\otimes
 \eta^{\otimes(-2)},\quad \ol_W(-K_W)_{|S}\cong \ol_S(-K_S)\otimes\det\ma{N}_{S/W},$$
and $K_E=(K_X+E)_{|E} $.
 We get that
 $\ol_X(-K_X)_{|E}=\eta\otimes\alpha_{|E}^*(\ol_W(-K_W)_{|S})$ is the tautological class for $E=\pr_S(\ma{N}_{S/W}^{\vee}\otimes\ol_W(-K_W)_{|S})$.

 Finally we recall that for a vector two vector bundle $\ma{E}$ we have $\ma{E}\cong\ma{E}^{\vee}\otimes\det(\ma{E})$, and the formulae are from \cite[Lemma 3.2]{delta3_4folds}.
\end{proof}
\begin{remark}\label{sections}
  Let $Y$ be a smooth quasi-projective variety, $X:=\pr_Y(\ma{E})$ where $\ma{E}$ is a rank two vector bundle on $Y$, and $\pi\colon X\to Y$ the $\pr^1$-bundle. Set $L:=\det\ma{E}\in\Pic(Y)$.

  Let $C\subset Y$ be a smooth projective rational curve and set $S:=\pi^{-1}(C)\cong\mathbb{F}_e$, where $\mathbb{F}_e$ is the Hirzebruch surface and $e\geq 0$. Let $\Gamma^-$ and $\Gamma^+$ be respectively the negative section and a positive section of $\pi_{|S}\colon S\to C$ (namely $(\Gamma^-)^2=-e$ and $(\Gamma^+)^2=e$ in $S$). Then $L\cdot C\equiv e\mod 2$ and
  $$-K_X\cdot\Gamma^-=-K_Y\cdot C-e,\quad\quad -K_X\cdot\Gamma^+=-K_Y\cdot C+e.$$
\end{remark}
\begin{proof}
We have
 $-K_{X}=\pi^*(-K_{Y})-K_{\pi}$ and  $$-K_{X|S}=\pi^*(-K_{Y})_{|S}-K_{\pi|S}=
 \pi^*(-K_{Y})_{|S}-    K_S+    (\pi_{|S})^*K_{\pr^1}.$$
In particular if $\Gamma$ is $\Gamma^+$ or $\Gamma^-$
 we get
 $-K_{X}\cdot\Gamma=-K_{Y}\cdot C-K_S\cdot\Gamma-2$, which yields the formulae above since
 $-K_S\cdot\Gamma^-=2-e$ and $-K_S\cdot\Gamma^+=2+e$.

 To see that $L\cdot C\equiv e\mod 2$, write $\ma{E}_{|C}\cong\ol_{\pr^1}(a)\oplus\ol_{\pr^1}(b)$ with $a\leq b$. Then $L\cdot C=a+b$ and $e=b-a$.
\end{proof}
Finally recall some results on $K$-negative contractions.
\begin{thm}[\cite{AW}, Theorem on p.~256]\label{32}
Let $X$ be a smooth quasi-projective $4$-fold and $f\colon X\to Y$ a $K$-negative divisorial elementary contraction of type $(3,2)$. Then $f$ can have at most finitely many $2$-dimensional fibers over $y_1,\dotsc,y_r$, $\Sing(Y)\subseteq \{y_1,\dotsc,y_r\}$, and $Y$ is locally factorial and has at most nodes as singularities. Moreover over $Y\smallsetminus\{y_1,\dotsc,y_r\}$ $f$ is just the blow-up of a smooth, irreducible surface.
\end{thm}
\begin{thm}[\cite{gloria}, Th.~2.2]\label{gdn}
Let $X$ be a normal and locally factorial projective variety with canonical singularities, and with at most finitely many non-terminal points. Let  $f\colon X\to Y$ be a $K$-negative birational elementary contraction, with fibers of dimension $\leq 1$. Then $f$ is divisorial.
\end{thm}  
\begin{thm}[\cite{kawsmall}]\label{kawsm}
  Let $X$ be a smooth, quasi-projective $4$-fold, and  $f\colon X\to Y$ a $K$-negative small elementary contraction. Then $\Exc(f)$ is a finite, disjoint union of exceptional planes.
\end{thm}
\begin{thm}[\cite{kachiflips}, Th.~1.1, Cor.~2.2 and references therein]\label{kachiflip}
  Let $X$ be a projective $4$-fold with at most locally factorial, terminal, isolated l.c.i.\ singularities. Let $f\colon X\to Y$ a $K$-negative small elementary contraction. Then for every irreducible component $L$ of $\Exc(f)$ we have $(L, -K_{X|L}) \cong(\pr^2, \ol_{\pr^2}(1))$.
\end{thm}
\begin{proposition}[\cite{codimtwo}, Prop.~3.7(2)]\label{degreeone}
Let $X$ be a smooth Fano $4$-fold and $\ph\colon X\dasharrow X'$ a birational, rational contraction such that $X'$ is $\Q$-factorial. Set $P:= X'\smallsetminus\dom(\ph^{-1})$. Then for every irreducible curve $C\subset X'$ with $-K_{X'}\cdot C=1$ we have either $C\subset P$ or $C\cap P=\emptyset$.
\end{proposition}  
A \emph{conic bundle} is a projective morphism $f\colon X\to Y$ where $X$ is a quasi-projective variety with Gorenstein, log terminal singularities, $Y$ is smooth, and there exists a rank $3$ vector bundle $\ma{E}$ on $Y$ such that $X\subset\pr_Y(\ma{E})$, $f$ is the restriction of the projection $\pi\colon\pr_Y(\ma{E})\to Y$, and every fiber of $f$ is a plane conic in the corresponding fiber of $\pi$; see for instance \cite[Ch.~I]{beauville}.
\begin{thm}[\cite{AW}, Prop.~4.1 and references therein]\label{conic}
Let $X$ be a smooth quasi-projective variety and $f\colon X\to Y$ a $K$-negative contraction such that every fiber of $f$ has dimension one. Then $Y$ is smooth and $f$ is a conic bundle.
\end{thm}
\begin{proposition}[\cite{eleonoraflatness}, Prop.~1.3]\label{conicsing}
  Let $X$ be a Gorenstein quasi-projective variety with log terminal singularities and $f\colon X\to Y$ a $K$-negative contraction with $Y$ smooth, such that every fiber of $f$ has dimension one. Then
 $f$ is a conic bundle.
\end{proposition}
\subsection{Fano $4$-folds and fixed divisors}\label{2.2}
\noindent  In this section we recall some results on the birational geometry of Fano $4$-folds. First of all, we describe in Lemma \ref{SQMFano} the structure of SQM's of a Fano $4$-fold. Then we recall a classification result for fixed prime divisors in Fano $4$-folds with Picard number $\geq 7$, there are only four possible types, and we describe the associated divisorial elementary rational contractions (Th.-Def.~\ref{fixed}). Finally we describe the possible relative positions of some pairs of fixed prime divisors that appear as exceptional divisors for the same birational map (Lemmas \ref{2dimfaces} - \ref{casesiiandiii}).
\begin{lemma}[\cite{eff}, Rem.~3.6]\label{SQMFano}
  Let  $X$ be a smooth Fano $4$-fold and $\ph\colon X\dasharrow \w{X}$ a SQM.
  We have the following:
  \begin{enumerate}[$(a)$]
 \item $\w{X}$ is smooth, $X\smallsetminus\dom(\ph)=L_1\cup\cdots\cup L_r$ where $L_i$ are pairwise disjoint exceptional planes,
and
$\w{X}\smallsetminus\dom(\ph^{-1})=\ell_1\cup\cdots\cup\ell_r$
where $\ell_i$ are pairwise disjoint exceptional lines; moreover $\ph$ factors as
$$\xymatrix{  X \ar@/^1pc/@{-->}[rr]^{\ph}&{\wi{X}}\ar[l]^f\ar[r]_g&{\w{X}}
}$$
where $f$ is the blow-up of $L_1\cup\cdots\cup L_r$ and $g$ is the blow-up of $\ell_1\cup\cdots\cup\ell_r$.
\item
  Let  $C\subset \w{X}$ be an irreducible curve, different from $\ell_1,\dotsc,\ell_r$, and intersecting $\ell_1\cup\cdots\cup\ell_r$ in $s\geq 0$ points; then $-K_{\w{X}}\cdot C\geq 1+s$.
  \item If $-K_{\w{X}}\cdot C= 1$, then  $C\cap(\ell_1\cup\cdots\cup\ell_r)=\emptyset$.
\end{enumerate}  
\end{lemma}
\begin{lemma}[\cite{eff}, Rem.~3.7]\label{Kneg}
  Let  $X$ be a smooth Fano $4$-fold and $f\colon X\dasharrow Y$ a rational contraction. Then there exists a SQM $\xi\colon X\dasharrow \w{X}$ such that the composition $f\circ\xi^{-1}\colon\w{X}\to Y$ is regular and $K$-negative.
\end{lemma}
\begin{lemma}\label{dim32}
Let $X$ be a smooth Fano $4$-fold with $\rho_X\geq 6$, and $D\subset X$  the exceptional divisor of a divisorial elementary contraction of type $(3,2)$. Let $X\dasharrow \w{X}$ be a SQM, and $\w{D}\subset\w{X}$ the transform of $D$. Then $D$ does not contain exceptional planes, and $\dim\N(D,X)=\dim\N(\w{D},\w{X})$.
\end{lemma}
\begin{proof}
By \cite[Rem.~2.17(2)]{blowup} $D$ does not contain exceptional planes, thus the statement follows from \cite[Cor.~3.14]{eff}.
\end{proof}  
   \begin{thmdef}[\cite{blowup}, Th.~5.1, Cor.~5.2, Def.~5.3, Lemma 5.25, Def.~5.27]\label{fixed}
     Let $X$ be a smooth Fano $4$-fold with  $\rho_X\geq 7$, or $\rho_X=6$ and $\delta_X\leq 2$, and $D$ a fixed prime divisor in $X$.
      \begin{enumerate}[$(a)$]
      \item There exists a unique diagram:
        $$X\stackrel{\xi}{\dasharrow}\w{X}\stackrel{\sigma}{\la}Y$$
        where $\xi$ is a SQM, $\sigma$ is a divisorial  elementary contraction with exceptional divisor the transform $\w{D}$ of $D$, and $Y$ is Fano (possibly singular);
      \item $\sigma$ is of type $(3,0)^{\sm}$, $(3,0)^Q$, $(3,1)^{\sm}$, or $(3,2)$, and we define $D$ to be {\bf of type  $(3,0)^{\sm}$, $(3,0)^Q$, $(3,1)^{\sm}$, or $(3,2)$,} accordingly;
        \item if $D$ is of type $(3,2)$, then $X=\w{X}$. In the other cases $\xi$ factors as a sequence of $D$-negative and $K$-negative flips.
        \item We define $C_D\subset D\subset X$ to be the transform of a general irreducible curve $C_{\w{D}}\subset \w{D}\subset \w{X}$ contracted by $\sigma$, of minimal anticanonical degree. Then $C_D\cong\pr^1$, $D\cdot C_D=-1$, and $C_D\subset\dom(\xi)$.
        \item Given a SQM $X\dasharrow X'$ and a divisorial elementary  contraction $\sigma'\colon X'\to Y'$ with $\Exc(\sigma')$ the transform of $D$, there is a commutative diagram:          $$\xymatrix{X\ar@{-->}[r]^{\xi}&{\w{X}}\ar@{-->}[r]^{\psi_{\s X}}\ar[d]_{\sigma}&{X'}\ar[d]^{\sigma'}\\
            & Y\ar@{-->}[r]^{\psi_{\s Y}}&{Y'}
          }$$
          where $\psi_{\s X}$ and $\psi_{\s Y}$ are SQM's, $\w{D}\subset\dom(\psi_{\s X})$,
          and $\sigma(\w{D})\subset\dom(\psi_{\s Y})$.
      \end{enumerate}  
    \end{thmdef}

    Let $X$ be a smooth Fano $4$-fold, $\ph\colon X\dasharrow\wi{X}$ a SQM, and $E\subset \wi{X}$ a fixed prime divisor. We define {\bf the type of $E$} to be the type
of its transform $E_{\s X}\subset X$, and we define $C_E\subset E\subset \wi{X}$ to be the transform of $C_{E_{\s X}}\subset X$. Note that, since $\dim(\wi{X}\smallsetminus\dom(\ph^{-1}))=1$ (see Lemma \ref{SQMFano}$(a)$), we have $C_E\subset\dom(\ph^{-1})$.

We say that two fixed prime divisors $D,E$ are {\bf adjacent}\label{padj} if $[D],[E]\in\Nu(X)$ generate a two dimensional face of $\Eff(X)$, and moreover 
$\langle [D],[E]\rangle\cap\Mov(X)=\{0\}$. 
   \begin{lemma}\label{2dimfaces}
     Let $X$ be a smooth Fano $4$-fold with $\rho_X\geq 7$ and let $D,E$ be adjacent fixed prime divisors, $E$ of type $(3,2)$, and $D$ of type $(3,1)^{\sm}$ or $(3,0)^Q$, such that $D\cap E\neq\emptyset$. Then $D\cdot C_E=0$ and one of the following holds:
     \begin{enumerate}[$(i)$]
     \item $D$ is of type $(3,1)^{\sm}$, $E\cdot C_D=1$, and $E\cap L=\emptyset$ for every exceptional plane $L\subset D$;
     \item $D$ is of type $(3,1)^{\sm}$, $E\cdot C_D=0$, there exists an exceptional plane $L_0\subset D$ such that $D\cdot C_{L_0}=-1$, $E\cdot C_{L_0}=1$, $C_D\equiv C_E+C_{L_0}$, and
       $E\cap L=\emptyset$ for every exceptional plane $L\subset D$ with $C_L\not\equiv C_{L_0}$;
       \item $D$ is of type $(3,0)^Q$, $E\cdot C_D=1$, there exists an exceptional plane $L_0\subset D$ such that $D\cdot C_{L_0}=-1$, $E\cdot C_{L_0}=2$, $C_D\equiv C_E+C_{L_0}$, and
       $E\cap L=\emptyset$ for every exceptional plane $L\subset D$ with $C_L\not\equiv C_{L_0}$.
     \end{enumerate}
   \end{lemma}
   \begin{proof}
     Recall that $E$ does not contain exceptional planes (Lemma \ref{dim32}); then
     $D\cdot C_E=0$ by \cite[Lemma 4.9]{fibrations}.
If $D$ is of type $(3,1)^{\sm}$ and $E\cdot C_D>0$, then we have $(i)$ by \cite[Lemma 4.23]{small}.

Suppose that, if $D$ is of type $(3,1)^{\sm}$, we have  $E\cdot C_D=0$. Then we apply \cite[Lemma 6.9, Prop.~6.1, Prop.~6.4, and Cor.~6.10]{small}, and get $(ii)$ or $(iii)$.
   \end{proof}  
   Given two adjacent fixed prime divisors $D$ and $E$ in $X$, up to a SQM we can contract both of them with divisorial elementary contractions. However if $D$ and $E$ intersect, in general the SQM and the  type of divisorial elementary contractions may depend on the order with which we contract the two divisors. In the next lemmas we describe this situation for the cases given by Lemma \ref{2dimfaces}; this will be used in Section \ref{relrho2}. The vertical arrows (labeled by $\alpha$) are divisorial elementary contractions with exceptional divisor $E$ or its transforms, while the horizontal arrows (labeled by $\sigma$) are divisorial elementary contractions with exceptional divisor $D$ or its transforms.
   \begin{lemma}[\cite{small}, Lemma 4.23 and its proof]\label{casei}
 Let $X$ be a smooth Fano $4$-fold with $\rho_X\geq 7$ and let $D,E$ be adjacent fixed prime divisors as in Lemma \ref{2dimfaces}$(i)$.
  Then we have a diagram $$\xymatrix{X\ar@{-->}[r]^{\xi}\ar[d]_{\alpha}&{\w{X}}\ar[d]^{\tilde\alpha}
       \ar[r]^{\sigma}&{X_0}\ar[d]^{\alpha_0}\\
       W\ar@{-->}[r]^{\xi_{\scriptscriptstyle W}}&{\w{W}}\ar[r]^{\sigma_{\scriptscriptstyle W}}&{W_0}
     }$$
     where:
     \begin{enumerate}[$\bullet$]
        \item  
        $\alpha$, $\tilde\alpha$, and $\alpha_0$ are divisorial elementary contractions of type $(3,2)$ with exceptional divisor $E$ or its transforms;
     \item $\xi$ and $\sigma$ are as in Th.-Def.~\ref{fixed} $(a)$ for $D$,  $\sigma$ blows-up a fiber $C\cong\pr^1$ of $\alpha_0$, and $\Exc(\sigma)\cong\pr_{\pr^1}(\ol\oplus\ol(1))\cong\Bl_{\text{\em line}}\pr^3$;
     \item  $\Exc(\sigma_{\scriptscriptstyle W})\subset\w{W}_{\reg}$, $\Exc(\sigma_{\scriptscriptstyle W})=\tilde\alpha(\Exc(\sigma))\cong\pr^3$, and $\sigma_{\scriptscriptstyle W}$ is the blow-up of the smooth point $\alpha_0(C)\in W_0$;
         \item $\xi$ and $\xi_{\scriptscriptstyle W}$ are SQM's.
          \end{enumerate}   
   \end{lemma}
  \begin{lemma}\label{casesiiandiii}
    Let $X$ be a smooth Fano $4$-fold with $\rho_X\geq 7$ and let $D,E$ be adjacent fixed prime divisors as in Lemma \ref{2dimfaces}$(ii)$ or $(iii)$. Then
we have a diagram $$\xymatrix{X\ar@{-->}@/^1pc/[rr]^{\xi}\ar@{-->}[r]_{\xi_1}\ar[d]_{\alpha}&{\w{X}}\ar@{-->}[r]_{\xi_2}\ar[d]^{\tilde\alpha}&\ar@{-->}[r]{\wi{X}}
       \ar[r]^{\sigma}&{X_0}\ar[d]^{\alpha_0}\\
       W\ar@{-->}[r]^{\xi_{\scriptscriptstyle W}}&{\w{W}}\ar[rr]^{\sigma_{\scriptscriptstyle W}}&&{W_0}
     }$$
     where:
      \begin{enumerate}[$\bullet$]
        \item 
          $\alpha$, $\tilde\alpha$, and $\alpha_0$ are divisorial elementary contractions of type $(3,2)$ with exceptional divisor $E$ or its transforms;
        \item $\xi$ and $\sigma$ are as in Th.-Def.~\ref{fixed} $(a)$ for $D$;
        \item  $\xi_1$ is a sequence of $D$-negative and $K$-negative flips, and $E\subset\dom(\xi_1)$;
          \item 
            $\xi_2$ is the flip of the small extremal ray generated by $[C_{L_0}]$ (see  Lemma \ref{2dimfaces});           
        \item $\sigma_{\scriptscriptstyle W}$ is a divisorial elementary contraction with $\Exc(\sigma_{\scriptscriptstyle W})$ the transform of $D$,  and $\Exc(\sigma_{\scriptscriptstyle W})\subset\w{W}_{\reg}$;
      \item  $\xi_{\scriptscriptstyle W}$ is a SQM;
        \item if  $D$ is of type $(3,1)^{\sm}$ and $\sigma$ blows-up a curve $\Gamma\subset X_0$, then $\Gamma\cdot\Exc(\alpha_0)>0$, and
          $\sigma_{\scriptscriptstyle W}$ is of type $(3,1)^{\sm}$ and blows-up the curve $\alpha_0(\Gamma)\subset W_0$;
          \item 
            if $\sigma$ is of type $(3,0)^Q$, then 
            $\sigma_{\scriptscriptstyle W}$ is of type $(3,0)^{\sm}$.
          \end{enumerate}
        \end{lemma}
        \begin{proof}
          We order the sequence of $D$-negative flips in $\xi$ by performing first all $E$-trivial flips; by Lemma \ref{2dimfaces} the loci of these flips are exceptional planes disjoint from $E$, so that $E\subset X$ and its transform $\w{E}\subset \w{X}$ are contained in the open subsets where $\xi_1$ is an isomorphism.  
               Then $\NE(\w{X})$ must have a unique $\w{E}$-negative extremal ray, which gives the contraction $\tilde\alpha\colon \w{X}\to \w{W}$ of type $(3,2)$.
          
      The transform $\w{D}\subset \w{X}$ of $D$ contains exceptional planes $L^1_0:=L_0,L_0^2,\dotsc,L_0^d$ such that $C_{L_0^i}\equiv C_{L_0}$ for every $i=1,\dotsc,d$, and they are the indeterminacy locus of the last flip $\xi_2$. Moreover $(\xi_2)_{|\w{D}}$ is regular, and $(\xi_2)_{|\w{D}}\colon \w{D}\to \wi{D}$ is the blow-up of $d$ smooth points, with exceptional divisors $L_0^1,\dotsc,L_0^d$. 

          If $D$ is of type $(3,1)^{\sm}$, then $\wi{D}$ is a $\pr^1$-bundle over a curve $\Gamma$, and the composite map $\w{D}\to \Gamma$ has $d$ singular fibers given by $L_0^i\cup S_i$, where $S_i\cong\mathbb{F}_1$ is the transform of the corresponding fiber of $\wi{D}\to\Gamma$. Moreover
          $\w{E}\cap \w{D}=S_1\cup\cdots\cup S_d$
       and  $(\tilde\alpha)_{|\w{D}}\colon \w{D}\to \tilde\alpha(\w{D})$ contracts $S_i$ to curves, and $\tilde\alpha(\w{D})$ is another $\pr^2$-bundle over $\Gamma$. There are exceptional lines $\ell_0^i\subset\wi{X}$ corresponding to $L_0^i$ (see Lemma \ref{SQMFano}$(a)$), and if $\wi{E}\subset\wi{X}$ is the transform of $E$, we have $\wi{E}\cdot\ell_0^i=-1$. In $X_0$, the images $\sigma(\ell_0^i)$ for $i=1,\dotsc,d$ are the fibers of $\alpha_0$ that meet $\Gamma$.

       If $D$ is of type $(3,0)^Q$, then $d=1$ by \cite[Prop.~6.4 and its proof]{small}, and
       $\w{D}\cong\Bl_{\pt}\wi{D}\cong\Bl_{\text{\em conic}}\pr^3$.
If $\wi{D}$ is a smooth quadric, then 
the conic is smooth, and $\w{E}\cap \w{D}\cong\mathbb{F}_2$. 
If $\wi{D}$ is the cone over $\pr^1\times\pr^1$, then the conic is reducible, and   $\w{E}\cap \w{D}$ has two irreducible components, both isomorphic to $\mathbb{F}_1$. Finally $\alpha_0(\w{D})\cong\pr^3$.
        \end{proof}  
        \section{Special rational contractions}\label{secspecial}
        \noindent  In this section we introduce special contractions and rational contractions. First in \S \ref{secspecial1} we consider $K$-negative special contractions $f\colon X\to Y$ where $X$ is a smooth projective $4$-fold and $\dim Y=3$. We recall some results from \cite{fibrations}, in particular concerning the discriminant divisor $\Delta\subset Y$, and define the \emph{intrinsic discriminant} $\Delta_{\intr}$ of $f$ as the union of the irreducible components $\Delta_0$ of $\Delta$ such that $f^*(\Delta_0)$ is irreducible.
        
        Using the classification of the possible $2$-dimensional fibers of $f$ by Andreatta-Wi\'sniewski and Kachi, we show that if $y_0\in Y$ is a singular point, then either $y_0\in\Delta_{\intr}$, or $y_0$ is a node and $f^{-1}(y_0)$ is the union of two copies of $\pr^2$ meeting transversally at one point, and the lines in the two $\pr^2$'s are numerically equivalent in $X$ (Th.~\ref{2dimfibers}). We will use this in Section \ref{relrho2} to relate (in our setting) the presence of these special fibers over nodes to the non-rationality of $Y$ (Lemma \ref{sm}).

        Then in \S \ref{secspecial2} we turn to special \emph{rational} contractions, and present the results needed to show that if a Fano $4$-fold $X$ has a rational contraction onto a $3$-fold and $\delta_X\leq 1$, then there is also a special rational contraction $X\dasharrow Y$ with $\dim Y=3$ and $\rho_X-\rho_Y\in\{1,2\}$ (Prop.~\ref{factor} and Lemma \ref{rhodelta}).
        \subsection{Special, $K$-negative contractions from a $4$-fold to a $3$-fold}\label{secspecial1}
        \begin{definition}\label{defspecial}
          Let $X$ be a normal and $\Q$-factorial projective variety, and a Mori dream space.
  A contraction of fiber type $f\colon X\to Y$ is special if $Y$ is $\Q$-factorial and, for every prime divisor $D\subset X$, either $f(D)=Y$, or $f(D)$ is a prime divisor in $Y$.

  Equivalently, when $\dim X=4$ and $\dim Y=3$, a contraction $f\colon X\to Y$ is {\bf special} if $Y$ is $\Q$-factorial and $f$ has at most isolated $2$-dimensional fibers and no $3$-dimensional fiber.
\end{definition}
Special contractions of Mori dream spaces were introduced and studied in \cite{fibrations}, to which we refer the interested reader for more details. 
\begin{thm}[\cite{fibrations}]\label{special}
  Let $X$ be a smooth projective $4$-fold and a Mori dream space, and let $f\colon X\to Y$ be a special, $K$-negative contraction with $\dim Y=3$. Set $m:=\rho_X-\rho_Y-1$. We have the following:
  \begin{enumerate}[$(a)$]
  \item $Y$ can have at most isolated, locally factorial, canonical singularities, contained in the images of the $2$-dimensional fibers of $f$;
  \item there are $m$ prime divisors $B_1,\dotsc,B_m\subset Y$ such that $f^*B_i$ is reducible for every $i=1,\dotsc,m$, and $f^*D$ is irreducible for every prime divisor $D$ different from $B_1,\dotsc,B_m$;
  \item   $B_1,\dotsc,B_m$ are pairwise disjoint;
  \item  $f^*B_i$ has two irreducible components $E_i$ and $\wi{E}_i$.
    The general fiber of $f$ over $B_i$ is $e_i+\hat{e}_i$, where $e_i,\hat{e}_i$ are integral curves with $E_i\cdot e_i<0$, $\wi{E}_i\cdot\hat{e}_i<0$, and $-K_X\cdot e_i=-K_X\cdot\hat{e}_i=1$, for every $i=1,\dotsc,m$.
    \end{enumerate}
  \end{thm}
  \begin{proof}
Statement $(a)$ follows from \cite[Prop.~2.20]{fibrations} and Th.~\ref{conic}, and the other statements from
\cite[Lemmas 3.4 and 3.5]{fibrations}.
\end{proof}
\begin{prg}\label{setting}
  Let $X$ be a smooth projective $4$-fold and a Mori dream space, and $f\colon X\to Y$ a special, $K$-negative contraction with $\dim Y=3$, as in Th.~\ref{special}.
  
If $F_1,\dotsc,F_r\subset X$ are the $2$-dimensional fibers of $f$, and $X_0:=X\smallsetminus(F_1\cup\cdots\cup F_r)$, then
$f_{|X_0}\colon X_0\to Y\smallsetminus(f(F_1)\cup\cdots\cup f(F_r))$ is a conic bundle (see Th.~\ref{conic}). 
We denote by $\Delta\subset Y$ the closure of the discriminant divisor of $f_{|X_0}$, and we  still refer to $\Delta$ as the {\bf discriminant divisor of $f$}. Note that $f(F_i)$ may or may not be in $\Delta$.

By Th.~\ref{special}$(d)$, $B_1,\dotsc,B_m$ are irreducible components of the discriminant divisor $\Delta$.
We define the {\bf intrinsic discriminant divisor} $\Delta_{\intr}$ of $f$ to be the union of the irreducible components of $\Delta$ different from $B_1,\dotsc,B_m$, namely a component $\Delta_0$ of the discriminant divisor is in $\Delta_{\intr}$ if and only if $f^*\Delta_0$ is irreducible. We have $\Delta=\Delta_{\intr}$ if and only if $m=0$, equivalently $f$ is elementary.
\end{prg}
\begin{lemma}\label{conicbdl}
   In the setting of \ref{setting},  for every $i=1,\dotsc,m$
   we have $B_i\cap\Delta_{\intr}=\emptyset$, $B_i$ is a connected component of $\Delta$,
      and $B_i$ is  smooth outside the
   images of the $2$-dimensional fibers. 
\end{lemma}  
\begin{proof}
  Let $i\in\{1,\dotsc,m\}$. Outside the $2$-dimensional fibers, $f$ is a conic bundle,  and $B_i$ is an irreducible component of the discriminant divisor. Since $X$ is smooth, where $f$ is a conic bundle, the singularities of the discriminant divisor correspond to double lines (see \cite[Prop.~1.2]{beauville}). On the other hand there cannot be a double line $F$ over $B_i$, because
if $\Gamma=F_{\text{\it red}}$, we would have $\Gamma\subset E_i\cap\wi{E}_i$ and $\Gamma\equiv e_i$, $\Gamma\equiv \hat{e}_i$, but $e_i\not\equiv\hat{e}_i$ by 
 Th.~\ref{special}$(d)$. Thus $B_i$ is smooth outside the images of the $2$-dimensional fibers and cannot meet other irreducible components of the discriminant divisor (note that, being $Y$ locally factorial, two prime divisors cannot intersect in finitely many points).
\end{proof}
\begin{remark}\label{fiber}
 In the setting of \ref{setting},  let $F:=f^{-1}(y)$ be a fiber of $f$. Then $\dim\N(F,X)>1$ if and only if $y\in B_1\cup\cdots\cup B_m$.
\end{remark}
\begin{proof}
  Suppose that $y\in B_1$. By Th.~\ref{special}$(d)$, $F\cap E_1$ must contain a one-cycle $\Gamma_1$ which is a degeneration of the curve $e_1$, in particular $\Gamma_1\equiv e_1$; similarly $F\cap\wi{E}_1$ contains a one cycle $\wi{\Gamma}_1$ numerically equivalent to $\hat{e}_1$. Therefore $[e_1],[\hat{e}_1]\in\N(F,X)$, and this classes are linearly independent because $E_1\cdot e_1<0$ while $E_1\cdot\hat{e}_1\geq 0$. We conclude that  
  $\dim\N(F,X)>1$.
  
Conversely if $y\not\in B_1\cup\cdots\cup B_m$, then
$F\cap E_i=\emptyset$ for $i=1,\dotsc,m$, hence 
$\N(F,X)\subseteq(\ker f_*)\cap E_1^{\perp}\cap\cdots\cap E_m^{\perp}$, where $f_*\colon \N(X)\to\N(Y)$ is the pushforward (see Rem.~\ref{disjoint}).
Moreover $\ker f_*$ is generated by the classes $[e_1],\dotsc,[e_m],[F_0]$ where $F_0$ is a general fiber of $f$, and one can easily check that $(\ker f_*)\cap E_1^{\perp}\cap\cdots\cap E_m^{\perp}=\R[F_0]$.
We conclude that $\dim\N(F,X)=1$.
\end{proof}
We recall that $\pr^2\bullet\pr^2$ is the union of two planes in $\pr^4$ intersecting in one point.
\begin{lemma}[\cite{kachi}]\label{bullet}
   In the setting of \ref{setting}, 
let $F:=f^{-1}(y_0)$ be a fiber such that $F\cong\pr^2\bullet\pr^2$. Then $y_0\not\in\Delta$ and $\dim\N(F,X)=1$.
\end{lemma}
\begin{proof}
  The fact that $y_0\not\in\Delta$ follows from \cite[Th.~3.1]{kachi}; note that, in the terminology of \cite{kachi}, a limit conic in $F=F_1\cup F_2$ is a union of  a line in $F_1$ and a line in $F_2$, both containing the point $F_1\cap F_2$; hence $F$ is not connected by limit conics, because for general $x,y\in F_1$ there is no limit conic containing both of them. We also note that,
  even if in \cite{kachi} the contraction $f$ is assumed to be elementary, the proof of Th.~3.1 is local around $F$, and only needs that $F$ is an isolated $2$-dimensional fiber.

  Finally $\dim\N(F,X)=1$ follows from Rem.~\ref{fiber}.
\end{proof}
\begin{thm}[\cite{AW,kachi}]\label{2dimfibers}
   In the setting of \ref{setting}, 
      let $F:=f^{-1}(y_0)$ be a $2$-dimensional fiber of $f$. 
Then one of the following holds:
\begin{enumerate}[$(i)$]
\item $Y$ is smooth at $y_0$;
    \item $\dim\N(F,X)=1$ and $y_0\in\Delta_{\intr}$;
  \item $\dim\N(F,X)=1$, $F\cong\pr^2\bullet\pr^2$, $Y$ has a node at $y_0$, and $y_0\not\in\Delta$.
   \end{enumerate}
\end{thm}
\begin{proof}
  If $\dim\N(F,X)=1$, then  there is 
an open neighborhood $Y_0$ of $y_0$ such that, if $X_0:=f^{-1}(Y_0)$, then $f_{|X_0}\colon  X_0\to Y_0$ is elementary, and we apply the results in 
\cite{kachi}. In particular we see that for type A \cite[Th.~0.6]{kachi} we have $(i)$; for type B  \cite[Th.~0.7]{kachi} we have $(i)$ or $(iii)$; for type C \cite[Th.~0.8]{kachi} we have $y_0\in\Delta$, and then $(ii)$ by Rem.~\ref{fiber}.

Suppose instead that $\dim\N(F,X)>1$. The possible $F$ are classified in \cite[Prop.~4.11]{AW}, and moreover $F\not\cong\pr^2\bullet\pr^2$ by Lemma \ref{bullet}, 
thus we have the  possibilities:
$\mathbb{F}_1$, $\pr^1\times\pr^1$, $\pr^2\cup(\pr^1\times\pr^1)$, or $\pr^2\cup \mathbb{F}_1$, where in the two reducible cases,   the components meet along a curve which is a line in $\pr^2$ and a line of the ruling  in $\pr^1\times\pr^1$, or the $(-1)$-curve in $\mathbb{F}_1$.

In all cases we have $\dim\N(F,X)=2$ and $F$ does not contain exceptional planes, so we can choose 
an open neighborhood $Y_0$ of $y_0$ such that, if $X_0:=f^{-1}(Y_0)$, then $\dim\N(X_0/Y_0)=2$, and no fiber of $f$ over $Y_0$ contains an exceptional plane. Set $f_0:=f_{|X_0}\colon  X_0\to Y_0$, and
let us consider a factorization of $f_0$ in elementary steps:
$$X_0\stackrel{\alpha}{\la} W_0\stackrel{\beta}{\la}Y_0.$$

If $\alpha$ is of fiber type, then $\dim W_0=3$ and $\beta$ is birational. However $\beta$ cannot be small because $Y$ is locally factorial (Th.~\ref{special}$(a)$), and it cannot be divisorial otherwise there would be a divisor in $X_0$ sent by $f_0$ to a closed subset of codimension $\geq 2$ in $Y_0$, impossible because $f$ is special. Therefore $\alpha$ is birational and $\dim W_0=4$.

Since $f$ is $K$-negative, $\alpha$ is $K$-negative too. Hence $\alpha$ cannot be small, otherwise $\Exc(\alpha)$ 
would be a union of exceptional planes contained in fibers of $f_0$ (see Th.~\ref{kawsm}).
Then $\alpha$ is divisorial,  $\dim\alpha(\Exc(\alpha))\leq 2$ and $\dim f_0(\Exc(\alpha))\leq 2$. On the other hand $f$ is special, so 
$f_0(\Exc(\alpha))$ must be a divisor in $Y_0$. We conclude that $\alpha$ is
of type $(3,2)$, and $f_0(\Exc(\alpha))=Y_0\cap B_i$ for some $i$.

We also note that the general fiber $\Gamma$ of $\beta$ is a curve disjoint from $\alpha(\Exc(\alpha))$, and $f$ is $K$-negative, thus $K_{W_0}\cdot\Gamma<0$. Since $\dim\N(W_0/Y_0)=1$, this means that $\beta$ is $K$-negative.

We have
$$F\stackrel{\alpha}{\la}\alpha(F)=\beta^{-1}(y_0)\stackrel{\beta}{\la}y_0$$
and  $y_0\in f_0(\Exc(\alpha))$
by Rem.~\ref{fiber}, thus $F\cap\Exc(\alpha)\neq\emptyset$, and $\alpha_{|F}$ is not an isomorphism.

If $F$ is irreducible, then $F\cong \mathbb{F}_1$ or $\pr^1\times\pr^1$, and $\alpha_{|F}$ is either a $\pr^1$-bundle, or the contraction of the $(-1)$-curve in $\mathbb{F}_1$. Up to exchanging this factorization of $f_0$ with the other one, we can assume that $\alpha_{|F}$ is a $\pr^1$-bundle. Then $\alpha$ has one-dimensional fibers over $\alpha(F)$, therefore
$\alpha(F)\subset (W_0)_{\reg}$ by Th.~\ref{32}. Moreover
$\alpha(F)=\beta^{-1}(y_0)$ is a one-dimensional fiber of $\beta$, 
and we have $(i)$ by Th.~\ref{conic}.

If instead $F$ is reducible, then $F=F_1\cup F_2$ with $F_1\cong\pr^2$ and $F_2\cong \mathbb{F}_1$ or $\pr^1\times\pr^1$. Again 
up to exchanging this factorization of $f_0$ with the other one, we can assume that $\alpha_{|F_1}$ is an isomorphism and $\alpha_{|F_2}$ is a $\pr^{1}$-bundle, so that $\alpha(F)\cong\pr^2$ and
$\alpha(F)\cap\alpha(\Exc(\alpha))$ is a curve.
Again
$\alpha$ has one-dimensional fibers over $\alpha(F)\cap\alpha(\Exc(\alpha))$, hence 
$\alpha(F)\subset (W_0)_{\reg}$ by Th.~\ref{32}. In this case $\alpha(F)=\beta^{-1}(y_0)\cong\pr^2$ is a $2$-dimensional fiber of $\beta$,
and we get $(i)$ by \cite[Th.~5.9.6]{AW}.
\end{proof}
  \begin{example}\label{artinmum}
    We give an example of a smooth Fano $4$-fold $X$ with $\rho_X=2$ with an elementary, $K$-negative contraction $f\colon X\to Y$ with $\dim Y=3$, where $Y$ has 10 nodes $y_i$, 
 $f^{-1}(y_i)\cong\pr^2\bullet\pr^2$ for every $i$, and $f$ is smooth with fiber $\pr^1$ over $Y\smallsetminus\{y_1,\dotsc,y_{10}\}$, so that $\Delta=\emptyset$.
 The resolution of the base $Y$ is the well-known Artin and Mumford's $3$-fold \cite{artinmumford};
we outline the description of $X$ and refer the reader to
 \cite[\S 9]{beaurational} for more details.

 Let $G$ be the grassmannian of lines in $\pr^3$, and $R\subset G$
the subvariety given by the lines contained in a fixed pencil of quadrics in $\pr^3$. Then 
$R$  is an Enriques surface, called
 Reye congruence.
  Let $\sigma\colon X\to G$ be the blow-up along $R$; then $X$ is a smooth projective $4$-fold with $\rho_X=2$. Moreover there is an elementary contraction $f\colon X\to Y$, where $Y$ is a locally factorial $3$-fold with precisely $10$ nodes, and $f$ is smooth with fiber $\pr^1$ outside the nodes. In particular $f$ and $\sigma$ are $K$-negative, and $X$ is Fano.
  The fibers of $f$ over the nodes are isomorphic to $\pr^2\bullet\pr^2$.
\end{example}
\begin{remark}\label{conto}
  Let $f\colon X\to Y$ be a conic bundle where $Y$ is smooth, quasi-projective and $X$ has at most isolated singularities, and let $\Delta\subset Y$ be the discriminant divisor.

  Let $p\in\Delta$ be such that $f^{-1}(p)$ has two components, and let $q\in f^{-1}(p)$ be the singular point of the fiber.
  Then $X$ is singular at $q$ if and only if $\Delta$ is singular at $p$.
\end{remark}
\begin{proof}
  The statement being local on $Y$, we can assume that $Y$ is affine and that $X\subset Y\times\pr^2$ is defined by the equation
  $$F=\sum_{i,j=0}^2a_{ij}(y)x_ix_j=0,$$
  where $y=(y_j)$ are local coordinates on $Y$ at $p$, $a_i\in\ol(Y)$, and $(x_0:x_1:x_2)$ are coordinates on $\pr^2$. We can also assume that $f^{-1}(p)$ has equation $x_1^2+x_2^2=0$, so that $a_{11}(p)=a_{22}(p)=1$ and $a_{ij}(p)=0$ for all other indices. Then $f^{-1}(p)$ is singular at $q_0=(1:0:0)$, and $q=(p,q_0)$.

   We have $\frac{\partial F}{\partial x_i}(q)=0$ for $i=0,1,2$, and $\frac{\partial F}{\partial y_j}(q)=\frac{\partial a_{00}}{\partial y_j}(p)$ for every $j$.

 On the other hand, the local equation of $\Delta$ at $p$ is given by $D=\det(a_{ij})$. The determinant is given by a sum (with signs) of products of three elements of the matrix $(a_{ij})$ lying in different rows and columns. Each such product can be expressed as a product of two functions vanishing at $p$, except $a_{00}a_{11}a_{22}$. We conclude that
  $$\frac{\partial D}{\partial y_j}(p)=\frac{\partial (a_{00}a_{11}a_{22})}{\partial y_j}(p)=\frac{\partial a_{00}}{\partial y_j}(p)=\frac{\partial F}{\partial y_j}(q)$$
and the statement follows.
 See also \cite[Prop.~1.2]{beauville}.
\end{proof}
\subsection{Special rational contractions from a Fano $4$-fold to a $3$-fold}\label{secspecial2}
 \begin{definition}\label{specialrat}
          Let $X$ be a normal and $\Q$-factorial projective variety, and a Mori dream space.
          A rational contraction of fiber type $f\colon X\dasharrow Y$ is {\bf special}
                    if there is a SQM $X\stackrel{\xi}{\dasharrow} \w{X}$ such that the composition $f\circ \xi^{-1}\colon \w{X}\to Y$ is regular and special; this does not depend on the choice of the SQM $\xi$.
\end{definition}
We will need the following properties.
\begin{proposition}[\cite{fibrations}, Prop.~2.13]\label{factor}
  Let $X$ be a normal and $\Q$-factorial projective variety, and a Mori dream space. Let $f\colon X\dasharrow Y$ be a  rational contraction of fiber type. Then $f$ can be factored as $X\stackrel{f'}{\dasharrow} Y'\stackrel{g}{\to}Y$ where $f'$ is a special rational contraction and $g$ is birational.
\end{proposition}
\begin{lemma}\label{rhodelta}
  Let $X$ be a smooth Fano $4$-fold and $X\dasharrow Y$ a special rational contraction with $\dim Y=3$. 
  If $\rho_X-\rho_Y\geq 3$, then $\delta_X\geq\rho_X-\rho_Y-1\geq 2$.
\end{lemma}
\begin{proof}
  The argument is similar to \cite[proof of Lemma 3.10]{eleonora}.
  We consider a factorization
  $$X\stackrel{\xi}{\dasharrow}\w{X}\stackrel{f}{\la}Y$$
  where $\xi$ is a SQM and $f$ is a $K$-negative special contraction (see Lemma \ref{Kneg}),
and set $m:=\rho_X-\rho_Y-1$. By Th.~\ref{special}
there are pairwise disjoint prime divisors $B_1,\dotsc,B_m\subset Y$ such that $f^*(B_i)$ has two irreducible components $E_i$ and $\wi{E}_i$. Moreover the general fiber of $f$ over $B_i$ is $e_i+\hat{e}_i$ where $E_i\cdot e_i<0$, $\wi{E}_i\cdot\hat{e}_i<0$, and $-K_{\w{X}}\cdot e_i=-K_{\w{X}}\cdot \hat{e}_i=1$.

We show that $E_i$ and $\wi{E}_i$ are covered by irreducible curves of anticanonical degree one. Indeed if $p\in E_i$, there must be an effective one-cycle $\Gamma$ which is a degeneration of $e_i$ and containing $p$ in its support. Then $\Gamma\equiv e_i$, thus every irreducible component of $\Gamma$ is contracted by $f$, which is $K$-negative. Since $-K_{\w{X}}\cdot\Gamma=1$, $\Gamma$ must be an integral curve. Similarly for $\wi{E}_i$.

By  Lemma \ref{SQMFano}$(c)$ this implies that $E_i\cup\wi{E}_i\subset\dom(\xi^{-1})$. Let $E_i',\wi{E}_i'\subset X$ be the transforms
of $E_i,\wi{E}_i$ respectively; we have $(E_i'\cup\wi{E}_i')\cap (E_j'\cup\wi{E}'_j)=\emptyset$ for every $i\neq j$.

Since  $\rho_X-\rho_Y\geq 3$, we have
$m\geq 2$.
  The divisor $E'_1$ is disjoint from $E'_2,\dotsc,E'_m,\wi{E}'_m$, thus
  $$\N(E'_1,X)\subseteq (E_2')^{\perp}\cap\cdots\cap (E_m')^{\perp}\cap (\wi{E}_m')^{\perp}$$
  (see Rem.~\ref{disjoint}).
  Moreover intersecting with (the transforms of) the curves $e_i$, $\hat{e}_m$ we see that the classes  $[E'_2],\dotsc,[E'_m],[\wi{E}'_m]\in\Nu(X)$ are linearly independent, and this gives $\codim\N(E'_1,X)\geq m$, hence $\delta_X\geq m=\rho_X-\rho_Y-1$. 
\end{proof}
\section{Weak Fano $3$-folds}\label{weakfano}
\noindent We recall that a normal and $\Q$-factorial projective variety $Y$ is \emph{weak Fano} if $-K_Y$ is nef and big. This is an auxiliary
section where we present some results on Fano and weak Fano $3$-folds $Y$ with locally factorial and canonical singularities; in particular we are interested in  bounding $-K_Y^3$. These results will be applied in the rest of paper to study the base of a special rational contraction $X\dasharrow Y$ where $X$ is a Fano $4$-fold and $\dim Y=3$. The reader may skip this section and come back to it when needed.

Let $Y$ be a weak Fano $3$-fold with locally factorial, canonical singularities. Then $Y$ is log Fano, and $|-mK_Y|$ for $m\gg 0$ defines a birational map $\ph\colon Y\to Z$, that we call the \emph{anticanonical map} of $Y$. Moreover $Z$ is a Gorenstein Fano $3$-fold, the \emph{anticanonical model} of $Y$.
\begin{lemma}[\cite{prok,karz}]\label{bound}
  Let $Y$ be a Fano $3$-fold with at most locally factorial and canonical singularities.
  Then $-K_Y^3\leq 64$.
\end{lemma}
\begin{proof}
We use the results in \cite{prok,karz} on the anticanonical degree of Fano $3$-folds with Gorenstein canonical singularities.
  
  By \cite[Th.~1.5]{prok} we have $-K_Y^3\leq 72$, and if $-K_Y^3= 72$
then $Y$ should be
$\pr(1,1,1,3)$ or $\pr(1,1,4,6)$, but these two varieties are not locally factorial.

  Then \cite[Th.~1.5]{karz} shows that, if $64<-K_Y^3<72$, then there are two possibilities for $Y$, with $-K_Y^3=66$ or $-K_Y^3=70$.

  If $-K_Y^3=70$, then by \cite[\S 3, in particular p.\ 1226]{karz} there is a birational morphism $\tau\colon W\to Y$ with $\Exc(\tau)$ an irreducible curve, which is contracted to the unique singular point of $Y$. Then $Y$ is not even $\Q$-factorial, indeed let $H\subset W$ be a general very ample divisor, so that $H$ intersects $\Exc(\tau)$ in points. Then $\tau(H)\subset Y$ is a prime Weil divisor which cannot be $\Q$-Cartier, because $\tau^{-1}(\tau(H))=H\cup\Exc(\tau)$ is not a divisor, hence the pullback of $m\tau(H)$ does not exist for any $m\in\Z_{>0}$.

  If $-K_Y^3=66$, then by \cite[Th.~1.5, Prop.~5.2]{karz} $Y$ is toric and singular, thus again $Y$ cannot be locally factorial.
\end{proof}
\begin{lemma}[\cite{ou}]\label{ou}
  Let $Y$ be a Fano $3$-fold with at most isolated, locally factorial, and canonical singularities.
  Assume also that $-K_Y^3> 24$, $\rho_Y=2$, and that $Y$ has two distinct elementary contractions of fiber type. Then $Y$ is smooth and rational, and one of the following holds:
  \begin{enumerate}[$(i)$]
  \item $Y\cong\pr^1\times\pr^2$ and $-K_Y^3=54$;
  \item  $Y\cong\pr_{\pr^2}(T_{\pr^2})$ and $-K_Y^3=48$;
\item $Y\subset\pr^2\times\pr^2$ is a divisor of degree $(1,2)$, and $-K_Y^3=30$.
\end{enumerate}
\end{lemma}
\begin{proof}
 Without the assumption on $-K_Y^3$, the possible $Y$'s are classified in
  \cite[Th.~1.2]{ou}; there are $7$ families, and they are all degenerations of smooth Fano $3$-folds. By checking the anticanonical degree we get the statement. Note that in case $(iii)$, the first projection from $\pr^2\times\pr^2$ realizes $Y$ as a $\pr^1$-bundle over $\pr^2$, thus $Y$ is smooth and rational.   
  \end{proof}
\begin{proposition}\label{sing}
    Let $Y$ be a weak Fano $3$-fold, not Fano, with at most isolated, canonical, and locally factorial singularities. Assume also that $\rho_Y=2$,  that $Y$ has two distinct elementary rational contractions of fiber type, and that the anticanonical map of $Y$ is small with exceptional locus contained in $Y_{\reg}$.  Then one of the following holds:
    \begin{enumerate}[$(i)$]
    \item $Y\cong\pr_{\pr^1}(\ol\oplus\ol(1)^{\oplus 2})$ and $-K_Y^3=54$;
  \item up to flops, $Y$ is as in \cite[Th.~3.6(1)]{jahnkepeternell} and $-K_Y^3=40$;
       \item $-K_Y^3\leq 32$.
       \end{enumerate}
     \end{proposition}
     Before proving this proposition, we need some auxiliary results. The techniques used below are standard in the study of weak Fano threefolds, see for instance \cite{fanoEMS,jahnkepeternell,JPRII,prok,prokratII}.

Let $Y$ be as in Prop.~\ref{sing}, and set $Y_1:=Y$.
  Since $\rho_{Y_1}=2$ and $-K_{Y_1}$ is nef and big but not ample, $Y_1$ has two elementary contractions, one $g_1\colon Y_1\to W_1$ which is $K$-negative, and another one which is $K$-trivial.  This last one is the anticanonical map $\ph\colon Y_1\to Z$, so it is small by assumption, with exceptional locus contained in $(Y_1)_{\reg}$; let $\xi\colon Y_1\dasharrow Y_2$ be its flop. Then $Y_2$  is still weak Fano, and since terminal flops preserve the singularity type (see \cite[Th.~6.15]{kollarmori}), also the indeterminacy locus of $\xi^{-1}$ is contained in $(Y_2)_{\reg}$, and $Y_2$ has the same singularities as $Y_1$; note that $-K_{Y_1}^3=-K_{Y_2}^3$. The second elementary contraction $g_2\colon Y_2\to W_2$ is $K$-negative.
  $$\xymatrix{{Y_1}\ar[d]_{g_1}\ar@{-->}[r]^>>>>>{\xi}&{Y_2}\ar[d]^{g_2}\\
{W_1}&{W_2}
}$$

We note that $g_i$ cannot be small by Th.~\ref{gdn}, therefore $\Mov(Y_1)=\Nef(Y_1)\cup \xi^*(\Nef(Y_2))$.
Since $Y_1$ has two elementary rational contractions of fiber type, these must be $g_1$ and $g_2\circ\xi$, namely $g_1$ and $g_2$ must be of fiber type.

Let $i\in\{1,2\}$. If $\dim W_i=1$, then $W_i\cong\pr^1$ and $g_i$ is a fibration in del Pezzo surfaces.
If  $\dim W_i=2$, then the surface $W_i$ is smooth (see \cite[Lemma 5.5]{ou}), and it is  rational  with $\rho_{W_i}=1$, thus $W_i\cong\pr^2$; in this case $g_i$ is a conic bundle (see Prop.~\ref{conicsing}).

Let $D_i\subset Y_i$ be a general fiber of $g_i$ if $W_i\cong\pr^1$, or the pullback of a general line if $W_i\cong\pr^2$. Let also 
 $\w{D}_2\subset Y_1$ be the transform of $D_2\subset Y_2$.
  Then $[D_1],[\w{D}_2]\in\Nu(Y_1)$ generate the cone of effective divisors, and since $Y_1$ is locally factorial, we have
  \stepcounter{thm}
  \begin{equation}\label{freccia}
    m(-K_{Y_1})=a_1D_1+a_2\w{D}_2
    \end{equation}
with $m,a_1,a_2$ positive integers such that $\gcd(m,a_1,a_2)=1$.
\begin{lemma}\label{P2bundle}
If $W_1\cong\pr^1$ and $D_1\cong\pr^2$, then 
$Y_1\cong Y_2\cong\pr_{\pr^1}(\ol\oplus\ol(1)^{\oplus 2})$ and $-K_Y^3=54$.
\end{lemma}
\begin{proof}
We have $Y_1=\pr_{\pr^1}(\ma{F})$ where $\ma{F}$ is a rank $3$ vector bundle on $\pr^1$, because   $Y_1$ is locally factorial (see \cite[Rem.~3.2]{hoeringnovelli} and references therein); moreover $\ma{F}$ is decomposable. Since $Y_1$ is weak Fano, not Fano, and it has small anticanonical map, the only possibility is $Y_1\cong\pr_{\pr^1}(\ol\oplus\ol(1)^{\oplus 2})$. Then $-K_{Y_1}^3=54$ and $Y_2\cong Y_1$, and we have the statement.
\end{proof}  
\begin{lemma}\label{P1bundle}
  If $W_1\cong\pr^2$ and $g_1$ is smooth, then one of the following holds:
  \begin{enumerate}[$(i)$]
  \item $Y$ is as in \cite[Th.~3.6(1)]{jahnkepeternell} and $-K_Y^3=40$;
       \item $-K_Y^3\leq 24$.
       \end{enumerate}
     \end{lemma}
     \begin{proof}
We have $Y=\pr_{\pr^2}(\ma{F})$ with $\ma{F}$ a rank $2$ vector bundle on $\pr^2$. The possible $\ma{F}$ such that $Y$ is weak Fano have been classified in \cite{langer} when $\ma{F}$ has odd degree, and \cite{yasutake} when $\ma{F}$ has even degree.

When $\ma{F}$ has even degree, we can assume that $c_1(\ma{F})=0$, and since $Y$ is not Fano, by \cite[Prop.~2.10 and 2.11]{yasutake}  we have $c_2(\ma{F})\in\{4,5,6\}$ and $-K_Y^3=54-8c_2(\ma{F})\in\{6,14,22\}$.

When $\ma{F}$ has odd degree, we can assume that $c_1(\ma{F})=-1$, then $-K_Y^3=8(7-c_2(\ma{F}))$. The possibilities for $\ma{F}$ are given in \cite[Th.~3.2]{langer}; moreover
in this case 
$Y=\pr_{\pr^2}(\ma{F})$ has automatically index\footnote{The index of a weak Fano variety is the divisibility of $-K$ in the Picard group, as for Fano varieties.} 2; when the anticanonical map is small 
these threefolds have been classified also in \cite[Th.~3.6]{jahnkepeternell}.

Let us assume that $-K_Y^3> 24$, equivalently that  $c_2(\ma{F})<4$.
Since $Y$ is not Fano, by \cite[Th.~3.2]{langer} we have
$c_2(\ma{F})\in\{-2,1,2,3\}$.
Moreover, since the anticanonical map is small, we have 
$c_2(\ma{F})\in\{2,3\}$ by \cite[Th.~3.6]{jahnkepeternell}.
If   $c_2(\ma{F})=2$ then $-K_Y^3=40$ and $Y$ is as in \cite[Th.~3.2.5]{langer}, which is the same as \cite[Th.~3.6(1)]{jahnkepeternell} (and also \cite[2.13(iv)]{JPRII}).  Finally, if
$c_2(\ma{F})=3$, $Y$ should be as in \cite[Th.~3.6(2)]{jahnkepeternell} but in this case  $g_2\colon Y_2\to W_2$ is birational, which is  excluded by our assumptions.
\end{proof}
\begin{lemma}\label{conicdP}
  Suppose that $g_i$ contracts some irreducible curve $\Gamma_i\subset Y_i$ with $-K_{Y_i}\cdot\Gamma_i=1$ for $i=1,2$.
  Then $-K_Y^3\leq 24$.
  \end{lemma}
\begin{proof}
  Intersecting \eqref{freccia} with $\Gamma_1$ we get $m=a_2b$ with $b:=\w{D}_2\cdot\Gamma_1$, thus $a_2(b(-K_{Y_1})-\w{D}_2)\sim a_1 D_1$. This implies that $\ol_{Y_1}(b(-K_{Y_1})-\w{D}_2)\in (g_1)^*\Pic(W_1)=\Z \ol_{Y_1}(D_1)$ (see \cite[Th.~3.7(4)]{kollarmori}), therefore $a_2|a_1$. Working in $Y_2$ and intersecting with $\Gamma_2$, we get $a_1|a_2$ and hence $a_1=a_2$. On the other hand $a_2|m$ and $\gcd(a_2,m)=1$, and we conclude that $a_1=a_2=1$ and $m(-K_{Y_1})=D_1+\w{D}_2$, so that:
  \stepcounter{thm}
  \begin{equation}\label{verylast}
  -K_{Y_1}^3\leq (-K_{Y_1})^2\cdot (-mK_{Y_1})=
  (-K_{Y_1})^2\cdot D_1+(-K_{Y_1})^2\cdot\w{D}_2.\end{equation}
  
Let $i\in\{1,2\}$. If $W_i\cong\pr^1$, then $D_i$ is a del Pezzo surface, thus $(-K_{Y_i})^2\cdot D_i=(-K_{D_i})^2\leq 9$.
If instead $W_i\cong\pr^2$, then $g_i$ is a conic bundle; the surface $D_i$ is  smooth, rational, and has a conic bundle over $\pr^1$ with $d_i$ singular fibers, where $d_i$ is the degree of the discriminant divisor of $g_i$. An elementary computation gives
$(-K_{Y_i})^2\cdot D_i=12-d_i\leq 12$.

Finally we have $(-K_{Y_1})^2\cdot \w{D}_2=(-K_{Y_2})^2\cdot D_2$; this is a well-known fact for flops, that can be seen as follows. Consider  a locally factorial resolution of $\xi\colon Y_1\dasharrow Y_2$:
$$\xymatrix{&{\wi{Y}}\ar[dl]_{\alpha_1}\ar[dr]^{\alpha_2}&\\
Y_1\ar@{-->}[rr]^{\xi}&&{Y_2}
  }$$
  We have $\alpha_1^*K_{Y_1}=\alpha_2^*K_{Y_2}$; recall that $\ph\colon Y_1\to Z$ is the anticanonical map. If $E\subset \wi{Y}$ is an irreducible exceptional divisor of $\alpha_1$, we have $(\ph\circ\alpha_1)(E)=\{\pt\}$, and $\alpha_1^*K_{Y_1}=(\ph\circ\alpha_1)^*K_Z$, thus $(\alpha_1^*K_{Y_1})_{|E}\equiv 0$. This implies that, for every pair of divisors $B,B'$ in $Y_1$, if $\w{B},\w{B}'\subset Y_2$ are their transforms, we have  $K_{Y_1}\cdot B\cdot B'=  K_{Y_2}\cdot \w{B}\cdot \w{B}'$.

From \eqref{verylast}
 we conclude that $-K_{Y_1}^3\leq 24$.
\end{proof}
\begin{lemma}\label{quadric}
  Suppose that $W_1\cong \pr^1$, $D_1\cong\pr^1\times\pr^1$,  $D_2\not\cong\pr^2$, and that if $W_2\cong\pr^2$, then $g_2$ is not smooth.
  Then $-K_Y^3\leq 32$.
\end{lemma}
\begin{proof}
  We proceed similarly to the proof of Lemma \ref{conicdP}. Let $\Gamma_1\subset D_1$ be
a line in the quadric, and note that $(-K_{D_1})^2=8$. Suppose first that there exists an irreducible curve $\Gamma_2\subset Y_2$ contracted by $g_2$ such that $-K_{Y_2}\cdot\Gamma_2=1$; as before we have $(-K_{Y_2})^2\cdot D_2\leq 12$.
Intersecting \eqref{freccia} with $\Gamma_1$ we get $2m=a_2b$ with $b:=\w{D}_2\cdot\Gamma_1$, thus $a_2(b(-K_{Y_1})-2\w{D}_2)=2a_1D_1$ and $a_2|2a_1$;
intersecting with $\Gamma_2$ we get $a_1|a_2$. This implies that either $a_1=a_2$ or $2a_1=a_2$. Then using that $a_1|m$ and $\gcd(a_1,a_2,m)=1$ we deduce that $a_1=1$,
$a_2\leq 2$, and $-K_Y^3\leq 8a_1+12a_2\leq 32$.

Suppose now $g_2$ does not contract curves of anticanonical degree one. Then, by our assumptions, it must be $W_2\cong\pr^1$ and $D_2\cong\pr^1\times\pr^1$. By taking $\Gamma_2$ a line in the quadric $D_2$ and proceeding as above, we get $a_i\leq 2$ for $i=1,2$, $(-K_{D_2})^2=8$, and $-K_Y^3\leq 8a_1+8a_2\leq 32$.
\end{proof}
\begin{proof}[Proof of Prop.\ \ref{sing}]
  If for some $i\in\{1,2\}$ we have $W_i\cong\pr^1$ and $D_i\cong\pr^2$, the statement follows from Lemma \ref{P2bundle}. Similarly, if for some $i\in\{1,2\}$ we have $W_i\cong\pr^2$ and $g_i$ is smooth,   the statement follows from Lemma \ref{P1bundle}.  Therefore we can assume that, for $i=1,2$, either $W_i\cong\pr^2$ and $g_i$ is singular, or  $W_i\cong\pr^1$ and $D_i\not\cong\pr^2$.

  Now if  for $i=1,2$ $g_i$ contracts some irreducible curve of anticanonical degree one, we apply Lemma \ref{conicdP}. Otherwise, up to switching $Y_1$ and $Y_2$ we can assume that $W_1\cong\pr^1$ and $D_1\cong\pr^1\times\pr^1$, and we apply Lemma \ref{quadric}.
\end{proof}  
 \begin{lemma}\label{quadricsm}
  In the setting of Lemma \ref{quadric}, suppose moreover that $Y$ is smooth.
   Then one of the following holds:
  \begin{enumerate}[$(i)$]
  \item $Y$ is as in \cite[Th.~3.5(3)]{jahnkepeternell} and $-K_Y^3=32$;
       \item $-K_Y^3\leq 24$.
       \end{enumerate}
\end{lemma}
\begin{proof}
  Smooth weak Fano threefolds with a del Pezzo fibration, and having small anticanonical map, have been classified in \cite{JPRII} and \cite{takeuchi}. In our setting $Y_1$  has a quadric fibration and $g_2$ is either a singular conic bundle, or another del Pezzo fibration. Then, by \cite[Th.~2.3]{takeuchi}, if  $-K_Y^3> 24$ the only possibility for $Y_1$ is \cite[(2.3.4)]{takeuchi}; then $Y_1$ has index two, and it is the same as \cite[Th.~3.5(3)]{jahnkepeternell} (and also  \cite[2.13(1.iii)]{JPRII}).
\end{proof}
\begin{proposition}\label{smoothwfrho2}
Let $Y$ be a smooth weak Fano $3$-fold with $\rho_Y=2$ and with two distinct elementary rational contractions of fiber type. Assume moreover that $Y$ is not Fano, and that the anticanonical map of $Y$ is small. 
Then
 one of the following holds:
  \begin{enumerate}[$(i)$]
  \item
    $Y\cong\pr_{\pr^1}(\ol\oplus\ol(1)^{\oplus 2})$ and $-K_Y^3=54$;
  \item up to flops $Y$ is as in \cite[Th.~3.6(1)]{jahnkepeternell} and $-K_Y^3=40$;
     \item $Y$ is as in  \cite[Th.~3.5(3)]{jahnkepeternell} and $-K_Y^3=32$;
       \item $-K_Y^3\leq 24$.
       \end{enumerate}
     \end{proposition}
     \begin{proof}
We proceed as in the proof of Prop.~\ref{sing}, just applying Lemma \ref{quadricsm} instead of Lemma \ref{quadric}.
     \end{proof}  
\section{The  case of relative Picard number two}\label{relrho2}
\noindent In this section we study Fano $4$-folds having a special rational contraction onto a $3$-dimensional target, with relative Picard number two, and we prove Th.~\ref{Bintro} from the Introduction; for the reader's convenience, we report the statement here.
\begin{thm}\label{B}
  Let $X$ be a smooth Fano $4$-fold that not isomorphic to a product of surfaces, and having a special rational contraction 
$f_{\s X}\colon X\dasharrow Y$ with $\dim Y=3$ and $\rho_X-\rho_Y=2$.  
Then $\rho_X\leq 9$.

Moreover, if $\rho_X\geq 7$, then $X$ is the blow-up of $W$ along a normal surface $S$,
where $W$ is the Fano model of the blow-up of $\pr^4$ at $\rho_X-2$  points (see Ex.~\ref{Fanomodel}), and $S\subset W$ is the transform of a surface $A\subset\pr^4$ containing the blown-up points, as follows:
\begin{enumerate}[$(i)$]
\item $A$ is a cubic scroll;
\item $A$ is a cone over a twisted cubic;
\item $A$ is a sextic (singular) K3 surface, with rational double points of type $A_1$ or $A_2$ at the blown-up points, and  $\rho_X=7$.
\end{enumerate}
In cases $(i)$ and $(iii)$ the surface $S$ is smooth, while in $(ii)$ $S$ has one singular point, given by the vertex of the cone.
Moreover $Y$ is smooth, and up to flops $Y\cong\Bl_{\pts}\pr^3$.
\end{thm}
\begin{prg}[Outline of the proof]\label{overview}
The proof of Th.~\ref{B} is quite long and articulated, and it will take the whole section; let us outline the strategy. We assume that $\rho_X\geq 7$, so that $\rho_Y\geq 5$.

We consider a SQM $\xi\colon X\dasharrow\w{X}$ such that $f:=f_{\s X}\circ\xi^{-1}\colon\w{X}\to Y$ is regular and $K$-negative.
There is a unique prime divisor $B\subset Y$ such that $f^*(B)=E_1+E_2$ is reducible, and it is a connected component of the discriminant divisor of $f$ (\ref{discrf}).

We show that $f$ factors as $\w{X}\stackrel{\tilde\alpha}{\to} \w{W}\stackrel{\pi}{\to}Y$, where $\tilde{\alpha}$ is a divisorial elementary contraction of type $(3,2)$ with exceptional divisor $E_1$, and that there is a divisorial elementary contraction $\alpha\colon X\to W$ of type $(3,2)$ with exceptional divisor the transform of $E_1$, so that there is a SQM $W\dasharrow\w{W}$ (see diagram \eqref{diagramW_0} below).

The first part  of the proof follows the same lines as \cite{eff,fibrations}: we show that $Y$ is weak Fano  and that, up to flops, there is a blow-up $k\colon Y\to Y_0$ of $r$ smooth, distinct points $p_1,\dotsc,p_r\in Y_0$, where $Y_0$ is
a weak Fano $3$-fold $Y_0$ with $\rho_{Y_0}\leq 2$  (\ref{weakFano}, Lemma \ref{blowups}).
Then we show, in sequence:
\begin{enumerate}[$\bullet$]
\item the anticanonical map $\ph\colon Y\to Z$ of $Y$ is small, and $\rho_Z=1$ (\ref{properties}, Lemma \ref{rhoZ});
\item $-K_Y=\lambda B$ for some $\lambda\in\Q_{>0}$ (Lemma \ref{index});
\item the composition $\ph\circ  f_{\s X}\colon X\to Z$ is regular and factors through $\alpha$, so it gives a contraction of fiber type $W\to Z$ (\ref{regular} - see diagram \eqref{diagramW_0} below); 
  \item $B$ is the discriminant of $f$, and $\Delta_{\intr}=\emptyset$ (Lemma \ref{discriminant}, see \ref{setting} for $\Delta_{\intr}$); 
\item $Y$ and $Y_0$ are nodal (\ref{nodes}).
\end{enumerate}

Finally we show that $Y$ and $Y_0$ are smooth and rational, and thanks to the constraints given by our setting,  that (up to flops) there are only six possibilities for $Y_0$ (Lemma \ref{table}). We mention here that we use the rationality of $Y$ to deduce smoothness, because the fibers of our special contraction over the nodes are unions of two copies of $\pr^2$ intersecting transversally at one point, and the lines in the two $\pr^2$'s are numerically equivalent; this gives an obstruction to rationality (Lemma \ref{sm}).

Then we show that $f$ descends to a special contraction $f_0\colon X_0\to Y_0$ with $\rho_{X_0}-\rho_{Y_0}=2$,
where $X\dasharrow X_0$ is a birational (rational) contraction, and we can describe it using the classification of fixed prime divisors of $X$ (Lemma \ref{Xr}).
Also the factorization of $f$ descends to a
factorization of $f_0$ as $X_0\stackrel{\alpha_0}{\to} W_0\stackrel{\pi_0}{\to} Y_0$, where $\alpha_0$ is a divisorial elementary contraction of type $(3,2)$ with exceptional divisor the transform of $E_1$ (\ref{facto}). Moreover $\pi_0$
is a $\pr^1$-bundle outside possibly finitely many $2$-dimensional fibers, and the fibers $\pi_0^{-1}(p_i)$ over the blown-up points are isomorphic to $\pr^1$ (Lemma \ref{smooth}).

Using the properties of fixed prime divisors in $X$, we show that the birational map $\w{W}\dasharrow W_0$ factors as a SQM $\w{W}\dasharrow\wi{W}$ followed by a divisorial contraction $\sigma_{\s W}\colon \wi{W}\to W_0$, where for each $i=1,\dotsc,r$ $\sigma_{\s W}$ blows-up either the fiber $\pi_0^{-1}(p_i)$, or a point in that fiber (Lemma \ref{bigdiagram}).
\stepcounter{thm}
\begin{equation}
\label{diagramW_0}
\xymatrix{X\ar@{-->}[r]^{\xi}\ar[d]_{\alpha}&{\w{X}}\ar@/_1pc/[dd]_<<<<<<<{f}\ar[d]^{\tilde{\alpha}}\ar@{-->}[r]&{\wi{X}}\ar[r]^{\sigma}&{X_0}\ar[d]_{\alpha_0}\ar@/^1pc/[dd]^{f_0}\\
 W\ar[d]\ar@{-->}[r]^{\xi_{\s W}} &{\w{W}}\ar[d]^{\pi}\ar@{-->}[r]&{\wi{W}}\ar[r]^{\sigma_{\scriptscriptstyle W}}&{W_0}\ar[d]_{\pi_0}\\
 Z &Y\ar[rr]^k\ar[l]^{\ph}&&{Y_0}
}\end{equation}

Now we use the constraints on the possible anticanonical degrees of curves in $W_0$ and $\wi{W}$, together with the $\pr^1$-bundle structure of $\pi_0$ outside finitely many points of $Y_0$, to exclude five out of six possibilities for $Y_0$, and conclude that $Y_0\cong\pr^3$ and
$Y\cong\Bl_{r\,\pts}\pr^3$ (Lemma \ref{1} -- \ref{last}).

We also show that $W_0$ is Fano (\ref{last}) and that, up to contracting $E_2$ instead of $E_1$ in the factorizations of $f$ and $f_0$, $\sigma_{\s W}$ blows-up one point in each fiber $\pi_0^{-1}(p_i)$ for $i=1,\dotsc,r$ (\ref{switch}),
and $\pi_0$ is a $\pr^1$-bundle over $\pr^3$ (Lemma \ref{psi}).

Fano $4$-folds with a $\pr^1$-bundle structure over $\pr^3$ are classified, and in our setting the only possibility is
$W_0\cong\Bl_{\pt}\pr^4$; this implies that $\wi{W}=\Bl_{r+1\,\pts}\pr^4$ and that $W$ is the Fano model of $ \Bl_{r+1\,\pts}\pr^4$ (Lemma \ref{Wr}).
This, together with the fact that there is a non-trivial contraction of fiber type $W\to Z$, implies that $\rho_W\leq 8$ and hence $\rho_X\leq 9$ (Lemma \ref{bound9}), so we get the bound on $\rho_X$.

Finally we identify the surface $S\subset W$ which is blown-up by $\alpha\colon X\to W$ as the transform of a suitable surface $A\subset\pr^4$, and we get the three possibilities given in the statement.
\end{prg}
\begin{proof}[Proof of Th.~\ref{B}]
\begin{prg}\label{logFano}
  We follow \cite[proof of Th.~6.1]{fibrations}.
  We assume that $\rho_X\geq 7$, so that $\rho_Y\geq 5$. Since $X$ is not a product of surfaces,
  Th.~\ref{deltageq4} and \ref{delta=3}  imply that $X$ has Leschetz defect $\delta_X\leq 2$.
  
Consider a $K$-negative resolution of $f_{\scriptscriptstyle X}$ (see Lemma \ref{Kneg}):
$$\xymatrix{
  X\ar@{-->}[r]_{\xi}\ar@{-->}@/^1pc/[rr]^{f_{\scriptscriptstyle X}}&{\w{X}}\ar[r]_{f}&Y.
}$$
Then $Y$ can have at most isolated, locally factorial, canonical singularities, contained in the images of the $2$-dimensional fibers of $f$ (Th.~\ref{special}$(a)$). Moreover $Y$ is log Fano, hence $-K_Y$ is big, and $-K_Y\cdot\NE(g)>0$ for every elementary contraction of fiber type $g\colon Y\to Y_0$ (see \cite[Th.~11.4.19]{lazII}).
\end{prg}
\begin{prg}\label{discrf}
By Th.~\ref{special} and \cite[Lemma 4.10]{fibrations}, there exists a unique prime divisor $B\subset Y$ such that ${f}^*B$ is reducible, and ${f}^*B=E_1+E_2$ with $E_i\subset\w{X}$ fixed prime divisors of type $(3,2)$ (see Th.-Def.~\ref{fixed}).

As in the proof of Lemma \ref{rhodelta} we see that
$(E_1\cup E_2)\cap\ell =\emptyset$ for every exceptional line $\ell\subset\w{X}$, therefore $E_1\cup E_2\subset\dom(\xi^{-1})$. We denote by $E_1',E_2'\subset X$ the transforms of $E_1,E_2\subset\w{X}$, so that $E_1'\cup E_2'\subset\dom(\xi)$.

We also note that $B$ is a connected component of the discriminant of $f$, and that $B$ is smooth outside (possibly) the images of the $2$-dimensional fibers of ${f}$, by Lemma \ref{conicbdl}.
\end{prg}
\begin{lemma}\label{rays}
  The cone $\NE(f)$ has two extremal rays, both of type $(3,2)$, with exceptional divisors $E_1$ and $E_2$. Let $\tilde\alpha\colon\w{X}\to\w{W}$ be the elementary contraction with exceptional divisor $E_1$. We have a diagram:
  \stepcounter{thm}
  \begin{equation}\label{diagramrays}
  \xymatrix{
    X\ar@{-->}[r]^{\xi}\ar[d]_{\alpha} & {\w{X}} \ar[d]_{\tilde{\alpha}}\ar[dr]^f&\\
    W\ar@{-->}[r]^{\xi_{\s W}}&{\w{W}}\ar[r]^{\pi}&Y
  }
  \end{equation}
  where $\alpha\colon X\to W$ is a divisorial elementary contraction of type $(3,2)$ with $\Exc(\alpha)=E_1'$, and $\xi_{\s W}\colon W\dasharrow\w{W}$ is a SQM.
Finally $W$ and $\w{W}$ are locally factorial and have at most nodes, at the images of some $2$-dimensional fibers of $\alpha$ and $\tilde{\alpha}$; moreover $W$ is Fano.  
\end{lemma}
We set $S:=\alpha(E_1')\subset W$. Then $S\subset\dom(\xi_{\s W})$, and with a sligh abuse of notation we still denote by $S$ its transform in $\w{W}$, which is $\tilde\alpha(E_1)$. Note that $\pi(S)=B$.
\begin{proof}
  Since $E_1'\subset X$ is a fixed prime divisor of type $(3,2)$, by Th.-Def.~\ref{fixed} there exists $\alpha\colon X\to W$ with the properties above, and $W$ is Fano. Moreover
  $E_1'$ does not contain exceptional planes by \cite[Rem.~2.17(2)]{blowup}, and neither does $E_1$, as they are contained in the open subsets where $\xi$ is an isomorphism.

  We know by Th.~\ref{special}$(d)$ that $f$ contracts a curve $e_1$ with $E_1\cdot e_1<0$, thus $\NE(f)$ has an $E_1$-negative extremal ray $R$; moreover $R$ is $K$-negative, because $f$ is. Since $\Lo(R)\subseteq E_1$ and $E_1$ does not contain exceptional planes, $R$ cannot be small (see Th.~\ref{kawsm}), thus it is divisorial with $\Lo(R)=E_1$. Since $\dim f(E_1)=2$, also the image of $E_1$ under the contraction of $R$ must be a surface, and 
$R$ is of type $(3,2)$. The same argument holds for $E_2$, and the rest of the statement follows from Th.~\ref{32}.
\end{proof}
\begin{lemma}\label{small}
  Let $g\colon Y\to Y_0$ be a birational contraction with $\dim\Exc(g)=1$.
Then $B\cap\Exc(g)=\emptyset$ and $\Exc(g)$ is the disjoint union of smooth rational curves $C$ contained in $Y_{\reg}$, with normal bundle $\ol_{\pr^1}(-1)^{\oplus 2}$; in particular
 $\NE(g)\subset K_Y^{\perp}$. 
Moreover  $f^{-1}(C)\cong\mathbb{F}_1$, and the $(-1)$-curve $\ell\subset f^{-1}(C)$ is an exceptional line in $\w{X}$. 
\end{lemma}
\begin{proof}
  The same proof as the one of \cite[Lemma 4.5]{eff} applies, with the only difference that, in the notation of \cite[Lemma 4.5]{eff}, $\dim\N(\w{U}/U)$ could be bigger than $2$. We take $\tau$ to be any extremal ray of $\NE(\w{U}/U)$ not contained in $\NE(f_{|\w{U}})$. 
  Finally we have
$B\cap\Exc(g)=\emptyset$ because $\ell\cap (E_1\cup E_2)=\emptyset$  (see \ref{discrf}).
\end{proof}
\begin{lemma}\label{divisorial}
  Let $g\colon Y\to Y_0$ be a divisorial elementary contraction, and set $G:=\Exc(g)$. Then
   $g$ is the blow-up of a smooth point $p\in Y_0$, and
  $D:={f}^{*}G$ is a fixed prime divisor, not of type $(3,2)$.
\end{lemma}
\begin{proof}
Since $g$ is elementary and $g(G)$ is either a point or an irreducible curve, we have $\dim \N(g(G),Y_0)\leq 1$, hence
$\dim\N(G,Y)\leq 2$ and $\dim\N({f}^{-1}(G),\w{X})\leq 4$ (see Rem.~\ref{linalg}). If there is a component $D$ of ${f}^{-1}(G)$ which is a fixed prime divisor of type $(3,2)$, let $D_{\s X}\subset X$ be its transform. Then $\dim\N(D_{\s X},X)=\dim\N(D,\w{X})\leq 4$  by Lemma \ref{dim32},
while $\delta_X\leq 2$ and $\rho_X\geq 7$ (see \ref{logFano}), a contradiction. Thus $G\neq B$
(recall that $f^*B=E_1+E_2$ with $E_i$ of type $(3,2)$, see \ref{discrf}) and $D:={f}^{-1}(G)$ is a prime divisor; moreover $D$ is fixed, because $G$ is, and not of type $(3,2)$.

We show that $g$ is of type $(2,0)$. By contradiction, suppose that $g$ is of type $(2,1)$. As in \cite[proof of Lemma 4.6]{eff} we show that there is an open subset $\w{U}\subset\w{X}$ such that $D\cap\w{U}$ is covered by curves of anticanonical degree $1$. By Lemma \ref{SQMFano}$(c)$, $D_{\s X}$ still has a nonempty open subset covered by curves of anticanonical degree $1$; this implies that $D_{\s X}$ and $D$ are of type $(3,2)$ by  \cite[Lemma 2.18]{blowup}, a contradiction.

Thus $g$ is of type $(2,0)$; set $p:=g(G)\in Y_0$.
As in \cite[proof of Lemma 4.6]{eff}
we have a diagram:
\stepcounter{thm}
\begin{equation}\label{diagram0}
\xymatrix{{\w{X}}\ar@{-->}[r]^h\ar[d]_f&{\wi{X}}\ar[r]^k&{\w{X}_1}\ar[dl]^{f_1}\\
  Y\ar[r]^g&{Y_0}&
}
\end{equation}
where $h$ is a sequence of $D$-negative flips relative to $g\circ f$, and $k$ is a divisorial elementary contraction  with $\Exc(k)$ the transform of $D$, so that $k(\Exc(k))$ is contained in the fiber $f_1^{-1}(p)$.

\medskip

We show that $\dim f_1^{-1}(p)=1$. If $D$ is of type $(3,0)^{\sm}$ or $(3,1)^{\sm}$, this is shown in
\cite[proof of Lemma 4.6, Step 2]{eff}, so we can assume that $D$ is of type $(3,0)^Q$. Moreover, 
if $G\cap B\neq\emptyset$, the statement is shown in \cite[6.4.3 -- 6.4.6]{fibrations}, thus we can  also assume that $G\cap B=\emptyset$, so that $p\not\in g(B)$, and $E_1\cap D=\emptyset$.
This implies that $E_1\subset\dom(h)$; 
let $\wi{E}_1\subset\wi{X}$ be its transform.
Then $\wi{E}_1$ is disjoint from $\Exc(k)$; moreover $\wi{E}_1$ is covered by curves of anticanonical degree one, therefore it is disjoint from all exceptional lines of $\wi{X}$ (see Lemma \ref{SQMFano}$(c)$). We also have
 $f_1(k(\wi{E}_1))=g(B)$, thus $\wi{E}_1\cap (f_1\circ k)^{-1}(p)=\emptyset$.

Suppose by contradiction that $\dim f_1^{-1}(p)=2$.
We proceed as in \cite[proof of Lemma 4.6, Steps 3 -- 6]{eff}, with the difference that in our setting $f_1$ is not elementary, and $\dim\NE(f_1\circ k)=3$.

As in \cite{eff} we see that
$f_1\circ k$ is not $K$-negative, 
so there is an extremal ray $R$ of $\NE(f_1\circ k)$ such that $-K_{\wi{X}}\cdot R\leq 0$; moreover $\NE(f_1\circ k)$ also contains the extremal ray $\NE(k)$.
We have $\wi{E}_1\cdot R=\wi{E}_1\cdot\NE(k)=0$; on the other hand $\wi{E}_1$ is not trivial on the whole cone $\NE(f_1\circ k)$,
otherwise $\wi{E}_1$ would be the pullback of a divisor from $Y_0$, but $(f_1\circ k)^{-1}(g(B))$ has two irreducible components, $\wi{E}_1$ and the transform of $E_2$.
Thus $\wi{E}_1^{\perp}\cap\NE(f_1\circ k)=R+\NE(k)$, and $\wi{E}_1^{\perp}$ also contains the classes of all exceptional lines of $\wi{X}$. Then we can work within this $2$-dimensional cone and as in \cite{eff} show that $h$ is just a $K$-negative flip and $\dim f_1^{-1}(p)=1$. This concludes the proof that $\dim f_1^{-1}(p)=1$.

\bigskip

Now as in \cite[6.4.7, 6.4.8]{fibrations} we see that $f_1$ is $K$-negative, that $Y_0$ is smooth at $p$, and finally that $g$ is just the blow-up of $p$.
\end{proof}
\begin{remark}
The assumption that $\rho_X\geq 7$ is essential in the above proof. Consider for instance $S=\Bl_{2\,\pts}\pr^2$, $X=S\times S$, $Y=\pr^1\times S$, and $f\colon X\to Y$ the natural product map given by a conic bundle $S\to\pr^1$. Then $Y$ has a divisorial  elementary contraction $k\colon Y\to \pr^1\times\mathbb{F}_1$ which is the blow-up of a smooth curve $\pr^1\times\{\pt\}$.
\end{remark}  
\begin{prg}\label{weakFano}
 $Y$ is weak Fano.

Indeed since $Y$ is log Fano, $\NE(Y)$ is closed and every one-dimensional face has the form $\NE(g)$ for some elementary contraction $g$ of $Y$. By \ref{logFano} and Lemmas \ref{small} and \ref{divisorial}, we always have $-K_Y\cdot\NE(g)\geq 0$, thus $-K_Y$ is nef and big.
\end{prg}
\begin{prg}\label{Freccia}
Let $Y\dasharrow Y'$ be a SQM. Then the composition $X\dasharrow Y'$ is again a special rational contraction with $\rho_X-\rho_{Y'}=2$ (see \cite[Rem.~2.8]{eff}), so all the previous steps apply to $Y'$ as well. As in \cite[p.\ 622]{eff}, using Lemmas \ref{small} and \ref{divisorial} one shows that if $E\subset Y$ is a fixed prime divisor, then $E$ can contain at most finitely many curves of anticanonical degree zero.
\end{prg}  
\begin{prg}\label{properties}
  Let $\ph\colon Y\to Z$ be the anticanonical map.
  If $Y$ is not Fano, then $\ph$ is small, $\Exc(\ph)\subset Y_{\reg}$, and  $B\cap\Exc(\ph)=\emptyset$. Moreover if $C$ is a 
  connected component of $\Exc(\ph)$, then $C\cong\pr^1$
  with $\ma{N}_{C/Y}\cong\ol(-1)^{\oplus 2}$, 
  $f^{-1}(C)\cong\mathbb{F}_1$, and the $(-1)$-curve $\ell\subset\mathbb{F}_1$ is an exceptional line in $\w{X}$.

   Indeed by \ref{Freccia} $\ph$ is generically finite on every fixed prime divisor of $Y$, thus it is small, and the statement follows from Lemma \ref{small}.
\end{prg}
\begin{lemma}\label{milano}
  Let $k\colon Y\to Y'$ be a blow-up of distinct smooth points. Then
$Y'$ is weak Fano and has the same singularities as $Y$, and the
  following hold:
  \begin{enumerate}[$(a)$]
     \item every small elementary contraction of $Y'$ is $K$-trivial;
  \item if $Y'$ is not Fano, then the anticanonical map $\ph'\colon Y'\to Z'$ is small, and $\Exc(\ph')$ is contained in $(Y')_{\reg}$ and   does not contain any point blown-up by $k$. Moreover if $C$ is a 
  connected component of $\Exc(\ph')$, then $C\cong\pr^1$
  with $\ma{N}_{C/Y'}\cong\ol(-1)^{\oplus 2}$;
    \item  every  divisorial elementary contraction of $Y'$ is the blow-up of a smooth point, with exceptional divisor not containing any point blown-up by $k$.
      \end{enumerate}
\end{lemma}
\begin{proof}
 We note that if $p\in Y'$ is a point blown-up by $k$, and $C\subset Y'$ is an irreducible curve containing $p$, then the transform $\w{C}\subset Y$ of $C$ satisfies $-K_Y\cdot\w{C}\geq 0$ and $\Exc(k)\cdot \w{C}\geq 1$, which implies that $-K_{Y'}\cdot C\geq 2$.  This shows that $Y'$ is weak Fano. Moreover $Y'$ cannot have $K$-negative small contractions  by Th.~\ref{gdn}; this gives $(a)$.
  
Suppose that $Y'$ is not Fano. By what precedes, the points blown-up by $k$ cannot lie in $\Exc(\ph')$, thus $\Exc(\ph')$ is contained in the open subset where $k$ is an isomorphism, and $k^{-1}(\Exc(\ph'))\subseteq\Exc(\ph)$. Hence $(b)$ follows from the analogous property of $\ph$ (see \ref{properties}).

 Finally let $g\colon Y'\to Y'_0$ be a divisorial elementary contraction. It is shown in \cite[p.~623]{eff} that
 $\Exc(g)\cap k(\Exc(k))=\emptyset$; then
 by Lemma \ref{divisorial} $g$ must be the blow-up of a smooth point in $Y_0'$, and  we have $(c)$.
\end{proof}  
\begin{lemma}\label{blowups}
Up to flops  there exists  a blow-up $k\colon Y\to Y_0$ of $r$ distinct smooth points $p_1,\dotsc,p_r\in Y_0$ such that $\rho_{Y_0}\leq 2$, and if $\rho_{Y_0}=2$, then $Y_0$ has two distinct elementary rational contractions of fiber type.
\end{lemma}
 We denote by $G_i\subset Y$ the exceptional divisor over $p_i\in Y_0$.
\begin{proof}
As in \cite[p.~622]{eff} we consider all divisorial extremal rays of $\NE(Y)$ and get a map
$k\colon Y\to Y_0$
which is the blow-up of $r$ distinct smooth points.  Moreover $Y_0$ is weak Fano and has the same singularities as $Y$.

Let $\psi\colon Y_0\dasharrow Y_0'$ be a flop. The composition $\psi\circ k\colon Y\dasharrow Y_0'$ is a rational contraction (see \cite[Rem.~2.8]{eff}), and  there is a SQM $Y'\dasharrow Y$ such that the composition $k'\colon Y'\to Y_0'$ is regular.
By Lemma \ref{milano}
the $r$ points blown-up by $k$ lie in $\dom(\psi)$. Thus $k'$ is again the blow-up of $r$ distinct smooth points, and we may replace 
$k\colon Y\to Y_0$ with $k'\colon Y'\to Y_0'$ if needed (see \ref{Freccia}). Iterating the reasoning, the same holds for any SQM $Y_0\dasharrow Y_0'$.

Suppose that there is a divisorial  elementary rational contraction $\pi\colon Y_0\dasharrow Y_0'$. Up to replacing $Y_0$ and $Y$ with a SQM, we can assume that $\pi$ is regular.
Then
  $\pi$ must be the blow-up of a smooth point in $Y_0'$ by Lemma \ref{milano}$(c)$, and we replace $Y_0$ with $Y_0'$. 

 In this way, in a finite number of steps, up to flops and up to increasing  the number $r$ of blown-up points,
  we reduce to the case where $Y_0$ has no divisorial elementary rational contraction, hence $Y_0$ has no fixed prime divisors (see \cite[Rem.~2.19]{eff}).

  Suppose now that there is an elementary rational contraction of fiber type $\pi\colon Y_0\dasharrow S$ with $\dim S=2$; again up to flops we can assume that $\pi$ is regular.
  Then $S$ cannot have divisorial elementary  contractions: indeed if $S\to S_1$ were such a contraction with exceptional divisor an irreducible curve $C$, then $\pi^{-1}(C)$ should be a fixed divisor in $Y_0$, against our reductions.
 Moreover $S$ is smooth (see for instance \cite[Lemma 5.5]{ou}) and rational, 
 hence either $S\cong\pr^2$, or $S\cong\pr^1\times\pr^1$.

 If $S\cong\pr^1\times\pr^1$ (so that $\rho_{Y_0}=3$), let us consider the blow-up  of the first point $Y_{1}\to Y_0$ (recall from \ref{logFano} that $\rho_Y\geq 5$). The composition $Y_{1}\to S$ is not equidimensional, and by \cite[Prop.~2.13]{fibrations} it factors as
 $$\xymatrix{{Y_{1}}\ar@{-->}[d]\ar[r]^{\Bl_{\pt}}&{Y_0}\ar[d]^{\pi}\\
{S'}\ar[r]&{S=\pr^1\times\pr^1}
   }$$
   where $Y_{1}\dasharrow S'$ is an elementary rational contraction and
$S'\to S$ is the blow-up of a smooth point, thus
$S'\cong\Bl_{2\pts}\pr^2$.

Consider a $(-1)$-curve of $S'$ contracted by $S'\to\pr^2$, and its pullback $G$ in $Y_{1}$. This is a fixed prime divisor, and up to 
to replacing $Y_{1}$ with a SQM, we can assume that $G=\Exc(g)$ for some divisorial elementary contraction $g\colon Y_{1}\to Y_0'$; then $g$ must be the blow-up of a smooth point, and there is an elementary rational contraction $Y_0'\dasharrow\mathbb{F}_1$.
$$\xymatrix{{Y_{1}}\ar@{-->}[d]\ar[r]^{g}&{Y_0'}\ar@{-->}[d]\\
{S'}\ar[r]&{\mathbb{F}_1}
   }$$
 Then as before we replace $Y_{1}\to Y_0$ with $g\colon Y_{1}\to Y_0'$; now using the elementary rational contraction $Y_0\dasharrow\mathbb{F}_1$, 
we   blow-down one more point and get $Y_{0}''$ with $\rho_{Y_0''}=2$.

In the end the only possible non-small elementary rational contractions of $Y_0$ are  $Y_0\to\{pt\}$, $Y_0\dasharrow\pr^1$,  or $Y_0\dasharrow\pr^2$. Since $Y_0$ is a Mori dream space, it has at least $\rho_{Y_0}$  non-small elementary rational contractions; we conclude that 
$\rho_{Y_0}\leq 2$ and we get the statement.
\end{proof}
\begin{prg}\label{degree}
 Since $\rho_Y\geq 5$ (see \ref{logFano}) and $-K_Y$ is nef and big (see \ref{weakFano}), we have $r=\rho_Y-\rho_{Y_0}\geq 3$ and
$0<-K_Y^3=-K^3_{Y_0}-8r$,
thus $-K^3_{Y_0}>8r\geq 24$.
\end{prg}
\begin{lemma}\label{ou2}
If $\rho_{Y_0}=2$ and $Y_0$ is Fano, then $Y_0$ is smooth and rational; more precisely either $Y_0\cong\pr_{\pr^2}(T_{\pr^2})$, or $Y_0$ is isomorphic to a divisor of degree $(1,2)$ in $\pr^2\times\pr^2$.
\end{lemma}
\begin{proof}
  Since $Y_0$ is Fano, Lemma \ref{milano}$(a)$ implies that $Y_0$ has no small contraction. 
  By Lemma \ref{blowups} $Y_0$ has two distinct elementary rational contractions of fiber type, hence they must be regular. By
 \ref{degree} we can apply Lemma \ref{ou} to $Y_0$; either we get the statement, or 
 $Y_0\cong\pr^2\times\pr^1$. However the blow-up of $\pr^2\times\pr^1$ at a point has a
  divisorial elementary contraction of type $(2,1)$, which is excluded by Lemma \ref{milano}$(c)$.
\end{proof}
We recall that $\ph\colon Y\to Z$ is the anticanonical map of $Y$.
    \begin{lemma}\label{rhoZ}
We have $\rho_Z=1$.
\end{lemma}
\begin{proof}
  We have
 $\rho_Y-\rho_Z=\dim\NE(\ph)=\dim(\NE(Y)\cap K_Y^{\perp})$, thus
 we have
 to exhibit $\rho_Y-1$ curves in $Y$, with anticanonical degree zero, whose classes 
in $\N(Y)$ are linearly independent.

  Consider one of the points $p_i\in Y_0$ blown-up by $k$, and set $Y_{i}:=\Bl_{p_i}Y_0\stackrel{\sigma}{\to} Y_0$. Let $G\subset Y_{i}$ be the exceptional divisor, and consider an extremal ray $R_i$ of $\NE(Y_{i})$ such that $G\cdot R_i>0$. Then the contraction of $R_i$ must be finite on $G\cong\pr^2$, and has fibers of dimension at most one. By Lemma \ref{milano}$(c)$ we conclude that the contraction of $R_i$ is either small, or of fiber type.

  In the small case, let $\Gamma_i'\subset Y_{i}$ be a curve with class in $R_i$, and $\Gamma_i\subset Y$ its transform. Then $\Gamma_i'$ is contained in the open subset here $Y\to Y_i$ is an isomorphism, hence
   $K_Y\cdot\Gamma_i=K_{Y_i}\cdot\Gamma'_i=0$, $\Gamma_i\cdot G_i>0$, and $\Gamma_i\cdot G_j=0$ for every $j=1,\dotsc,r$, $j\neq i$.

  Suppose that the contraction $\psi\colon Y_i\to S$ of $R_i$ is of fiber type; note that
$\psi(G)=S$, thus $\rho_S=1$, $\rho_{Y_i}=2$, and
  $\rho_{Y_0}=1$. Moreover $\psi$ is $K$-negative and a conic bundle, see Prop.~\ref{conicsing}.
 Let us consider another point $p_j\in Y_{i}$ blown-up by $k$, $j\neq i$, and note that $p_j\not\in G$. Since every irreducible curve of $Y_i$ containing $p_j$ must have anticanonical degree $\geq 2$, the fiber $F_j$ of $\psi$ containing $p_j$ is a smooth rational curve with $-K_{Y_{i}}\cdot F_j=2$, and its transform $\Gamma_{ij}\subset Y$ has $K_Y\cdot\Gamma_{ij}=0$, $G_i\cdot\Gamma_{ij}>0$, $G_j\cdot\Gamma_{ij}>0$, and $G_h\cdot\Gamma_{ij}=0$ for every $h\in\{1,\dotsc,r\}\smallsetminus\{i,j\}$.

 \medskip

 If $R_i$ is small for every $i=1,\dotsc,r$, we get $\Gamma_1,\dotsc,\Gamma_r\subset Y$ with linearly independent classes in $K_Y^{\perp}$. If $\rho_{Y_0}=1$, then $r=\rho_Y-1$ and we are done. If $\rho_{Y_0}=2$ and $Y_0$ is not Fano, let $\Gamma_0\subset Y$ be the transform of an irreducible curve of anticanonical degree zero in $Y_0$. Then $G_i\cdot\Gamma_0=0$ for every $i=1,\dotsc,r$ (see Lemma \ref{milano}$(b)$), so that $\Gamma_0,\Gamma_1,\dotsc,\Gamma_r\subset Y$ yield again $\rho_Y-1$ linearly independent classes in $K_Y^{\perp}$.

  If $\rho_{Y_0}=2$ and $Y_0$ is Fano, by Lemma \ref{ou2} we know that $Y_0$ has two conic bundle structures. Let $\Gamma_1',\Gamma_1''\subset Y$ be the transforms of the fibers through $p_1$ of the two conic bundles; as before these fibers must be smooth. Then $K_Y\cdot  \Gamma_1'=K_Y\cdot  \Gamma_1''= G_i\cdot 
  \Gamma_1'= G_i\cdot \Gamma_1''=0$ for every $i=2,\dotsc,r$. Moreover the classes $[\Gamma_1'],[\Gamma_1'']$ are linearly independent, because their pushforwards in $\N(Y_0)$ are. Thus we have the $\rho_Y-1$ curves
  $\Gamma_1',\Gamma_1'',\Gamma_2,\dotsc,\Gamma_r$, and we are done.

  \medskip

  Suppose now that $R_1$ is not small; in particular $\rho_{Y_0}=1$ and $r=\rho_Y-1$. We get curves $\Gamma_{12},\dotsc,\Gamma_{1r}$ in $Y$ with $K_Y\cdot \Gamma_{1j}=0$, $G_1\cdot\Gamma_{1j}>0$, $G_j\cdot\Gamma_{1j}>0$, and $G_i\cdot \Gamma_{1j}=0$ for every $i,j=2,\dotsc,r$, $i\neq j$.

  Recall that $r\geq 3$ (see \ref{degree}). If $R_2$ is small, then we get a curve $\Gamma_2\subset Y$ such that $K_Y\cdot\Gamma_2=0$, $G_2\cdot \Gamma_2>0$, and $G_i\cdot\Gamma_2=0$ for every $i\neq 2$; the classes of  $\Gamma_2,\Gamma_{12},\dotsc,\Gamma_{1r}$ are linearly independent.

  If instead $R_2$ is not small, we get a curve $\Gamma_{23}\subset Y$ with $K_Y\cdot\Gamma_{23}=0$, $G_2\cdot\Gamma_{23}>0$, $G_3\cdot\Gamma_{23}>0$, and $G_i\cdot \Gamma_{23}=0$ for every $i\in\{1,4,\dotsc,r\}$. Then one can check that  the classes of $\Gamma_{23},\Gamma_{12},\dotsc,\Gamma_{1r}$ are again linearly independent. Indeed given a relation $a\Gamma_{23}+b\Gamma_{12}+c\Gamma_{13}\equiv 0$, intersecting with $G_i$ for $i=1,2,3$ and analysing the signs of $a,b,c$, one gets $a=b=c=0$.
\end{proof}
\begin{lemma}\label{index}
  Set $B_0:=k(B)\subset Y_0$. Then $p_1,\dotsc,p_r\in B_0$, and for some $\lambda\in\Q_{>0}$ we have
   $-K_Y=\lambda B$ and $-K_{Y_0}=\lambda B_0$.
\end{lemma}  
\begin{proof}
By Lemma \ref{rhoZ} we have 
$\dim\NE(\ph)=\rho_Y-1$. On the other hand $B$ is disjoint from $\Exc(\ph)$ by \ref{properties}, thus $B^{\perp}\supset\NE(\ph)$,
$B^{\perp}=K_Y^{\perp}$, and finally $-K_Y=\lambda B$ for some $\lambda\in\Q_{>0}$. This also implies that
$-K_{Y_0}=\lambda B_0$ in $Y_0$.

Let $i\in\{1,\dotsc,r\}$ and let $\Gamma\subset G_i$ be a curve. Then $-K_Y\cdot \Gamma>0$, thus $B\cdot \Gamma>0$ and $p_i\in B_0$.
\end{proof}
\begin{lemma}\label{bijection}
  There is a bijection between the set of
exceptional lines in $\w{X}$, and the set of 
curves of anticanonical degree zero in $Y$, via $\ell\mapsto f(\ell)$.
\end{lemma}
\begin{proof}
  Let $\ell\subset\w{X}$ be an exceptional line, and note that $f(\ell)$ is a curve, because $K_{\w{X}}\cdot\ell=1$ while $f$ is $K$-negative (see \ref{logFano}). We have $\ell\cap (E_1\cup E_2)=\emptyset$ (see \ref{discrf}), thus $f(\ell)\cap B=\emptyset$, and $B\cdot f(\ell)=0$. By Lemma \ref{index} this is equivalent to $-K_Y\cdot f(\ell)=0$. The converse is given by \ref{properties}.
\end{proof}
\begin{prg}\label{regular}
  The composition $\ph\circ f_{\scriptscriptstyle X}\colon X\dasharrow Z$ is a regular contraction,
 and it factors through $\alpha\colon X\to W$ (see Lemma \ref{rays}).
 $$\xymatrix{X\ar@/_1pc/[dd]_{\alpha}\ar@{-->}[r]^{\xi}\ar@{-->}[dr]^{f_{\scriptscriptstyle X}}\ar[d]&
   {\w{X}}\ar[d]^>>>f\ar@/^1pc/[dd]^{\tilde\alpha}\\
   Z& Y\ar[l]^{\ph}\\
   W\ar[u]\ar@{-->}[r]^{\xi_{\s W}}&{\w{W}}\ar[u]^{\pi}}$$

 Indeed, by Lemma \ref{bijection},  $\ph\circ f\colon\w{X}\to Z$ contracts all exceptional lines of $\w{X}$, therefore $\ph\circ f\circ\xi=\ph\circ f_{\s X}$ is regular (see Lemma \ref{SQMFano}$(a)$).
Similarly, by Lemma \ref{rays} the indeterminacy locus $\w{W}\smallsetminus\dom(\xi_{\s W}^{-1})$ is isomorphic, via $\tilde\alpha$, to $\w{X}\smallsetminus\dom(\xi^{-1})$, and $\ph\circ\pi\colon \w{W}\to Z$ contracts to points this locus, therefore $\ph\circ\pi\circ\xi_{\s W}\colon W\to Z$ is regular.
\end{prg}  
\begin{lemma}\label{discriminant}
  The discriminant of $f$ is $B$, and $f$ has no intrinsic discriminant (see \ref{setting}).
\end{lemma}
\begin{proof}
  Suppose by contradiction that the discriminant of $f$ has an irreducible component $\Delta_0$ besides $B$. Then $\Delta_0\cap B=\emptyset$ (see \ref{discrf}),
  thus $B\cdot C=0$ for every curve $C\subset \Delta_0$. By Lemma 
\ref{index} this gives $K_Y\cdot C=0$ for every curve $C\subset \Delta_0$, namely
   $\Delta_0$ is contained in $\Exc(\ph)$, but $\ph$ is small (see \ref{properties}).
\end{proof}
\begin{prg}\label{nodes}
  Let $y_0\in Y$ be a singular point and $F:=f^{-1}(y_0)$. Then $y_0$ is a node, $y_0\not\in B$, $\dim\N(F,\w{X})=1$, and $F\cong\pr^2\bullet\pr^2$.

  This follows from Th.~\ref{2dimfibers} and Lemma \ref{discriminant}, because 
the fiber $F$ is $2$-dimensional (see \ref{logFano}), and $\Delta_{\intr}=\emptyset$.
  \end{prg}
  \begin{lemma}\label{sm}
    If $Y$ is singular, then it is not rational.
\end{lemma}
\begin{proof}
  Set 
  $\Sing(Y):=\{y_1,\dotsc,y_m\}$ with $m\geq 1$, and 
$F_i:=f^{-1}(y_i)$.
By  \ref{nodes} $y_i$ is a node, $y_i\not\in B$, $\dim\N(F_i,\w{X})=1$, and $F_i\cong\pr^2\bullet\pr^2$, for $i=1,\dotsc,m$.

Let $z_1,\dotsc,z_r\in Y$ be the smooth points such that $\dim f^{-1}(z_i)=2$.
Then $f$ is a conic bundle over $U:=Y\smallsetminus\{y_1,\dotsc,y_m,z_1,\dotsc,z_r\}$, with discriminant $B\cap U$ (Lemma \ref{discriminant}).
  Let us consider the factorization of $f$ in elementary steps as in Lemma \ref{rays}:
   $$\xymatrix{
{\w{X}}\ar[r]_{\tilde{\alpha}}\ar@/^1pc/[rr]^{f}& {\w{W}}\ar[r]_{\pi} &{Y}
     }$$
and recall that $\tilde{\alpha}$ is an elementary contraction of type $(3,2)$ with exceptional divisor $E_1$, and $\w{W}$ is locally factorial.
Then
 $\pi$ is smooth with fiber $\pr^1$ over $U$.

 Since $y_1\not\in B$ and $f(E_1)=B$, we have $\pi^{-1}(y_1)\cong F_1 \cong \pr^2\bullet\pr^2$. Write
 $\pi^{-1}(y_1) =L_1\cup L_2\subset\w{W}$ with $L_i\cong\pr^2$, and let $C_{L_i}$ be a line in $L_i$.
We have $\N(\pi^{-1}(y_1),\w{W})=\tilde\alpha_*(\N(F_1,\w{X}))$, hence
  $\dim\N(\pi^{-1}(y_1),\w{W})=1$. Moreover $-K_{\w{W}}\cdot C_{L_1}=-K_{\w{W}}\cdot C_{L_2}=1$, and we conclude that $C_{L_1}\equiv C_{L_2}$.

If $F_0\cong\pr^1$ is a general fiber of $\pi$,  then $F_0\equiv C_{L_1}+C_{L_2}\equiv 2C_{L_1}$; in particular no divisor in $\w{W}$ can have intersection $1$ with $F_0$, and
 $\pi_{|\pi^{-1}(U)}\colon \pi^{-1}(U)\to U$ cannot be the projectivization of a vector bundle on $U$. Therefore $\pi_{|\pi^{-1}(U)}$ defines a nonzero class in the (cohomological) Brauer group $\Br(U)$ (see for instance 
\cite[\S 6.2, 6.3]{debarrerationality}).

The following argument is from \cite[proof of Th.~6.7]{debarrerationality}.
Let $\wi{Y}\to Y$ be the resolution of the nodes; then $U\subset \wi{Y}$ and $E:=\wi{Y}\smallsetminus U$ is a disjoint union of $r$ points and $m$ smooth quadric surfaces. If $H^2(\wi{Y},\ol_{\wi{Y}})\neq 0$ then $\wi{Y}$ and $Y$ are not rational, thus we can assume that $H^2(\wi{Y},\ol_{\wi{Y}})= 0$, therefore $c_1\colon \Pic(\wi{Y})\to H^2(\wi{Y},\Z)$ is surjective.

The Thom-Gysin exact sequence
$$H^2(\wi{Y},\Z)\la H^2(U,\Z)\la H^1(E,\Z)=0$$
yields that the restriction $H^2(\wi{Y},\Z)\to H^2(U,\Z)$ is surjective, hence the composition $\Pic(\wi{Y})\to H^2(U,\Z)$ is surjective, and also $c_1\colon \Pic(U)\to H^2(U,\Z)$ is. This implies that
$\Br(U)\cong\Tors H^3(U,\Z)$ \cite[Prop.~6.4]{debarrerationality}.

Again using the Thom-Gysin exact sequence
$$0= H^1(E,\Z)\la H^3(\wi{Y},\Z)\la H^3(U,\Z)\la H^2(E,\Z)\cong\Z^{2m}$$
we find that $\Tors H^3(\wi{Y},\Z)\cong \Tors H^3(U,Z)$ is nonzero. Then $\wi{Y}$ and $Y$ are not rational \cite[Prop.~6.1]{debarrerationality}.
\end{proof}
  \begin{lemma}\label{table}
 The $3$-folds   $Y$ and $Y_0$ are smooth and rational, and up to flops, $Y_0$ belongs to the following list.

{\em   $$\begin{array}{||c|c|c|c|c||}
\hline\hline
\text{N.} & Y_0   & -K_{Y_0}^3&  \rho_{Y_0} & \\
\hline\hline

  1 &     \pr^3 & 64 &  1 &  \text{Fano}     \\

       \hline

  2 &      \pr_{\pr^2}(T_{\pr^2}) & 48 & 2 & \text{Fano}   \\
   
  \hline

3 &    \text{linear section of }\Gr(2,5) \subset\pr^9& 40 & 1 & \text{Fano}   \\

        \hline
           
    4 &  \text{\em\cite[Th.~3.6(1)]{jahnkepeternell}} & 40 & 2 & \text{weak Fano} \\

 \hline

    5 &   \text{\cite[Th.~3.5(3)]{jahnkepeternell}} & 32 & 2 & \text{weak Fano}  \\     

       \hline

    6 &   \text{divisor of degree $(1,2)$ in }\pr^2\times\pr^2 & 30 & 2 & \text{Fano}  \\ 
       
\hline\hline
         \end{array}$$}       
     \end{lemma}

     \smallskip
     
\begin{proof}
  If $\rho_{Y_0}=2$ and $Y_0$ is Fano, the statement follows from Lemma \ref{ou2}, and we get N.~2 and 6.
If $\rho_{Y_0}=2$ and $Y_0$ is smooth and not Fano, 
then $Y_0$ has small anticanonical map by Lemma \ref{milano}$(b)$,
  and two elementary rational contractions of fiber type by Lemma \ref{blowups};   moreover $-K_{Y_0}^3>24$ (see \ref{degree}).
By Prop.~\ref{smoothwfrho2} 
we see that, up to flops, $Y_0$ is either
 \cite[Th.~3.5(3) or Th.~3.6(1)]{jahnkepeternell}, or $\pr_{\pr^1}(\ol\oplus\ol(1)^{\oplus 2})$.
However the blow-up of $\pr_{\pr^1}(\ol\oplus\ol(1)^{\oplus 2})$ at a point (not lying on the curve of anticanonical degree zero) has a divisorial  elementary contraction of type $(2,1)$, which is excluded by Lemma \ref{milano}$(c)$. Thus we get N.~4 and 5.

  Let us assume that either $\rho_{Y_0}=1$, or $\rho_{Y_0}=2$ and $Y_0$ is singular and not Fano.
  Consider the anticanonical map $\ph_0\colon Y_0\to Z_0$, and note that $Y_0$ is nodal by \ref{nodes}. By Lemma \ref{milano}$(b)$ $\ph_0$ contracts curves of anticanonical degree zero in $(Y_0)_{\reg}$ to nodes in $Z_0$, therefore $Z_0$ is a nodal Gorenstein Fano $3$-fold with $\rho_{Z_0}=1$ and $-K_{Z_0}^3=-K_{Y_0}^3> 24$ by \ref{degree}. Moreover if $\Cl(Z_0)$ is the  Weil divisor class group, we have $\rk\Cl(Z_0)=\rho_{Y_0}\in\{1,2\}$. 
  
Then we can apply \cite[Th.~1.1 and 1.2]{prokratII} to $Z_0$, and conclude that $Z_0$ is rational. Note that in the notation of \cite{prokratII}, $Z_0$ has genus $g:=\frac{1}{2}(-K_{Z_0}^3)+1\geq 14$, and if $Z_0$ has index two, then it has degree $d:=\frac{1}{8}(-K_{Z_0}^3)\geq 4$.

Hence $Y_0$ and $Y$ are rational too, so they are smooth by Lemma \ref{sm}, and we conclude that $\rho_{Y_0}=1$ and $Y_0$ is a smooth Fano $3$-fold.
Moreover $r=\rho_Y-1\geq 4$ and $-K_{Y_0}^3>8r\geq 32$ (see \ref{degree}). 
 By classification (see \cite[\S 12.2]{fanoEMS}) there are three possibilities for $Y_0$: $\pr^3$, a section of the Pl\"ucker embedding of the Grassmannian $\Gr(2,5)$ in $\pr^9$ by a codimension $3$ linear subspace, or
 a quadric. However 
the blow-up of a smooth $3$-dimensional quadric at a point has a divisorial  elementary contraction of type $(2,1)$, which is excluded by Lemma \ref{milano}$(c)$.  Thus we get N.~1 and 3.
\end{proof}
\begin{prg}\label{wings}
  We note that $k\circ f\colon\w{X}\to Y_0$ is $K$-negative.

  Indeed by Lemma \ref{SQMFano}$(b)$ it is enough to check that $k\circ f$ does not contract any exceptional line $\ell\subset\w{X}$.   By Lemma \ref{bijection} $f(\ell)\subset Y$ is a curve of anticanonical degree zero, thus $f(\ell)\not\subset G_i$ for every $i=1,\dotsc,r$, and $k(f(\ell))$ is a curve.
\end{prg}
\begin{lemma}\label{Xr}
There is a diagram:
\stepcounter{thm}
\begin{equation}\label{diagram}
\xymatrix{X\ar@{-->}[r]^{\xi}\ar@{-->}[dr]_{f_{\scriptscriptstyle X}}&{\w{X}}\ar[d]^f\ar@{-->}[r]^{\psi}&{\wi{X}}\ar[r]^{\sigma}&{X_0}\ar[dl]^{f_0}\\
  &Y\ar[r]^k&{Y_0}&
}\end{equation}
where $\psi$ is a  sequence of $K$-negative flips, $\wi{X}$ is smooth, 
$\sigma$ is birational and divisorial with pairwise disjoint exceptional divisors $D_1,\dotsc,D_r$, and $f_0$ is a contraction.
Moreover,  for each $i=1,\dotsc,r$, the divisor 
$D_i\subset \wi{X}$ is the transform of $f^*G_i\subset\w{X}$, $\sigma(D_i)\subset f_0^{-1}(p_i)$, and there are three possitibilities:
\begin{enumerate}[$\bullet$]
\item $D_i$ is of type $(3,0)^{\sm}$ and is the exceptional divisor of the blow-up of a smooth point of $X_0$;
\item $D_i$ is of type $(3,1)^{\sm}$ and is the exceptional divisor of the blow-up of a smooth curve contained in $(X_0)_{\reg}$;
\item $D_i$ is of type $(3,0)^Q$ and   is contracted to an isolated hypersurface singularity of $X_0$, terminal and locally factorial.
\end{enumerate}

In particular $X_0$ has at most  locally factorial, terminal, isolated hypersurface singularities at $\sigma(D_i)$ for $D_i$ of type $(3,0)^Q$.
\end{lemma}
\begin{proof}
We set $D_i:=f^*G_i\subset\w{X}$,  and with a slight abuse of notation we still denote by $D_i\subset\wi{X}$ its transform, for $i=1,\dotsc,r$.

By Lemma \ref{divisorial} $D_i$ is a fixed prime divisor, not of type $(3,2)$, for every $i=1,\dotsc,r$;
by Th.-Def.~\ref{fixed}, $D_i$ can be of type $(3,0)^{\sm}$, $(3,1)^{\sm}$, or $(3,0)^Q$. 
Moreover $D_1,\dotsc,D_r\subset\w{X}$ are pairwise disjoint, thus $D_i\cdot C_{D_j}=0$ when $i\neq j$. Using \cite[Lemma 5.29(2)]{blowup} we see that
$\Mov(\w{X})\cap\langle [D_1],\dotsc,[D_r]\rangle=\{0\}$. Also note that $k(f(D_i))=p_i$ for every $i=1,\dotsc,r$.

By running in $\w{X}$ a MMP for $D_1+\cdots+D_r$, relative to $k\circ f$,
we get a
 diagram as \eqref{diagram}, 
  where
$\psi$ is a sequence of $D_i$-negative flips for some $i$, 
$X_0$ is $\Q$-factorial, $\sigma$ is birational with exceptional divisors $D_1,\dotsc,D_r$, and $f_0$ is a contraction. 

Since  the MMP is relative to $k\circ f$ which is $K$-negative (see \ref{wings}), $\psi$ is a sequence of $K$-negative flips. Therefore $\w{X}\smallsetminus\dom(\psi)$ is a finite, disjoint union of exceptional planes (see \cite[Lemma 4.1(c)]{fibrations}), each contained in some $D_i$.
Since the $D_i$'s are disjoint in $\w{X}$, they stay disjoint also in $\wi{X}$. Then the description from Th.-Def.~\ref{fixed} holds for each one of them. Finally the description of the singularities of $X_0$ follows from \cite[Lemma 2.19]{blowup}.
\end{proof}
\begin{lemma}\label{onedim}
 Let $i\in\{1,\dotsc,r\}$. Every fiber of $f$ over $G_i$ is one-dimensional.
\end{lemma}
\begin{proof}
  Suppose by contradiction that $f$ has a $2$-dimensional fiber $F$ over $G_i$, so that $F\subset D_i$. Since $D_i=f^*G_i$, we have $D_i\cdot C=0$ for every curve $C\subset F$; in particular $F\cap\dom(\psi)\neq\emptyset$, as $\w{X}\smallsetminus \dom(\psi)$ is a finite union of exceptional planes $L$ such that either $D_i\cdot C_L< 0$, or $D_i\cap L=\emptyset$ (see the proof of Lemma \ref{Xr}).

 Let $\wi{F}\subset\wi{X}$ be the transform of $F$. Then  $\wi{F}\subset \Exc(\sigma)$, and $\Exc(\sigma)$ cannot contain exceptional lines, see \cite[Rem.~5.6]{blowup}. The indeterminacy locus of $\psi^{-1}$ is a finite union of exceptional lines, therefore $\dim( \wi{F}\cap (\wi{X}\smallsetminus\dom(\psi^{-1})))\leq 0$. Since  $\dim \sigma(\wi{F})\leq \dim\sigma(\Exc(\sigma))\leq 1$, there is an irreducible curve $C\subset \wi{F}$ such that $C\subset \dom(\psi^{-1})$ and $\sigma(C)=\{pt\}$. We get $\Exc(\sigma)\cdot C<0$, and $D_i\cdot C'<0$ where $C'\subset F$ is the transform of $C$, a contradiction because  $D_i\cdot C'=0$.
\end{proof}  
\begin{prg}\label{BG}
  We have $B\cap G_i\subset B_{\reg}$ for every $i\in\{1,\dotsc,r\}$. 

  This follows from
  Lemma \ref{onedim}, because $B$ is smooth outside (possibly) the images of the $2$-dimensional fibers of $f$ (see \ref{discrf}).
\end{prg}
Recall that $\w{X}$ contains the fixed prime divisors $D_1,\dotsc,D_r,E_1,E_2$, and that we have associated curves $C_{D_i}\subset D_i$, $C_{E_j}\subset E_j$, see Th.-Def.~\ref{fixed}.
 \begin{lemma}\label{X0Fano}
If $Y_0$ is Fano, then 
 $X_0$ is Fano; otherwise $X_0$  has a SQM which is Fano. Moreover $f_0$ is $K$-negative.
\end{lemma}
\begin{proof}
  We keep the same notation as in the proof of Lemma \ref{Xr}. Recall from Lemma \ref{SQMFano} that the indeterminacy locus of $\xi^{-1}$ is the finite, disjoint union of all exceptional lines in $\w{X}$. In particular 
$\w{X}\smallsetminus\dom(\psi)$ is disjoint from $\w{X}\smallsetminus\dom(\xi^{-1})$, because exceptional planes and exceptional lines cannot meet (see Lemma \ref{SQMFano}$(c)$). Therefore $\wi{X}$ contains two types of exceptional lines, those contained in $\dom(\psi^{-1})$, and those in the indeterminacy locus of $\psi^{-1}$. Each of these last ones has positive intersection with some $D_i$, because $\psi$ is a sequence of flips that are $D_i$-negative for some $i$.

  Suppose first that $Y_0$ is Fano. Then every curve of anticanonical degree zero in $Y$ must meet $\Exc(k)=G_1\cup\cdots\cup G_r$, and by Lemma \ref{bijection} this means that every exceptional line $\ell\subset\w{X}$ must meet some $D_i=f^*G_i$. On the other hand  $\ell$ cannot be contained in any $D_i$, because
  $k\circ f$ is $K$-negative (see \ref{wings}),   thus  $D_i\cdot\ell>0$ for some $i$.
  We conclude that for every exceptional line $\ell'\subset\wi{X}$ we have $D_i\cdot\ell'>0$ for some $i$.
  
Consider now $\sigma^*(-K_{X_0})=-K_{\wi{X}}+\sum_{i=1}^rm_iD_i$ with $m_i=2$ (respectively $m_i=3$) if $D_i$ is of type $(3,0)^Q$ or $(3,1)^{\sm}$ (respectively $(3,0)^{\sm}$).  
Using that $D_i\cdot C_{D_j}=0$ for every $i\neq j$, proceeding as in \cite[proof of Th.~3.15]{eff} one shows that $-K_{\wi{X}}+\sum_im_iD_i$ is nef and that $(-K_{\wi{X}}+\sum_im_iD_i)^{\perp}\cap\NE(\wi{X})=\NE(\sigma)$; this implies that
$X_0$ is Fano.

\medskip

If instead $Y_0$ is not Fano, then we have $\rho_{Y_0}=2$ and $Y_0$ contains some curves $C_j$ of anticanonical degree zero, $j=1,\dotsc,s$, with $C_1\equiv\cdots \equiv C_s$. The transforms $\w{C}_j\subset Y$ of these curves are precisely the curves of anticanonical degree zero in $Y$ that are disjoint from $\Exc(k)$, and in turn by Lemma \ref{bijection} these are images of the exceptional curves $\ell_j\subset\w{X}$ that are disjoint from $D_1,\dotsc,D_r$.

We also have $\ell_1\equiv\cdots\equiv\ell_s$. Indeed fix $j\in\{1,\dotsc,s\}$. We have $f_*\ell_1 =\w{C}_1\equiv\w{C}_j=f_*\ell_j$ and, by Lemma \ref{rays}, $\ker f_*$ is generated by the classes $[C_{E_1}],[C_{E_2}]$ with  $E_1\cdot C_{E_2}>0$ and $E_1\cdot\ell_1=E_1\cdot\ell_j=0$ (see \ref{discrf}). Then $\ell_1\equiv\ell_j+aC_{E_1}+bC_{E_2}$ with $a,b\in\Q$, and intersecting with $K_{\w{X}}$ and $E_1$ we get $a=b=0$.

We note that the $\ell_j$'s are contained in $\dom(\psi)$; let us still denote by $\ell_j$ their images in $\wi{X}$. These are the unique exceptional lines in $\wi{X}$ that are disjoint from $\Exc(\sigma)$, and they are still numerically equivalent.

Let us show that their class generates an extremal ray $R$ of $\NE(\wi{X})$.
By contradiction, if $[\ell_1]$ does not belong to an extremal ray, we can write
$$\ell_1\equiv \sum_{\ell\not\equiv\ell_1}\lambda_\ell\ell+\sum_{i=1}^r\mu_iC_{D_i}+\Gamma^+$$
where $\lambda_\ell,\mu_i\in\Q_{\geq 0}$, $\ell$ are exceptional lines, and $\Gamma^+$ is an effective  one-cycle (with coefficients in $\Q$) such that $-K_{\w{X}}\cdot \Gamma^+\geq 0$ and $D_i\cdot  \Gamma^+\geq 0$ for every $i=1,\dotsc,r$. Intersecting with $D_i$ we get $\mu_i=\sum_\ell\lambda_\ell D_i\cdot\ell+D_i\cdot\Gamma^+\geq \sum_\ell\lambda_\ell$ for every $i=1,\dotsc,r$, and intersecting with $K_{\w{X}}$ and using that $-K_{\w{X}}\cdot C_{D_i}\geq 2$ for every $i$ we reach a contradiction.

Let $\eta\colon \wi{X}\dasharrow\wi{X}'$ be the flip of $R$. Then the composite map $X\dasharrow\wi{X}'$ factors as sequence of $K$-negative flips, each one $D_i$-negative for some $i$, and every exceptional line in $\wi{X}'$ has positive intersection with some $D_i$. We have a diagram:
$$\xymatrix{{\wi{X}}\ar[d]_{\sigma}\ar@{-->}[r]^{\eta}&{\wi{X}'}\ar[d]^{\sigma'}\\
  {X_0}\ar@{-->}[r]^{\eta_0}&{X_0'}
  }$$
  where $\Exc(\sigma)=D_1\cup\cdots\cup D_r$ is contained in $\dom(\eta)$, so that $\sigma$ and $\sigma'$ are locally isomorphic divisorial contractions, and as in the first part of the proof we see that $X_0'$ is Fano. Moreover $\eta_0$ is the flip of a small extremal ray generated by the class of the exceptional lines $\sigma(\ell_j)\subset X_0$.

  We also note that $f_0(\sigma(\ell_j))=C_j$, so that $f_0$ does not contract any exceptional line, and it is $K$-negative.
\end{proof}
\begin{lemma}\label{f0}
The contraction $f_0$ is  special, $\rho_{X_0}-\rho_{Y_0}=2$, 
and the fiber $f_0^{-1}(p_i)$ is one-dimensional for every $i=1,\dotsc,r$.
\end{lemma}  
\begin{proof}
  We have $\rho_X-\rho_{X_0}=\rho_Y-\rho_{Y_0}=r$, therefore
  $\rho_{X_0}-\rho_{Y_0}= \rho_X-\rho_Y=2$. Moreover $f_0$ is special because $Y_0$ is smooth (Lemma \ref{table}) and if $P\subset X_0$ is a prime divisor with $f_0(P)\subsetneq Y_0$, then its transform $\w{P}\subset\w{X}$ is different from $D_1,\dotsc,D_r$, thus $f(\w{P})\subset Y$ is a prime divisor different from $G_1,\dotsc,G_r$, and $f_0(P)=k(f(\w{P}))\subset Y_0$ is a prime divisor.

  Finally, locally around $p_i$, diagram \eqref{diagram} is isomorphic to diagram \eqref{diagram0}, therefore $\dim f_0^{-1}(p_i)=1$ for every $i=1,\dotsc,r$, as shown in the proof of Lemma \ref{divisorial}.
\end{proof}
\begin{prg}\label{B0}
  Since $f_0$ is special, it has at most isolated $2$-dimensional fibers;
  outside these, $f_0$ is a conic bundle (see Prop.~\ref{conicsing}).

  The discriminant divisor of $f_0$ is $B_0$, and  $(f_0)^*B_0$ has two irreducible components, which are the transforms of $E_1$ and $E_2$ in $X_0$. Indeed $f$ and $f_0$ coincide on $Y\smallsetminus\Exc(k)$ and $Y_0\smallsetminus \{p_1,\dotsc,p_r\}$ respectively, thus this follows from \ref{discrf} and Lemma
  \ref{discriminant}.
Similarly, since $B$ is smooth outside the images of the $2$-dimensional fibers of $f$, we see that
$B_0$ is smooth outside (possibly) $p_1,\dotsc,p_r$ and the images of the $2$-dimensional fibers of $f_0$.
\end{prg}
\begin{lemma}\label{rays2}
  The cone $\NE(f_0)$ has two extremal rays, both of type $(3,2)$, with exceptional divisors the transforms of $E_1$ and $E_2$.
\end{lemma}
\begin{proof}
  We proceed as in the proof of Lemma \ref{rays}; in particular  we have shown there that $E_1$ does not contain exceptional planes.

  Since $f$ and $f_0$ coincide on the general fibers over $B$ and $B_0$ respectively,
  $f_0$ contracts some curve $e_1$ with $E_1''\cdot e_1<0$, where $E_1''\subset X_0$ is
  the transform of $E_1\subset\w{X}$.  Thus $\NE(f_0)$ has an $E_1''$-negative extremal
ray $R$, and $R$ is $K$-negative because $f_0$ is. It is enough to show that $R$ is not small.

 Suppose by contradiction that $R$ is small.
 By Th.~\ref{kachiflip}  for every irreducible component $L$ of $\Lo(R)$ we have $(L, -K_{X_0|L}) \cong(\pr^2, \ol_{\pr^2}(1))$. Consider the birational map $X\dasharrow X_0$ and note that the indeterminacy locus of its inverse has dimensione one, so that it cannot contain $L$. By Prop.~\ref{degreeone} we conclude that $L$ is disjoint from this indeterminacy locus; in particular $L\subset (X_0)_{\reg}$ and $L$ is contained in the open subset where the map ${X}\dasharrow X_0$ is an isomorphism. Thefore
 $L$ is an exceptional plane by Th.~\ref{kawsm}, and its transform in $\w{X}$ gives an exceptional plane contained in $E_1$, a contradiction.
\end{proof}
\begin{prg}\label{reducible}
For every $i=1,\dotsc,r$ the fiber  $f_0^{-1}(p_i)$ is isomorphic to a reducible conic, with a component in the transform of $E_1$ and one in that of $E_2$.

  Indeed, since $p_i\in B_0$ (Lemma \ref{index}) and $\dim f_0^{-1}(p_i)=1$ (Lemma \ref{f0}), $f_0^{-1}(p_i)$ must be isomorphic to a singular conic. On the other hand 
  $f_0^{-1}(p_i)$ cannot be a double line, otherwise the class of $(f_0^{-1}(p_i))_{\text{\it red}}$ should belong to both extremal rays of $\NE(f_0)$.
\end{prg}
\begin{prg}\label{no30}
  No $D_i$ is of type $(3,0)^{\sm}$.

  Indeed $\sigma(D_i)\in f_0^{-1}(p_i)$ (Lemma \ref{Xr})
  and $f_0^{-1}(p_i)$
 is isomorphic to a reducible conic by \ref{reducible}, 
therefore
$\sigma(D_i)$ is contained in an integral curve $\Gamma$ of anticanonical degree one.
This is impossible in the case $(3,0)^{\sm}$, as the transform of $\Gamma$ in $\w{X}$ would have anticanonical degree $\leq -2$, contradicting Lemma \ref{SQMFano}$(b)$.
\end{prg}
\begin{prg}\label{adjacent}
  Consider the transform of $E_j$ in $X_0$, for $j\in\{1,2\}$. This is a fixed prime divisor in $X_0$, therefore $D_i$ is adjacent to $E_j$ in $X$ for every $i=1,\dotsc,r$ (see p.~\pageref{padj} and \cite[Lemma 4.4]{small}). We also note that $B\cap G_i\neq\emptyset$ in $Y$ by Lemma \ref{index}, thus $E_j\cap D_i\neq\emptyset$ in $\w{X}$, and the same must be in $X$, because the indeterminacy locus of $\xi^{-1}\colon\w{X}\dasharrow X$ has dimension one (Lemma \ref{SQMFano}$(a)$). By Lemma \ref{2dimfaces}
  we conclude that $E_j\cdot C_{D_i}=1$ if $D_i$ is of type $(3,0)^Q$, while  $E_j\cdot C_{D_i}\in\{0,1\}$ if $D_i$ is of type $(3,1)^{\sm}$.
\end{prg}
For $i=1,\dotsc,r$ we denote by $C_{G_i}\subset G_i$ a line in $G_i\cong\pr^2$.
\begin{lemma}\label{aeb} 
  For every $i=1,\dotsc,r$ let $x_i\in X_0$ be the singular point of the reducible conic $f_0^{-1}(p_i)$ (see \ref{reducible}).
  One of the following holds:
  \begin{enumerate}[$(a)$]
  \item  $-K_Y=2B$,  $-K_{Y_0}=2B_0$, $B_0$ is smooth at $p_1,\dotsc,p_r$, $X_0$ is smooth, every $D_i$ is of type $(3,1)^{\sm}$, and $(E_1+E_2)\cdot C_{D_i}=1$. Moreover, for every $i$, $\sigma$ blows-up the component of $f_0^{-1}(p_i)$ contained in $\sigma(E_j)$, where $j\in\{1,2\}$ is such that $E_j\cdot C_{D_i}=1$;
  \item   $-K_Y=B$,  $-K_{Y_0}=B_0$, $\Sing(X_0)=\{x_1,\dotsc,x_r\}$, and every $D_i$ is of type $(3,0)^Q$. Moreover $B\cdot C_{G_i}=2$ and
$B_0$ has at $p_i$ a rational double point of type $A_1$ or $A_2$, for every $i=1,\dotsc,r$.
    \end{enumerate} \end{lemma}
  \begin{proof}
    By Lemma \ref{index} we have $-K_Y=\lambda B$ and $-K_{Y_0}=\lambda B_0$ for some $\lambda\in\Q_{>0}$.
Then $$k^*(-K_{Y_0})=-K_Y+2\sum_{i=1}^rG_i\quad\text{and}\quad
    k^*(B_0)=B+\sum_{i=1}^r\bigl(B\cdot C_{G_i}\bigr)G_i,$$ which yields $2=\lambda B\cdot C_{G_i}$ for every $i=1,\dotsc,r$. Recall also that  $\Gamma_i:=B\cap G_i\subset B_{\reg}$ for every $i=1,\dotsc,r$ (see \ref{BG}).

Consider $C_{D_i}\subset D_i\subset\w{X}$. We have $G_i\cdot f_*(C_{D_i})=f^*(G_i)\cdot C_{D_i}=D_i\cdot C_{D_i}=-1$, therefore $f_*(C_{D_i})=C_{G_i}$. Then $B\cdot C_{G_i}=B\cdot f_*(C_{D_i})=f^*(B)\cdot C_{D_i}=(E_1+E_2)\cdot C_{D_i}$.

    If $B_0$ is smooth at some $p_{i_0}$, then $B\cdot C_{G_i}=1$ for every $i$, and $B_0$ is smooth at every $p_i$; moreover $\lambda=2$ and $(E_1+E_2)\cdot C_{D_i}=1$.
By Rem.~\ref{conto}, 
$X_0$ is smooth at $x_i$,  thus $X_0$ is smooth along $f_0^{-1}(p_i)$ and $D_i$ must be of type $(3,1)^{\sm}$, so we get $(a)$.

Otherwise $B_0$ is singular at every $p_i$, and again by Rem.~\ref{conto}  $X_0$ is singular at $x_i$, and $D_i$ must be of type $(3,0)^Q$.
Then $E_j\cdot C_{D_i}=1$ for every $j=1,2$ and $i=1,\dotsc,r$ by
\ref{adjacent}, hence 
 $B\cdot C_{G_i}=(E_1+E_2)\cdot C_{D_i}=2$. This also implies that $\lambda=1$, $B\in|-K_Y|$, and $B_0\in|-K_{Y_0}|$.

Moreover $\Gamma_i=B\cap G_i$ is a conic in $G_i\cong\pr^2$,  $B$ is smooth along $\Gamma_i$, and 
 $\Gamma_i\cdot_B\Gamma_i=G_i\cdot \Gamma_i=-2$. In particular $\Gamma_i$ cannot be non-reduced, otherwise $4|\Gamma_i^2$. If $\Gamma_i$ is smooth, then it is a $(-2)$-curve in $B$, and $p_i$ is a node for $B_0$. If $\Gamma_i$ is a reducible conic, each component is a $(-2)$-curve for $B$, and $B_0$ has a singularity of type $A_2$ at $p_i$.
\end{proof}
\begin{remark}\label{geometry}
  Let us give a more explicit description of $\psi\colon \w{X}\dasharrow\wi{X}$ and of the divisors $D_i$'s.  Recall diagram \eqref{diagram}. 
  In case $(a)$, in $\wi{X}$ we have $D_i\cong\pr_{\pr^1}(\ol^{\oplus 2}\oplus\ol(1))\cong\Bl_{\text{\it line}}\pr^3$ for every $i=1,\dotsc,r$.
  Then
  $\wi{X}\smallsetminus\dom(\psi^{-1})$ is the union of $r$ exceptional lines, the transforms of the second components of $f_0^{-1}(p_i)$ (contained in $\sigma(E_j)$ such that $E_j\cdot C_{D_i}=0$). Moreover in $\w{X}$ we have $D_i\cong\Bl_{\text{\it pt,line}}\pr^3$, $D_i$ contains one exceptional plane $L_i$  in the indeterminacy locus of $\psi$, and $f_{|L_i}\colon L_i\to G_i$ is an isomorphism.

  In case $(b)$, in $\wi{X}$ each $D_i$ is isomorphic to a quadric $Q$, smooth or a cone over $\pr^1\times\pr^1$.  Moreover
  $\wi{X}\smallsetminus\dom(\psi^{-1})$ is the union of $2r$ exceptional lines, the transforms of the components of $f_0^{-1}(p_i)$.  Finally in $\w{X}$ we have $D_i\cong\Bl_{\text{\it 2 pts}}Q\cong\Bl_{\text{\it pt,conic}}\pr^3$, and $D_i$ contains two exceptional planes in the indeterminacy locus of $\psi$, both isomorphic to $G_i$ via $f$. The conic can be smooth or reducible and is isomorphic to $\Gamma_i=B\cap G_i$.
\end{remark}  
 \begin{prg}\label{facto}
   Let us consider a factorization of $f_0$ in elementary steps:
   $$\xymatrix{
{X_0}\ar[r]_{\alpha_0}\ar@/^1pc/[rr]^{f_0}& {W_0}\ar[r]_{\pi_0} &{Y_0}
     }$$
where $\alpha_0$ is an elementary contraction of type $(3,2)$ with exceptional divisor the transform of $E_1$ (see Lemma \ref{rays2}), and set $S_0:=\alpha_0(\Exc(\alpha_0))\subset W_0$; note that $\pi_0(S_0)=B_0$.

Let $T\subset Y_0$ be the finite set given by the images of the $2$-dimensional fibers of $\pi_0$; set $U_Y:=Y_0\smallsetminus T$ and $U_W:=\pi_0^{-1}(U_Y)$.
\end{prg}
\begin{prg}\label{Y0}
  We have $k(\Exc(\ph))\subset U_Y$.

  Indeed if $y\in T$, then $\dim\pi_0^{-1}(y)=2$, hence $\dim f_0^{-1}(y)=2$. In particular $y\neq p_i$ for every $i=1,\dotsc,r$ (Lemma \ref{f0}), $y':=k^{-1}(y)\in Y$ is a point, and $f^{-1}(y')\cong f_0^{-1}(y)$, thus $\dim f^{-1}(y')=2$. Therefore we have $y'\not\in\Exc(\ph)$ by \ref{properties}, and $y\not\in k(\Exc(\ph))$.
\end{prg}
Set $F_i:=\pi_0^{-1}(p_i)\subset W_0$ for $i=1,\dotsc,r$.
\begin{lemma}\label{smooth}
  The $4$-fold $W_0$ is locally factorial and has at most nodes as singularities,
    $U_W$ is smooth, and $\pi_{0|U_W}\colon U_W\to U_Y$ 
  is a smooth morphism with fiber $\pr^1$. Moreover
 $p_i\in U_Y$ and $F_i\subset U_W$ for every $i=1,\dotsc,r$.
 \end{lemma} 
 \begin{proof}
   First of all we  note that $\dim F_i=1$ because $f_0^{-1}(p_i)=\alpha_0^{-1}(F_i)$ has dimension $1$  by Lemma \ref{f0}, therefore $F_i\subset U_W$.

   Every fiber of $\pi_0$ over $U_Y$ is an integral rational curve. Indeed if $y\in U_Y\smallsetminus B_0$, then $\pi_0^{-1}(y)\cong f_0^{-1}(y)\cong\pr^1$. If $y\in U_Y\cap B_0$, then $f_0^{-1}(y)$ is a reducible conic with one irreducible component in the transform of $E_1$, which is contracted to a point by $\alpha_0$; thus $\pi_0^{-1}(y)$ is again irreducible. Then $\pi_{0|U_W}$
   is smooth by \cite[Th.~II.2.8]{kollar},  and since $U_Y$ is smooth (see Lemma \ref{table}), $U_W$ is smooth. 
   
   Finally, since $\alpha_0(\Sing(X_0))\subseteq\{\alpha_0(x_1),\dotsc,\alpha_0(x_r)\}$ (Lemma \ref{aeb}), outside these points $W_0$ is locally factorial and has  at most nodes, at the images of $2$-dimensional fibers of $\alpha_0$ (see Th.~\ref{32}). On the other hand $\alpha_0(x_i)\in F_i\subset U_W$,  thus $W_0$ is smooth at $\alpha_0(x_i)$ for every $i=1,\dotsc,r$.
 \end{proof}
 Set $w_i:=\alpha_0(x_i)$ for $i=1,\dotsc,r$.
 \begin{prg}\label{isowi}
 For every $i=1,\dotsc,r$ we have $F_i\cap S_0=w_i$, $\pi_0(w_i)=p_i$, and
 $\pi_{0|S_0}\colon S_0\to B_0$ is birational and an isomorphism around $w_i$.

 Indeed clearly $w_i\in F_i$, and $w_i\in S_0$ because $x_i\in\Exc(\alpha_0)$ (see \ref{reducible}).  For any one-dimensional fiber $F$ of $\pi_0$ over $B_0$, not contained in $S_0$, let  $\w{F}\subset X_0$ be its transform. We have $-K_{X_0}\cdot \w{F}>0$ because $f_0$ is $K$-negative (see Lemma \ref{X0Fano}), $\Exc(\alpha_0)\cdot\w{F}>0$ because $F\cap S_0\neq\emptyset$, and $-K_{X_0}\cdot\w{F}=-K_{W_0}\cdot F-\Exc(\alpha_0)\cdot\w{F}=2-\Exc(\alpha_0)\cdot\w{F}$, thus
 $\Exc(\alpha_0)\cdot\w{F}=1$. Then $\pi_{0|S_0}\colon S_0\to B_0$ is birational and is an isomorphism around the point $F\cap S_0$.
Moreover note that
 $F_i\not\subset S_0$, otherwise $f_0^{-1}(p_i)=\alpha_0^{-1}(F_i)$ would have dimension $2$, contradicting Lemma \ref{f0}. 
\end{prg}
\begin{prg}\label{Sr}
  Recall cases $(a)$ and $(b)$ from Lemma \ref{aeb}.
In case $(a)$, $S_0$ is smooth at $w_1,\dotsc,w_r$.
In case $(b)$, $S_0$ has rational double points of type $A_1$ or $A_2$ at $w_1,\dotsc,w_r$.

This follows immediately from \ref{isowi} and Lemma \ref{aeb}.
\end{prg}
We recall from Lemma \ref{rays} that $\alpha\colon X\to W$ 
(respectively $\tilde\alpha\colon \w{X}\to\w{W}$) is an elementary contraction of type $(3,2)$ with exceptional divisor  $E_1'$ (respectively $E_1$).
\begin{lemma}\label{bigdiagram}
The birational map $\alpha_0\circ\sigma\circ\psi\circ(\tilde{\alpha})^{-1}\colon\w{W}\dasharrow W_0$
 factors as a SQM $\psi_{\s W}\colon \w{W}\dasharrow\wi{W}$ followed by a divisorial contraction $\sigma_{\s W}\colon \wi{W}\to W_0$, with exceptional divisors the transforms $D_{{\s W},i}$ of $D_i$, and $D_{{\s W},i}\subset\wi{W}_{\reg}$, for every $i=1,\dotsc,r$.
  {\small
   \stepcounter{thm}
\begin{equation}\label{diagramW}
\xymatrix{X\ar@{-->}[r]^{\xi}\ar[d]_{\alpha}&{\w{X}}\ar@/_1pc/[dd]_>>>>{f}\ar[d]^<<<<{\tilde{\alpha}}\ar@{-->}[r]^{\psi}&{\wi{X}}\ar[r]^{\sigma}&{X_0}\ar[d]_{\alpha_0}\ar@/^1pc/[dd]^{f_0}\\
 W\ar@{-->}[r]_{\xi_{\s W}}\ar@{-->}@/^1pc/[rr]^<<<{\zeta} &{\w{W}}\ar[d]^{\pi}\ar@{-->}[r]_{\psi_{\s W}}&{\wi{W}}\ar[r]_{\sigma_{\scriptscriptstyle W}}&{W_0}\ar[d]_{\pi_0}\\
  &Y\ar[rr]^k&&{Y_0}
}\end{equation}}

Recall cases $(a)$ and $(b)$ from Lemma \ref{aeb}.

In case $(a)$ every $D_{{\s W},i}$ can be of type either $(3,0)^{\sm}$ with $\sigma_{\scriptscriptstyle W}(D_{{\s W},i})=w_i=F_i\cap S_0$ (if $E_1\cdot C_{D_i}=1$; in this case $\sigma_{\scriptscriptstyle W}$ blows-up $w_i$), or $(3,1)^{\sm}$ with $\sigma_{\scriptscriptstyle W}(D_{W,i})=F_i=\pi_0^{-1}(p_i)$ (if $E_1\cdot C_{D_i}=0$; in this case $\sigma_{\scriptscriptstyle W}$ blows-up $F_i$), and the two cases are interchanged by contracting $E_2$ instead of $E_1$ in the factorization of $f_0$ as in \ref{facto}.

In case $(b)$ every $D_{{\s W},i}$ is of type $(3,0)^{\sm}$, with $\sigma_{\scriptscriptstyle W}(D_{{\s W},i})=w_i=F_i\cap S_0$, thus $\sigma_{\scriptscriptstyle W}$ blows-up $w_1,\dotsc,w_r$.
\end{lemma}
\begin{proof}
  By \ref{adjacent} $D_i$ is adjacent to $E_j$ for every $i=1,\dotsc,r$ and $j=1,2$, and if $D_i$ is of type  $(3,1)^{\sm}$, we have
$E_j\cdot C_{D_i}\in\{0,1\}$. Note that $\alpha_0\circ\sigma\colon\wi{X}\to W_0$ contracts first $D_1,\dotsc,D_r$ and then $E_1$.

Suppose that we are in case $(a)$. Then by Lemma \ref{aeb} every $D_i$ is of type $(3,1)^{\sm}$,
and
 either $E_1\cdot C_{D_i}=1$ and $E_2\cdot C_{D_i}=0$, or viceversa.
 In the first case, $E_1$ and ${D_i}$ are as  in Lemma \ref{2dimfaces}$(i)$, in the second as in Lemma \ref{2dimfaces}$(ii)$ (and conversely for $E_2$). Then the statement follows from Lemmas \ref{casei} and \ref{casesiiandiii}, that describe how the divisors  $E_1$ and $D_i$ intersect and how they can be contracted in different orders.
 
 Recall the geometric description of $D_i$ and $\psi$ given in Rem.~\ref{geometry}. If $E_1\cdot C_{D_i}=1$, then $\tilde{\alpha}(D_i)\subset\w{W}$ is isomorphic to $\Bl_{\pt}\pr^3$, and
 still contains an exceptional plane $\tilde\alpha(L_i)$, that lies in the indeterminacy locus of $\psi_{\s W}$. Then $D_{W,i}\subset\wi{W}$ is isomorphic to $\pr^3$, and is contracted to $w_i\in W_0$. If instead  $E_1\cdot C_{D_i}=0$, then $\tilde{\alpha}(D_i)\subset\w{W}$ is isomorphic to $\pr^2\times\pr^1$, and
is contained in $\dom(\psi_{\s W})$, so that 
 $D_{W,i}\subset\wi{W}$ is still isomorphic to $\pr^2\times\pr^1$ and is contracted to $F_i\subset W_0$.
 
  Suppose instead that we are in case $(b)$, so that every $D_i$ is of type $(3,0)^Q$. Then $D_i$ and $E_j$ are as in  Lemma \ref{2dimfaces}$(iii)$, and similarly as before the statement follows from Lemma \ref{casesiiandiii}.
In this case $\tilde\alpha(D_i)\subset\w{W}$ is isomorphic to $\Bl_{\pt}\pr^3$ and
 contains one exceptional plane, that lies in the indeterminacy locus of $\psi_{\s W}$.  Then $D_{W,i}\subset\wi{W}$ is isomorphic to $\pr^3$ and is contracted to $w_i\in W_0$.
\end{proof}
We set $\zeta:=\psi_{\s W}\circ\xi_{\s W}\colon W\dasharrow \wi{W}$.
Recall from Lemma \ref{rays} that $S:=\alpha(E_1')\subset W$.
\begin{lemma}\label{W}
  The $4$-fold $\wi{W}$ is locally factorial, has at most nodes as singularities, and
  can contain finitely many pairwise disjoint exceptional lines. If $C\subset\wi{W}$ is an irreducible curve that is not an exceptional line, then $-K_{\wi{W}}\cdot C\geq 1$, and if $-K_{\wi{W}}\cdot C=1$, then $C$ does not meet any exceptional line.

  Moreover the surface $S\subset W$ is contained in $\dom(\zeta)$.
\end{lemma}
With a slight abuse of notation, we still denote by $S$ the transform of $S$ in $\wi{W}$; note that $\sigma_{\scriptscriptstyle W}(S)=S_0$.
\begin{proof}
  By Lemma \ref{rays}, $W$ is Fano, is
 locally factorial, and has  at most nodes at images of some  $2$-dimensional fibers of $\alpha$.

Let $R$ be a small extremal ray of $\NE(W)$. By Th.~\ref{kachiflip} for every irreducible component $L$ of $\Lo(R)$ we have $(L, -K_{W|L}) \cong(\pr^2, \ol_{\pr^2}(1))$. Therefore Prop.~\ref{degreeone} implies that 
either $L\cap S=\emptyset$, or $L=S$. This last case
would give $\dim\N(S,W)=\rho_{\pr^2}=1$ and $\dim\N(\Exc(\alpha),X)\leq 2$ (see Rem.~\ref{linalg}), impossible because
$\rho_X\geq 7$ and $\delta_X\leq 2$ (see \ref{logFano}).
  We conclude that $\Lo(R)\cap S=\emptyset$, thus $\Lo(R)\subset W_{\reg}$, and $\Lo(R)$ is a finite disjoint union of exceptional planes by Th.~\ref{kawsm}.

  Let us consider now the SQM $\zeta\colon W\dasharrow\wi{W}$. 
  The same proof of \cite[Rem.~3.6]{eff} can be applied here:
  since $W$ is Fano, $\zeta$ can be factored as a sequence of $K$-negative flips; the locus of each flip is contained in the smooth locus, and the loci of the flips are all disjoint. Therefore $\wi{W}\smallsetminus \dom(\zeta^{-1})$ is a finite disjoint union of exceptional lines, contained in $\wi{W}_{\reg}$. If $C\subset\wi{W}$ is an irreducible curve such that $C\cap\dom(\zeta^{-1})\neq\emptyset$, and $\w{C}\subset W$ is its transform, we have $-K_{\wi{W}}\cdot C\geq -K_W\cdot \w{C}\geq 1$, and $-K_{\wi{W}}\cdot C> -K_W\cdot \w{C}$ if $C$ meets some exceptional line. Finally $S\subset\dom(\zeta)$ by what precedes.
\end{proof}
\begin{lemma}\label{imageexc}
  If $Y_0$ is Fano, then   $W_0$ is Fano; otherwise $W_0$ has a SQM which is Fano. If $\ell\subset W_0$ is an exceptional line, then $-K_{Y_0}\cdot \pi_0(\ell)=0$, and $p_i\not\in \pi_0(\ell)$ for every $i=1,\dotsc,r$. If $C\subset W_0$ is an irreducible curve that is not an exceptional line, then $-K_{W_0}\cdot C\geq 1$.
\end{lemma}
\begin{proof}
  This is similar to the proof of Lemma \ref{X0Fano}. We keep the same notation as in the proof of Lemma \ref{W}. If $Y_0$ is Fano, then the SQM $\zeta\colon W\dasharrow\wi{W}$ is a sequence of $K$-negative flips, each negative for some $D_{{\s W},i}$ in $W$. Therefore every exceptional line in $\wi{W}$ has positive intersection with some $D_{W,i}$. Then we consider $\sigma_{\scriptscriptstyle W}^*(-K_{W_0})=-K_{\wi{W}}+\sum_im_iD_{{\s W},i}$, where $m_i=2$ (respectively, $m_i=3$) if $D_{{\s W},i}$ is of type $(3,1)^{\sm}$ (respectively, $(3,0)^{\sm}$), and we show that $-K_{\wi{W}}+\sum_im_iD_{{\s W},i}$ is nef and that $(-K_{\wi{W}}+\sum_im_iD_{{\s W},i})^{\perp}\cap\NE(\wi{W})=\NE(\sigma_{\scriptscriptstyle W})$. This shows that $W_0$ is Fano.

  If instead $Y_0$ is weak Fano  with $\rho_{Y_0}=2$, then there is a $K$-positive flip $\eta_{\scriptscriptstyle W}\colon W_0\dasharrow W_0'$ such that $W_0'$ is Fano; the indeterminacy locus of $\eta_{\scriptscriptstyle W}$ is given by the exceptional lines $\alpha_0(\sigma(\ell_j))$, in the notation of the proof of  Lemma \ref{X0Fano}. In particular, if $\ell\subset W_0$ is an exceptional line, then $\ell=\alpha_0(\sigma(\ell_j))$ for some $j$, and $\pi_0(\ell)=f_0(\sigma(\ell_j))=C_j$ a curve of anticanonical degree zero in $Y_0$, so that  $p_i\not\in \pi_0(\ell)$ for every $i=1,\dotsc,r$ by Lemma \ref{milano}$(b)$.
\end{proof}
\begin{prg}
  Recall from Lemma \ref{smooth} that $\pi_{0|U_W}\colon U_W\to U_Y$ is a smooth morphism with fiber $\pr^1$.
Since $Y_0$ is rational (see Lemma \ref{table}), its Brauer group $\Br(Y_0)$ is trivial, and since $Y_0\smallsetminus U_Y$ is finite, we have $\Br(U_Y)\cong\Br(Y_0)=0$ by purity \cite[Th.~$2'$ p.~131]{gabber}. We conclude that
 $U_W=\pr_{U_Y}(\ma{E})$ for some rank two vector bundle  $\ma{E}$ on $U_Y$ (see for instance \cite[\S6.3]{debarrerationality}). Moreover $\det\ma{E}\in\Pic(U_Y)$ extends to a line bundle $L\in\Pic(Y_0)$.
\end{prg}
 \begin{lemma}\label{degree2}
  Let $i\in\{1,\dotsc,r\}$, and suppose that $C\subset Y_0$ is a smooth rational curve with $-K_{Y_0}\cdot C=2$ and $p_i\in C$. Then $p_j\not\in C$ for every $j\neq i$, and one of the following holds:
 \begin{enumerate}[$(i)$]
 \item $L\cdot C$ is odd, $\pi_0^{-1}(C)\cong\mathbb{F}_1$,
   and $\sigma_{\scriptscriptstyle W}$ blows-up $F_i$;
  \item
    $L\cdot C$ is even, $\pi_0^{-1}(C)\cong\pr^1\times\pr^1$, and
  $\sigma_{\scriptscriptstyle W}$ blows-up $w_i$.
\end{enumerate}
\end{lemma}  
\begin{proof}
  Since the transform of $C$ in $Y$ has anticanonical degree zero, we have  $C\subset U_Y$ by \ref{Y0}; moreover $p_j\not\in C$ for every $j\neq i$, otherwise $-K_Y$ would not be nef.
  
  Set $S:=\pi_0^{-1}(C)\cong\mathbb{F}_e$ with $e\in\Z_{\geq 0}$. We
apply Rem.~\ref{sections} and
  keep the same notation; in particular  $L\cdot C\equiv e\mod 2$. We have $-K_{W_0}\cdot\Gamma^-=2-e$ and $\pi_0(\Gamma^-)=C$,
  therefore by Lemma \ref{imageexc} $\Gamma^-$ is not an exceptional line and $e\in\{0,1\}$ .

 Suppose that  $\sigma_{\scriptscriptstyle W}$ blows-up $F_i$. Since $F_i\cap\Gamma^-\neq\emptyset$, and $-K_{W_0}\cdot\Gamma^-= 2-e$, the transform of $\Gamma^-$ in $\wi{W}$ has anticanonical degree $-e$, which implies that $e=1$ (see Lemma \ref{W}), and $L\cdot C$ is odd.

 If instead $\sigma_{\scriptscriptstyle W}$ blows-up $w_i$, we claim that $e\neq 1$, because in that case we would have $-K_{W_0}\cdot\Gamma^-=1$ and $-K_{W_0}\cdot\Gamma^+=3$.
If $w_i\in\Gamma^-$, then the transform of $\Gamma^-$ in $\wi{W}$ has anticanonical degree $-2$, impossible by Lemma \ref{W}. Otherwise, we can assume that $w_i\in\Gamma^+$, and then the transform of $\Gamma^+$  has  degree zero, again impossible.
Hence $e=0$ and $L\cdot C$ is even.
\end{proof}
\begin{lemma}\label{degree4}
   Suppose that $C\subset Y_0$ is a smooth rational curve with $-K_{Y_0}\cdot C=4$ and $p_i,p_j\in C$, with $i\neq j$. Then $p_k\not\in C$ for every $k\neq i,j$, and one of the following holds:
   \begin{enumerate}[$(i)$]
   \item   $L\cdot C$ is odd
         and $\sigma_{\scriptscriptstyle W}$ blows-up either $F_i$ and $F_j$, or
$w_i$ and $w_j$;
\item  $L\cdot C$ is even
    and
 $\sigma_{\scriptscriptstyle W}$ blows-up either $F_i$ and $w_j$, or $w_i$ and $F_j$.
   \end{enumerate}
\end{lemma}  
\begin{proof}
Again, since the transform of $C$ in $Y$ has anticanonical degree zero, we have $C\subset U_Y$
by \ref{Y0}, and $p_k\not\in C$ for every $k\neq i,j$.
Similarly as before we apply Rem.~\ref{sections}; we have $\pi_0^{-1}(C)\cong\mathbb{F}_e$ with $L\cdot C\equiv e\mod 2$,
$-K_{W_0}\cdot\Gamma^-=4-e$, and $-K_{W_0}\cdot\Gamma^+=4+e$; moreover $\Gamma^-$ is not an exceptional line by Lemma \ref{imageexc},
therefore $e\in\{0,1,2,3\}$. 

 If $\sigma_{\scriptscriptstyle W}$ blows-up $F_i$ and $F_j$, then $\Gamma^-$ meets both fibers and its transform in $\wi{W}$ has anticanonical degree $-e$, which yields $e=1$ (see Lemma \ref{W}), and we have $(i)$.

 If $\sigma_{\scriptscriptstyle W}$ blows-up $F_i$ and $w_j$, then the transform of $F_j$ in $\wi{W}$ is an exceptional line. 
 We show that $e=0$, which gives $(ii)$. If $e>0$, then $-K_{W_0}\cdot\Gamma^-\leq 3$; since $\Gamma^-\cap F_i\neq\emptyset$, if $w_j\in\Gamma^-$, the transform of $\Gamma^-$ in $\wi{W}$ would have anticanonical degree $-e-1$, impossible by Lemma \ref{W}.
Then
 $w_j\not\in\Gamma^-$, and we can assume that 
$w_j\in\Gamma^+$, so that the transforms of $\Gamma^-$ and $\Gamma^+$ have degrees, respectively, $2-e$ and $e-1$, which yields $e=3$. Then the transforms of $\Gamma^-$ and $F_j$ are exceptional lines in $\wi{W}$, but they intersect, which is again impossible. 

Finally suppose that $\sigma_{\scriptscriptstyle W}$ blows-up $w_i$ and $w_j$.
We show that $\Gamma^-$ cannot contain any of these two points. 
Indeed $\Gamma^-$ can contain at most one of them, say $w_i$. Then the transform of $\Gamma^-$ in $\wi{W}$ has degree $1-e$, which implies $e\in\{0,2\}$ (see Lemma \ref{W}). If $e=0$, then $S\cong\pr^1\times\pr^1$, and there is a section of $\pi_{0|S}$, containing $w_i$ and $w_j$, of degree $4$ or $6$, in both cases impossible by Lemma \ref{W}, as its transform would have degree $-2$ or $0$ respectively. If $e=2$, then both the transforms of $\Gamma^-$ and $F_j$ are exceptional lines in $\wi{W}$, but they intersect, which is again impossible by Lemma \ref{W}. 
Thus $w_i,w_j\not\in\Gamma^-$.

 Therefore $e>0$ and we can choose $\Gamma^+$ containing both points. Then its transform has degree $e-2$, hence $e\in\{1,3\}$  (see Lemma \ref{W}), and we have again $(i)$.
\end{proof}
We recall that, up to flops, $Y_0$ belongs to the list of six $3$-folds given in Lemma \ref{table}.
\begin{lemma}\label{blowup}
  Suppose that $Y_0$ is isomorphic to $\pr_{\pr^2}(T_{\pr^2})$ or to a linear section of $\Gr(2,5)$.
  If there exists a smooth rational curve $C\subset Y_0$  with $-K_{Y_0}\cdot C=4$ and $p_1,p_2\in C$, then $L\cdot C$ is odd.
\end{lemma}
\begin{proof}
  In both cases there exist
smooth rational curves
$\Gamma_1,\Gamma_2\subset Y_0$  with $\Gamma_1\equiv \Gamma_2$, $-K_{Y_0}\cdot \Gamma_i=2$, and $p_i\in \Gamma_i$, for $i=1,2$. Indeed for $\pr_{\pr^2}(T_{\pr^2})$ we just consider the fibers of one of the $\pr^1$-bundles onto $\pr^2$ (note that $p_1$ and $p_2$ cannot be contained in the same fiber, see Lemma \ref{degree2}).  For the linear section of $\Gr(2,5)$, it is well-known that it is covered by an irreducible family of lines in the Pl\"ucker embedding $\Gr(2,5)\subset\pr^9$, see for instance \cite[\S 2.2]{sanna}.

  Set $d:=L\cdot \Gamma_i$.
By Lemma \ref{degree2}, if $d$ is odd, then $\sigma_{\scriptscriptstyle W}$ blows-up  $F_1$ and $F_2$, while if $d$ is even, then $\sigma_{\scriptscriptstyle W}$ blows-up  $w_1$ and $w_2$. Thus $L\cdot C$ is odd by Lemma \ref{degree4}.
\end{proof}
\begin{lemma}\label{1}
$Y_0\not\cong\pr_{\pr^2}(T_{\pr^2})$.
\end{lemma}
\begin{proof}
Assume by contradiction that $Y_0=\pr_{\pr^2}(T_{\pr^2})$. For $i=1,2$
      let $\pi_i\colon Y_0\to\pr^2$ be a $\pr^1$-bundle, and $\Gamma_i\subset Y_0$ a fiber of $\pi_i$.

        We show that there exists a smooth rational curve $C\subset Y_0$ with $C\equiv \Gamma_1+\Gamma_2$ and containing $p_1$ and $p_2$.
  Note that $\pi_i(p_1)\neq\pi_i(p_2)$ for $i=1,2$, otherwise $\Bl_{p_1,p_2}Y_0$ would not be weak Fano.  
We consider the line  $\ell:=\overline{\pi_1(p_1)\pi_1(p_2)}\subset\pr^2$ and the surface $S:=\pi_1^{-1}(\ell)$.
Then $S\cong\mathbb{F}_1$, and $\pi_{2|S}\colon S\to \pr^2$ is a blow-up with exceptional curve 
$\overline{\Gamma}_2\subset S$, with $\overline{\Gamma}_2\equiv \Gamma_2$.

If $p_1\in\overline{\Gamma}_2$, then $\overline{\Gamma}_2$ and the fiber of $\pi_1$ through $p_2$   would give two intersecting curves of anticanonical degree zero in $\Bl_{p_1,p_2}Y_0$, contradicting Lemma \ref{milano}$(b)$. Thus none of $p_1,p_2$ is contained in $\overline{\Gamma}_2$, and we can find in $S$ a smooth rational curve $C$ such that $p_1,p_2\in C$ and $C\equiv \Gamma_1+\Gamma_2$, hence $-K_{Y_0}\cdot C=4$.

Set $d_j:=L\cdot \Gamma_j$ for $j=1,2$. Note that every $p_i$ is contained both in a fiber of $\pi_1$ and in a fiber of $\pi_2$, hence by Lemma \ref{degree2} we have $d_1\equiv d_2\mod 2$. Then
 $L\cdot C=d_1+d_2$ is even, but this contradicts Lemma \ref{blowup}.
\end{proof}
\begin{lemma}\label{2}
  $Y_0$ is not isomorphic to a linear section of $\Gr(2,5)$.
    \end{lemma}
    \begin{proof}
      By contradiction suppose that $Y_0$ is isomorphic to a linear section of $\Gr(2,5)\subset\pr^9$.
The Hilbert scheme of conics in $Y_0$ is studied in detail in \cite[\S 2.3]{sanna}; in particular it is irreducible, and $Y_0$ also contains double lines (\cite[Prop.~2.44]{sanna}), so that if $C\subset Y_0$ is a conic, $L\cdot C$ must be even.
      By \cite[Cor.~2.43]{sanna},  there exists a conic $C\subset Y_0$ containing $p_1$ and $p_2$; we have $-K_{Y_0}\cdot C=4$. By Lemma \ref{blowup} $C$ cannot be smooth.

      The two points $p_1$ and $p_2$ cannot be contained in the same component of $C$, thus $C=C_1\cup C_2$ with $p_i\in C_i$, $p_i\neq C_1\cap C_2$, and $-K_{Y_0}\cdot C_i=2$, for $i=1,2$. Then the transforms of $C_1$ and $C_2$ in $\Bl_{p_1,p_2}Y_0$ are two curves of anticanonical degree zero which intersect, contradicting Lemma \ref{milano}$(b)$.
 \end{proof}
\begin{lemma}\label{3}
   $Y_0$ is not isomorphic to a divisor of degree $(1,2)$ in $\pr^2\times\pr^2$.
  \end{lemma}
\begin{proof}
  Assume by contradiction that $Y_0\subset\pr^2\times\pr^2$ is a divisor of degree $(1,2)$. 
  Let 
$\pi_i\colon Y_0\to\pr^2$, for $i=1,2$, be the restrictions of the two projections, and
$C_i\subset Y_0$ a general fiber of $\pi_i$. Note that $\pi_2$ is a $\pr^1$-bundle, while $\pi_1$ is a conic bundle, with discriminant a cubic curve in $\pr^2$. In particular, if $\Gamma_1\subset Y_0$ is a component of a reducible fiber of $\pi_1$, then $C_1\equiv 2\Gamma_1$.

Let $D\subset Y_0$ be the pullback of the discriminant curve of $\pi_1$ in $\pr^2$, so that $D$ is covered by rational curves of anticanonical degree one. Note that $D\in|\pi_1^*\ol_{\pr^2}(3)|$, thus $D\cdot C_2=3$.

We show that also the divisor $\pi_0^{*}(D)$  in $W_0$ is covered by curves of anticanonical degree one.
Indeed let $\Gamma\subset D$ be a component of a general reducible conic. Since $\Gamma$ is general, we have $\Gamma\subset U_Y$
(recall from \ref{facto} and Lemma \ref{smooth} that $U_Y$ is the open subset of $Y$ where $\pi_0$ is smooth, and that $Y\smallsetminus U_Y$ is finite);  consider $S:=\pi_0^{-1}(\Gamma)\subset W_0$ and apply Rem.~\ref{sections}. Note that, since $Y_0$ is Fano, $W_0$ is Fano by Lemma \ref{imageexc},
and $-K_{W_0}\cdot\Gamma^-=1-e$, where $S\cong\mathbb{F}_e$. Hence  we have $e=0$, $S\cong\pr^1\times\pr^1$, and the horizontal curves in $S$ have anticanonical degree one; thus
$\pi_0^{*}(D)$ is  covered by curves of anticanonical degree one.
 
  There exist fibers $C_1$ and $C_2$, of $\pi_1$ and $\pi_2$ respectively, containing $p_1$; note that $C_1$ must be a smooth fiber, otherwise $\Bl_{p_1}Y_0$ would not be weak Fano. We have $L\cdot C_1=2L\cdot\Gamma_1$ even; then Lemma \ref{degree2} implies that $\sigma_{\scriptscriptstyle W}$ blows-up $w_1$, and that $L\cdot C_2$ is even too. 

  Let us consider now $S':=\pi_0^{-1}(C_2)$.
  Since $L\cdot C_2$ is even, by Lemma \ref{degree2} we have again $S'\cong\pr^1\times\pr^1$, and  the horizontal curves have anticanonical degree $2$ in $W_0$  (see Rem.~\ref{sections}). Let $\w{C}_2$ be the horizontal curve containing $w_1\in S'$. Then $\pi_0^{*}(D)\cdot \w{C}_2=D\cdot C_2>0$; on the other hand $w_1\not\in  \pi_0^{*}(D)$ (because $w_1$ cannot be contained in a curve of anticanonical degree one, by Lemma \ref{W}), thus $\w{C}_2$ must intersect $\pi_0^{*}(D)$ in some point different from $w_1$, and $\w{C}_2$ intersects some curve $\Gamma$ with $-K_{W_0}\cdot\Gamma=1$.

Then the transform of $\w{C}_2$ in $\wi{W}$ gives an exceptional line which meets a curve of anticanonical degree one, contradicting  Lemma \ref{W}.
\end{proof}
\begin{lemma}\label{4}
    $Y_0$ is Fano.
    \end{lemma}
    \begin{proof}
       By Lemma \ref{table}, if $Y_0$ is not Fano, then $\rho_{Y_0}=2$ and up to flops $Y_0$ is isomorphic to one of the weak Fano $3$-folds in  \cite[Th.~3.5(3), Th.~3.6(1)]{jahnkepeternell}. 

   We consider first  \cite[Th.~3.6(1)]{jahnkepeternell}, where
   $Y_0=\pr_{\pr^2}(\ma{F})$ with $\ma{F}$ a rank $2$ vector bundle.
   There is a flop $Y_0\dasharrow Y_0'$ where $Y_0'$ is \cite[Th.~3.5(4)]{jahnkepeternell}, namely
   $Y_0'\subset P':=\pr_{\pr^1}(\ol\oplus\ol(1)^{\oplus 3})$ is a general divisor in the linear system $|2\eta-F|$,  $\eta$ the tautological class and $F$ a fiber of the $\pr^3$-bundle $P'\to\pr^1$. The $\pr^3$-bundle restricts to a quadric bundle $Y_0'\to\pr^1$.
Moreover $P'\to\pr^1$ has a section $\ell'\subset P'$ with normal bundle $\ol_{\pr^1}(-1)^{\oplus 3}$, in fact $\ell'$ is an exceptional line in $P'$, and $\ell'\subset Y_0'$ is the flopping curve. There is a flip $P\dasharrow P'$ where $P=\pr_{\pr^2}(\ol\oplus\ol(1)^{\oplus 2})$ is a Fano $4$-fold containing $Y_0$, and the $\pr^1$-bundle 
$\pi\colon Y_0\to\pr^2$ is given by the restriction of the $\pr^2$-bundle $P\to\pr^2$. 

The fibers $C\subset Y_0$  of $\pi$ are the transforms of the lines  in the fibers of the quadric bundle $Y_0'\to\pr^1$ which intersect $\ell'$.
If $\Gamma\subset Y_0$ is the transform of a general line in a fiber of the quadric bundle, one can check that
$\Gamma\equiv C+\ell$, where $\ell\subset Y_0$ is the flopping curve.

Recall that $p_i\not\in\ell$ for every $i$ (see Lemma \ref{milano}$(b)$), thus $\ell$ is contained in the open subset where $k\colon Y\to Y_0$ is an isomorphism, and $f^{-1}(k^{-1}(\ell))\cong (f_0)^{-1}(\ell)$ (see \eqref{diagram}). By  \ref{properties} we conclude that $f_0^{-1}(\ell)\cong\mathbb{F}_1$
and  $\ell\subset U_Y$ (see \ref{facto} and Lemma \ref{smooth}), thus $L\cdot\ell$ is odd by Rem.~\ref{sections}.
 We denote by $p'_1$ the image of $p_1$ in $Y_0'$.
Consider the fiber $F$ of the quadric bundle through $p_1'$, and 
let $\overline{\Gamma}'\subset F$ be a line through $p_1'$. Then
$\overline{\Gamma}'\cap\ell'=\emptyset$,  otherwise in 
 $\Bl_{p'_1}Y_0'$ these curves would be two intersecting curves of anticanonical degree zero, contradicting Lemma \ref{milano}$(b)$.

 Let $\overline{\Gamma}\subset Y_0$ be the transform of $\overline{\Gamma}'$.
   Consider now the point $p_1\in Y_0$ and let $\overline{C}$ be the fiber of $\pi$ 
containing $p_1$. 
      We have $-K_{Y_0}\cdot \overline{\Gamma}=2$, $p_1\in\overline{\Gamma}$, and $\overline{\Gamma}\equiv \overline{C}+\ell$. In particular we see that $L\cdot\overline{\Gamma}$ and $L\cdot\overline{C}$ have different parity, but this contradicts 
      Lemma \ref{degree2}.

      \medskip

      The case where $Y_0$ is isomorphic to  the weak Fano $3$-fold in \cite[Th.~3.5(3)]{jahnkepeternell} is similar. Now $Y_0\subset \pr_{\pr^1}(\ol^{\oplus 2}\oplus\ol(1)^{\oplus 2})$ in the linear system $|2\eta|$, and $Y_0\to\pr^1$ is a quadric bundle. Moreover, considering the flop $Y_0\dasharrow Y_0'$, we have that $Y_0'$ is of the same type as $Y_0$. In this case $Y_0$ has two  flopping curves $\ell_1$ and $\ell_2$, with $\ell_1\equiv\ell_2$, both sections of the quadric bundle ($\ell_1$ and $\ell_2$ are disjoint for $Y_0$ general, but possibly $\ell_1=\ell_2$ for some special $Y_0$). Similarly as in the previous case, using lines in the fibers of the quadric bundles on $Y_0$ and $Y_0'$, we construct two smooth rational curves $C,\Gamma\subset Y_0$ through $p_1$ with $\Gamma\equiv C+\ell_1$ and $-K_{Y_0}\cdot\Gamma=-K_{Y_0}\cdot C=2$; here $C$ is a line in the fiber of the quadric bundle $Y_0\to\pr^1$ through $p_1$, while $\Gamma$ is a section. In the end we
      obtain a contradiction with the parity of the intersection of $L$ with these curves.
\end{proof}
\begin{prg}\label{last}
  We have $Y_0\cong\pr^3$, $Y\cong\Bl_{r\,\pts}\pr^3$, and $W_0$ is Fano.

  Indeed in Lemmas \ref{1} -- \ref{4}  we have excluded all the other cases of Lemma \ref{table}. Then $W_0$ is Fano by Lemma \ref{imageexc}.
  \end{prg}
  \begin{lemma}\label{P3}
    Let $\ell\subset\pr^3$ be a line. Then $L\cdot \ell$ is odd, and   $\sigma_{\scriptscriptstyle W}$ blows-up either $r$ points or $r$ fibers.
  \end{lemma}
  \begin{proof}
    Indeed if $\sigma_{\scriptscriptstyle W}$ blows-up both points and fibers, since $r\geq 3$ we can find three points $p_i,p_j,p_k$ such that $\sigma_{\scriptscriptstyle W}$ blows-up $w_i,w_j,F_k$ (or conversely $F_i,F_j,w_k$). Then by Lemma \ref{degree4} we have $L\cdot \overline{p_ip_j}$ odd and $L\cdot \overline{p_ip_k}$ even, a contradiction. Therefore $\sigma_{\scriptscriptstyle W}$ blows-up either $r$ points, or $r$ fibers, and $L\cdot \ell$ is odd again 
    by Lemma \ref{degree4}.  
\end{proof}
   \begin{prg}\label{switch}   \emph{Up to switching $E_1$ and $E_2$, from now on we assume that 
     $\sigma_{\scriptscriptstyle W}$ blows-up $r$ points.}

     Indeed, suppose that $\sigma_{\scriptscriptstyle W}$ blows-up $r$ fibers. This means that every $D_{{\s W},i}$ is of type $(3,1)^{\sm}$, so by Lemma \ref{bigdiagram} we are in case $(a)$ and, if we change the factorization of $f_0$ in \ref{facto} by contracting $E_2$ instead of $E_1$, in the new factorization every $D_{{\s W},i}$ is of type $(3,0)^{\sm}$, namely  $\sigma_{\scriptscriptstyle W}$ blows-up $r$ points.
     \end{prg}
      \begin{lemma}\label{psi}
 The morphism  $\pi_0$ has no 2-dimensional fibers.
   \end{lemma}
   \begin{proof}
    Recall that $F_1$ is a smooth fiber of $\pi_0$ (see Lemma \ref{smooth}). We consider the blow-up $\sigma_1\colon \wi{W}_1\to W_0$ of the first point $w_1\in F_1\subset W_0$; the transform of $F_1$ in $\wi{W}_{1}$ is an exceptional line, and
      we  have a diagram:
      $$\xymatrix{
{W_{1}}\ar[d]_{\pi_{1}}\ar@{-->}[r]&{\wi{W}_{1}}\ar[r]^{\sigma_1}&{W_0}\ar[d]^{\pi_0}\\
{Y_{1}}\ar[rr]^{\Bl_{p_1}}&&{\pr^3}
        }$$
        where $\wi{W}_{1}\dasharrow W_{1}$ flips the transform of $F_1$, which is contained in $(\wi{W}_{1})_{\reg}$ (see Lemma \ref{bigdiagram}).
        The exceptional divisor $D_{{\s W},1}=\Exc(\sigma_1)\subset \wi{W}_{1}$ is isomorphic to $\pr^3$, and its transform $D_1'\subset W_1$ is isomorphic to $\Bl_{\pt}\pr^3$, so that $\pi_{1|D_1'}\colon D_1'\to G_1$ is a $\pr^1$-bundle, where
       $G_1\cong\pr^2\subset Y_1$ is  the exceptional divisor over $p_1$.
        
Since $\pi_{1}$ and $\pi_0$ are isomorphic over $Y_{1}\smallsetminus G_1$ and $\pr^3\smallsetminus\{p_1\}$ respectively, it is enough to show that $\pi_{1}$ has no $2$-dimensional fibers.

We  have $Y_{1}=\Bl_{p_1}\pr^3$ and there is a $\pr^1$-bundle $\beta\colon Y_{1}\to\pr^2$. Moreover $Y_{1}$ is Fano, thus $W_{1}$ is Fano too, because if $W_1$ contained an exceptional line, its image in $Y_1$ would have anticanonical degree zero (see Lemma \ref{imageexc}).

  We note that since $\pi_{1}$ has at most finitely many $2$-dimensional fibers, and it is smooth outside these fibers, the composition $\theta:=\beta\circ\pi_{1}\colon W_{1}\to\pr^2$ has at most finitely many reducible fibers, and is equidimensional.
    Moreover     $\theta$ has a second factorization in elementary contractions: $$\xymatrix{{W_{1}}\ar[dr]^{\theta}\ar[d]_{\pi_{1}}\ar[r]^{\gamma}&A\ar[d]^{\delta}\\
{Y_{1}}\ar[r]^{\beta}&{\pr^2}
}$$
Since $\beta$ has one-dimensional fibers, and $\pi_{1}$ is finite on fibers of $\gamma$, we see that $\gamma$ has fibers of dimension at most $1$.

We show that $\gamma$ is of fiber type. By contradiction, suppose that $\gamma$ is birational. Since $W_{1}$ is Fano with at most isolated, locally factorial, and terminal singularities (see Lemma \ref{smooth}), $\gamma$ must be divisorial by
Th.~\ref{gdn}, thus it is of type $(3,2)$.  Let $E\subset W_{1}$ be the exceptional divisor.
We cannot have $\theta(E)=\{pt\}$, because $\theta$ is equidimensional. If $\theta(E)$ is 
a curve in $\pr^2$, then every fiber of $\theta$ over this curve
 is reducible, again a contradiction. Therefore $\theta(E)=\pr^2$.

 Set $q:=\beta(p_2)\in\pr^2$ and $C:=\beta^{-1}(q)
 \subset Y_{1}$;
   then $C$ is a smooth rational curve with $-K_{Y_1}\cdot C=2$ and $p_2\in C$.
 As in Lemma \ref{degree2} we see that $C$ is contained in the open subset where $\pi_1$ is smooth, and that
 $S:=\pi_{1}^{-1}(C)\cong\pr^1\times\pr^1$, because by \ref{switch} $\sigma_{\scriptscriptstyle W}$ blows-up $w_2$. On the other hand $S=\theta^{-1}(q)$, and
 $\gamma_{|S}$ is a non-trivial birational map,  thus we have a contradiction.

 Therefore
$\gamma$ is of fiber type, and  has fibers of dimension at most $1$, so $\dim A=3$. We note that $A$ is $\Q$-factorial and log Fano, and $\delta\colon A\to\pr^2$ is an elementary contraction, thus it must be equidimensional. Now if $F\subset W_{1}$ is a fiber of $\pi_{1}$, then $\gamma$ is finite on $F$, and $\gamma(F)$ is a fiber of $\delta$. Therefore $\dim F=\dim\gamma(F)=1$, and this concludes the proof.
\end{proof}
\begin{lemma}\label{Wr}
  We have $W_0\cong\pr_{\pr^3}(\ol\oplus\ol(1))\cong \Bl_{q_0}\pr^4$, $\wi{W}\cong\Bl_{q_0,q_1,\dotsc,q_r}\pr^4$,
  and $W$ is smooth and is the Fano model of $\Bl_{r+1\,\pts}\pr^4$ (see Ex.~\ref{Fanomodel}).
\end{lemma}
We denote by $\sigma_{q_0}\colon W_0\to\pr^4$ the blow-up map, and set $q_i:=\sigma_{q_0}(w_i)\in\pr^4$ for $i=1,\dotsc,r$.
\begin{proof}
  Since $\pi_0$ has no $2$-dimensional fiber by Lemma \ref{psi},
we have $T=\emptyset$ and
$W_0=U_W$ is smooth (see \ref{facto} and Lemma \ref{smooth}). Moreover
  $W_0=\pr_{\pr^3}(\ma{E})$ where the vector bundle
   $\ma{E}$ has odd degree by Lemma \ref{P3}.
  
Since $W_0$ is Fano (see \ref{last}),
the possible vector bundles $\ma{E}$ have been classified in \cite[Th.~(2.1)]{szurekwisn}; there is only one case where $\ma{E}$ is not decomposable (the so-called null-correlation bundle) and it has even degree, so it cannot occur here.
 Therefore $\ma{E}$ is decomposable of odd degree,
and  $W_0\cong\pr_{\pr^3}(\ol\oplus\ol(b))$ with  $b\in\{1,3\}$.

If $b=3$, the negative section of $W_0\to\pr^3$ is a fixed prime divisor $E$ covered by curves of anticanonical degree one, corresponding to lines in $\pr^3$. Consider the composite birational map $X\dasharrow W_0$. By Prop.~\ref{degreeone} $E$ must be contained in the open subset where this map is an isomorphism, and the transform $E_X\subset X$ of $E$ is a fixed prime divisor with $E_X\cong\pr^3$ and $\ma{N}_{E_X/E}\cong\ol_{\pr^3}(-3)$, contradicting Th.-Def.~\ref{fixed}.

Thus $b=1$ and $W_0\cong\pr_{\pr^3}(\ol\oplus\ol(1))\cong \Bl_{q_0}\pr^4$. Let $\sigma_{q_0}\colon W_0\to\pr^4$ be the blow-up map.
Note that $\Exc(\sigma_{q_0})$ is covered by curves of anticanonical degree $3$, thus $w_i\not\in\Exc(\sigma_{q_0})$, otherwise the transform in $\wi{W}$ of such a curve containing $w_i$ would have anticanonical degree $0$, contradicting Lemma \ref{W}.  Hence $q_i:=\sigma_{q_0}(w_i) \neq q_0$ for every $i=1,\dotsc,r$,
and $\wi{W}\cong\Bl_{q_0,q_1,\dotsc,q_r}\pr^4$.

Finally $W$ is Fano by Lemma \ref{rays}, and there is a SQM $\zeta\colon W\dasharrow\wi{W}$ (see \eqref{diagramW}), so we get the statement.
\end{proof}
\begin{prg}\label{linear}
  The points $q_0,\dotsc,q_r\in\pr^4$ are in general linear position.

  In fact if a line in $\pr^4$ contains $3$ points among the $q_i$'s, then its transform in $\wi{W}$ has anticanonical degree $-4$, which is impossible by Lemma \ref{W}. Similarly, if $4$ (respectively, $5$) among the $q_i$'s are contained in a plane (respectively, a hyperplane), we consider a conic (respectively, a twisted cubic) containing them, and get again a contradiction.
\end{prg}  
\begin{lemma}\label{bound9}
  We have $r\leq 6$ and $\rho_X\leq 9$.
\end{lemma}
\begin{proof}
By Lemma \ref{Wr}, $W$ is the Fano model of $\Bl_{r+1\,\pts}\pr^4$; in particular $\rho_W\leq 9$, see Ex.~\ref{Fanomodel}. Moreover, when $\rho_W=9$, $W$ does not have non-trivial contractions of fiber type by \cite[Prop.~1.7]{vb}.
  On the other hand by Lemma \ref{regular} there is a contraction  of fiber type $W\to Z$ where $Z$ is the anticanonical model of $Y$.
  We conclude that $\rho_W\leq 8$, hence $\rho_X=\rho_W+1\leq 9$ (see Lemma \ref{rays}) and $r=\rho_X-3\leq 6$.  
\end{proof}
Recall cases $(a)$ and $(b)$ from Lemma \ref{aeb} and \ref{Sr}, and that $S_0=\alpha_0(\Exc(\alpha_0))\subset W_0$ (see \ref{facto}). For the reader's convenience, we report here diagram \eqref{diagramW}.
{\small
  \stepcounter{thm}
  \begin{equation}\label{diagramW2}
  \xymatrix{X\ar@{-->}[r]^{\xi}\ar[d]_{\alpha}&{\w{X}}
    \ar[d]^<<<<{\tilde{\alpha}}\ar@{-->}[r]^{\psi}&{\wi{X}}\ar[r]^{\sigma}&{X_0}\ar[d]_{\alpha_0}
    &\\
  {W\supset S}\ar@{-->}[r]^{\xi_{\s W}}
  &{\w{W}\supset S}\ar[d]^{\pi}\ar@{-->}[r]^<<<<<<<<<<{\psi_{\s W}}&{\wi{W}=\Bl_{q_0,\dotsc,q_r}\pr^4\supset S}\ar[r]^>>>>{\sigma_{\scriptscriptstyle W}}&{\Bl_{q_0}\pr^4\supset S_0}\ar[d]_{\pi_0}\ar[r]^>>>>{\sigma_{q_0}}&{\pr^4\supset A}\ar@{-->}[dl]^{\pi_{q_0}}\\
  &{Y=\Bl_{p_1,\dotsc,p_r}\pr^3\supset B}\ar[rr]^k&&{\pr^3\supset B_0}&
}\end{equation}}
\begin{lemma}\label{A}
Set $A:=\sigma_{q_0}(S_0)\subset\pr^4$. Then
$q_0,\dotsc,q_r\in A$, and at these points $A$ is smooth in case $(a)$, has rational double points of type $A_1$ or $A_2$ in case $(b)$.
\end{lemma}
\begin{proof}
For $i=1,\dotsc,r$ we have $w_i\in S_0$ (see \ref{isowi}) thus $q_i=\sigma_{q_0}(w_i)\in A$. Moreover, since $q_0\neq q_i$, at $q_i$ the surface $A$ is locally isomorphic to $S_0$ at $w_i$, and we deduce the statement from \ref{Sr}.

We are left to prove the behaviour of $A$ at the point $q_0$. 
The map $\pi\colon\w{W}\to Y$ (see \eqref{diagramW2}) is a $\pr^1$-bundle induced by the projection $\pi_{q_0}\colon \pr^4\dasharrow\pr^3$ from $q_0$.  The projection of $A$  is the surface $B_0\subset\pr^3$ which has degree $2$ or $4$ by Lemma \ref{aeb}; in particular $A$ cannot be a plane.
If $A$ is a cone, up to exchanging $q_1$ and $q_2$ we can assume that $q_1$ is not the vertex.

Let us consider now the projection $\pi_{q_1}\colon \pr^4\dasharrow\pr^3$ from $q_1$, and the corresponding $\pr^1$-bundle $\pi_0'\colon W_0':=\Bl_{q_1}\pr^4\to\pr^3$. If $p_0',p_2',\dotsc,p_r'\in\pr^3$ are the images of $q_0,q_2,\dotsc,q_r\in\pr^4$ via $\pi_{q_1}$, we consider $Y':=\Bl_{p_0',p_2',\dotsc,p_r'}\pr^3$ and the composite maps $\wi{W}=\Bl_{q_0,\dotsc,q_r}\pr^4\dasharrow Y'$
and $f'_{\s X}\colon X\dasharrow Y'$ (compare with diagram \eqref{diagramW2});
$\wi{W}\dasharrow Y'$ is an elementary rational contraction, and
$f_{\scriptscriptstyle X}'$ is a rational contraction with  $\rho_X-\rho_{Y'}=2$.
$$
\xymatrix{X\ar@{-->}[r]\ar@{-->}[dr]_{f_{\scriptscriptstyle X}'}&  {\wi{W}}\ar@{-->}[d]\ar[r]^{\sigma_{\scriptscriptstyle W}'}&{W_0'}\ar[d]_{\pi_0'}\ar[r]&{\pr^4}\ar@{-->}[dl]^{\pi_{q_1}}\\
  &   {Y'}\ar[r]&{\pr^3}&}$$

We show that $f_{\s X}'$ is again special (Def.~\ref{defspecial} and \ref{specialrat}) . If $D\subset X$ is a prime divisor different from $E_1'$, then its transform in $\wi{W}$ is a prime divisor, and its image in $Y'$ is either $Y'$, or a prime divisor.  On the other hand
the image of $E_1'\subset X$ in $Y'$ is the transform, in $Y'$, of the projection of $A$ from $q_1$ in $\pr^3$. Since $A$ is not a cone with vertex $q_1$, this projection is a surface.

Therefore
we can  
   replace $f_{\scriptscriptstyle X}$ with $f_{\scriptscriptstyle X}'$, and get the statement for $q_0$ too.
\end{proof}
Recall from Lemma \ref{W} that 
$S\subset \wi{W}$ is the transform of $S_0\subset\Bl_{q_0}\pr^4$ and hence of $A\subset\pr^4$, and
we still denote by $S$ its transform in $W$ and in $\w{W}$. Moreover $\sigma_{|S}\colon S\to A$ is the blow-up of $q_0,\dotsc,q_r\in A$.
\begin{prg}\label{nodelta2}
  We have $\dim\N(S,\w{W})=\rho_W$.

  Indeed $A$ contains $q_0,\dotsc,q_r$, therefore $S\subset\wi{W}$ meets along a curve every exceptional divisor of the blow-up $\wi{W}\to\pr^4$, thus $\N(S,\wi{W})=\N(\wi{W})$. Moreover $S$ is contained in $\dom(\psi_{\s W})$ (see Lemma \ref{W}), therefore $\dim\N(S,\w{W})=\dim\N(S,\wi{W})=\rho_W$, see \cite[Rem.~3.13(1)]{eff}.
\end{prg}  
\begin{lemma}\label{cubic}
Assume that we are in case $(a)$. Then  $A$ is either a cubic scroll,  or a cone over a twisted cubic; moreover $f$ and $f_0$ have some $2$-dimensional fiber.
\end{lemma}  
\begin{proof}
By Lemma \ref{aeb}  in case $(a)$ the surface $B_0\subset\pr^3$ is a quadric, and it is the projection of $A\subset\pr^4$ from the smooth point $q_0$ (see Lemma \ref{A}), thus  $A$ has  degree $3$.
Moreover $A$ cannot be contained in a hyperplane, because the points $q_0,\dotsc,q_r\in A$ are in general linear position (see \ref{linear}). By the classification of projective varieties of minimal degree, we conclude that $A$ is either a cubic scroll, or a cone over a twisted cubic.

In both cases there is a line through $q_0$ contained in $A$, thus $\pi_0$ is not finite on $S_0$, and $f_0$ and $f$ have some $2$-dimensional fiber.
\end{proof}
\begin{prg}
Assume that we are in case $(a)$. If $A$ is a cubic scroll, then $S$ is smooth, because $S\cong\Bl_{q_0,\dotsc,q_r}A$. 
\end{prg}  
\begin{prg}
  Assume that we are in case $(a)$ and that $A$ is a cone over a twisted cubic.  Then by Lemma \ref{A} $q_0,\dotsc,q_r$ are distinct from the vertex $v$ of the cone, and $S\cong\Bl_{q_0,\dotsc,q_r}A$ has one singular point $v$, which is  of type $\frac{1}{3}(1,1)$.   Then $\alpha^{-1}(v)\cong\tilde\alpha^{-1}(v)\cong\pr^2$ by \cite[Th.~on p.~256]{AW}. 
Moreover
   $B_0$ is a quadric cone with vertex $v'\neq p_i$ for $i=1,\dotsc,r$, which gives a node $v'\in B\subset Y$. The fiber $\pi^{-1}(v')$ is the transform of the line $\overline{vq_0}\subset A$, it is contained in $S\subset \w{W}$, and  $f^{-1}(v')$ has two irreducible components, both of dimension $2$, given by $\tilde\alpha^{-1}(v)$ and by the closure of $\tilde\alpha^{-1}(\pi^{-1}(v')\smallsetminus\{v\})$.
\end{prg}
  \begin{lemma}\label{b}
Assume that we are in case $(b)$. Then $\alpha_0$ has no $2$-dimensional fiber, $\alpha$ and $\tilde\alpha$ are of type $(3,2)^{\sm}$,  $S$ is smooth,  $\Sing(S_0)=\{w_1,\dotsc,w_r\}$, and $\Sing(A)=\{q_0,\dotsc,q_r\}$.
\end{lemma}
\begin{proof}
  Let us consider both factorizations of $f_0$ in elementary contractions: \stepcounter{thm}
\begin{equation}\label{alpha'}  \xymatrix{{X_0}\ar[r]^{\alpha_0'}\ar[d]_{\alpha_0}\ar[dr]^{f_0}&{W_0'}\ar[d]^{\pi_0'}\\
    {W_0}\ar[r]_{\pi_0}&{\pr^3}
  }
  \end{equation}
  By Lemma \ref{bigdiagram} in case $(b)$ the situation is symmetric, and both $\sigma_{\scriptscriptstyle W}$ and $\sigma_{W}'$ blow-up $r$ points. Thus we can apply Lemma \ref{psi} and deduce that both $\pi_0$ and $\pi_0'$ have no $2$-dimensional fibers. In turn this implies that $\alpha_0$ and $\alpha_0'$ 
  do not have $2$-dimensional fibers either, because if $F$ were such a fiber (for instance for $\alpha_0$), then $\alpha_0'$ would be finite on $F$ and $\alpha_0'(F)$ would be a $2$-dimensional fiber of $\pi_0'$.

  By Lemma \ref{onedim} every fiber of $f$ over $G_i$ has dimension one, and $\pi^{-1}(G_i)=\tilde\alpha(D_i)$, therefore every fiber of $\tilde\alpha$ over $S\cap \tilde\alpha(D_i)\subset\w{W}$ has dimension one. Moreover $S$ is contained in $\dom(\psi_{\s W})$ (Lemma \ref{W}), and $\sigma_{\s W}$ is an isomorphism between $\wi{W}\smallsetminus \cup_iD_{{\s W},i}$ and $W_0\smallsetminus \{w_1,\dotsc,w_r\}$, hence $\tilde\alpha$ and $\alpha_0$ are isomorphic over $S\smallsetminus \cup_i\tilde\alpha(D_i)$ and $S_0\smallsetminus \{w_1,\dotsc,w_r\}$ (see diagram \eqref{diagramW2}). We conclude that $\tilde\alpha$ has only one-dimensional fibers, and the same holds for $\alpha$ (see Lemma \ref{rays}).

  Then $\alpha$ and $\tilde\alpha$ are of type $(3,2)^{\sm}$, and $S$ is smooth, by Th.~\ref{32}.
  This also implies that 
 $S_0\smallsetminus\{w_1,\dotsc,w_r\}$ is smooth, and finally $A\smallsetminus \{q_0,\dotsc,q_r\}$ is isomorphic, via $\sigma_{q_0}$, to an open subset of  $S_0\smallsetminus\{w_1,\dotsc,w_r\}$, hence it is smooth too.
\end{proof}
  \begin{lemma}
    Assume that we are in case $(b)$.   Then
$S$ is a smooth K3 surface, $B$ is a nodal K3 surface, and $\pi_{|S}\colon S\to B$ is birational and may contract some smooth fiber of $\pi$ to nodes $b\in B$. This happens if and only if $\dim f^{-1}(b)=2$; in this case $f^{-1}(b)\cong\pr^1\times\pr^1$.
\end{lemma}
\begin{proof}
Recall from \ref{isowi} that $\pi_{0|S_0}\colon S_0\to B_0$ is birational, with exceptional locus the fibers of $\pi_0$ contained in $S_0$, and it is an isomorphism around $w_1,\dotsc,w_r$.
  
  By Lemma \ref{aeb} the surface $B_0\subset\pr^3$ is a quartic with isolated singularities, so that it is a normal K3 surface. 
  If $\pi_0$ is finite on $S_0$, then $S_0\cong B_0$, $\Sing(B_0)=\{p_1,\dotsc,p_r\}$ by Lemma \ref{b}, and similarly $S\cong B$ are smooth K3 surfaces, so we have the statement.

Suppose that
there is a fiber $F_0:=\pi_0^{-1}(b)$ contained in $S_0$; note that $F_0\cong\pr^1$. Then $b\neq p_i$ and $w_i\not\in F_0$ for every $i$, hence 
$F_0\subset (S_0)_{\reg}$ by Lemma \ref{b}. Since $K_{B_0}=0$, we have $K_{S_0}\cdot F_0=mF_0^2$ for some $m\in\Z$, $F_0^2<0$, and by the genus formula $-2=(1+m)F_0^2$, thus $F_0^2\in\{-1,-2\}$.

  We claim that $F_0^2=-2$. By contradiction, if $F_0^2=-1$, then $-K_{S_0}\cdot F_0=1$, while
$-K_{W_0}\cdot F_0=2$, thus
$(\det\ma{N}_{S_0/W_0})\cdot F_0=1$. By Lemma \ref{b} and Th.~\ref{32}, in $W_0\smallsetminus\{w_1,\dotsc,w_r\}$ $\alpha_0$ is just the blow-up of the smooth surface $S_0\smallsetminus\{w_1,\dotsc,w_r\}$, hence
 $R:=\alpha_0^{-1}(F_0)\cong\pr_{F_0}(\ma{N}_{S_0/W_0}^{\vee})_{|F_0}$. By Rem.~\ref{sections} we get $R\cong\mathbb{F}_e$ with $e\in\Z_{>0}$ odd.  Hence $(\alpha'_0)_{|R}$ must be birational  (see diagram \eqref{alpha'} and the proof of Lemma \ref{b}), and $\alpha_0'(R)$ is a $2$-dimensional fiber of $\pi_0'$, a contradiction.
 Therefore $F_0^2=-2$ and $b$ is a node for $B_0$; moreover $m=0$, $f_0^{-1}(b)\cong\pr^1\times\pr^1$, and $K_{S_0}=0$.

 We have a diagram of birational maps:
 $$\xymatrix{S\ar[r]^{\sigma_{{\s W}|S}}\ar[d]_{\pi_{|S}}&{S_0}\ar[d]^{\pi_{0|S}}\ar[r]^{\sigma_{q_0|S_0}}&A\\
B\ar[r]^{k_{|B}}&{B_0}&
}$$
where on the first row $S$ is smooth, 
$S_0$ has rational double points of type $A_1$ or $A_2$ at $w_1,\dotsc,w_r$  (Lemma \ref{b} and \ref{Sr}), and $A$ has rational double points of type $A_1$ or $A_2$ at $q_0,\dotsc,q_r$. Moreover
$B_0$ has rational double points of type $A_1$ or $A_2$ at $p_1,\dotsc,p_r$ and nodes at the points $b_j$ such that $\pi_0^{-1}(b_j)\subset S_0$, and
the map $k_{|B}\colon B\to B_0$ resolves $p_1,\dotsc,p_r$ (\ref{BG}), while it is an isomorphism around the nodes $b_j$. Therefore $B$ can have at most nodes at the inverse images of $b_j$, and $\pi_{|S}\colon S\to B$ is a minimal resolution of singularities. Since $B\in|-K_Y|$ (Lemma \ref{aeb}), $B$ is a nodal K3 surface, and $S$ is a smooth K3 surface.

Conversely, if $b'\in B$ is such that $\dim f^{-1}(b')=2$, then $b'\not\in\Exc(k)$ (Lemma \ref{onedim}) and if $b'':=k(b')$, then $\dim f_0^{-1}(b'')=2$, but both $\pi_0$ and $\alpha_0$ have only one-dimensional fibers (Lemma \ref{b}). Therefore we must have $\pi_0^{-1}(b'')\subset S_0$ and $b''=b_j$ for some $j$.
\end{proof}
 \begin{lemma}\label{final}
 Assume that we are in case $(b)$.   Then $r=4$ and $\rho_X=7$, and
 $A\subset\pr^4$ is a (singular) sextic K3 surface.
 \end{lemma}
 \begin{proof}
   Since $A$ has a  double point at $q_0$ by Lemma \ref{A}, and the projection of $A$ from $q_0$ is a quartic surface $B_0\subset \pr^3$ (Lemma \ref{aeb}),
   $A$ must have degree $6$. Then $A$ is contained in  a unique quadric hypersurface $Q$, it is the complete intersection of $Q$ with a cubic hypersurface $M$ (see for instance \cite[Ex.~VIII.14]{beauvillesurf}), and has trivial canonical class.

   We note that $Q$ must be smooth at $q_i$ for every $i=0,\dotsc,r$.
   Indeed if $Q$ is singular at $q_i$, then it is a cone, and its projection from $q_i$ is a quadric in $\pr^3$. On the other hand $A$ cannot be a cone, because  $\Sing(A)=\{q_0,\dotsc,q_r\}$ (see Lemma \ref{b}).
Then, as observed in the proof of Lemma \ref{A}, the projection $\pi_{q_i}\colon\pr^4\dasharrow\pr^3$ induces a different special rational contraction $X\dasharrow Y^i$ with $\rho_X-\rho_{Y^i}=2$, to which our results apply. In particular the projection of $A$ from $q_i$ must be again a quartic surface in $\pr^3$, but this is impossible 
as $A\subset Q$. Therefore $Q$ must be smooth at $q_i$.

Since $A$ has double points at $q_0,\dotsc,q_r$ (Lemma \ref{A}), in turn the cubic $M$ must have a double point at each $q_i$.

Suppose by contradiction that $r>4$, let $q\in M$ be a general point,
and let $\Gamma\subset\pr^4$ be the rational normal quartic through $q_0,q_1,\dotsc,q_5,q$. Then $\Gamma\cdot M=12$, and $\Gamma$ and $M$ intersect with multiplicity $\geq 1$ in $q$ and $\geq 2$ in $q_0,\dotsc,q_5$, thus $\Gamma\subset M$.
Moreover $\Gamma$ intersects the quadric $Q$ in $q_0,\dotsc,q_5$ plus two additional points $a$ and $b$, that belong to $A$.

Let $\wi{\Gamma}\subset\wi{W}$,  $\Gamma_W\subset W$, and $\Gamma_X\subset X$ be the transforms of $\Gamma$. Then
$-K_{\wi{W}} \cdot \wi{\Gamma} = 20-18=2$, and
$\wi{\Gamma}$ intersects $S$ in two points, corresponding to $a$ and $b$.
Recall that $S\subset\dom(\zeta^{-1})$ (see Lemma \ref{W}), thus $\Gamma_W$ still intersects $S$ in two points. Moreover $-K_W\cdot\Gamma_W\leq -K_{\wi{W}} \cdot \wi{\Gamma} = 2$ (see \cite[Rem.~3.6]{eff}),
 thus $-K_X\cdot\Gamma_X\leq 0$, a contradiction.
\end{proof}
This concludes the proof of Th.~\ref{B}.
\end{proof}

\section{The case $\delta_X=2$ and the elementary case}\label{grenoble}
\noindent In this section we treat Fano $4$-folds with Lefschetz defect $2$, and Fano $4$-folds with an elementary rational contraction onto a $3$-fold.
For the case where $\delta_X=2$, we show the following more refined version of Th.~\ref{delta2}.
\begin{thm}\label{delta2detail}
  Let $X$ be a smooth Fano $4$-fold with $\delta_X=2$. Then $3\leq \rho_X\leq 6$, and one of the following holds:
  \begin{enumerate}[$(i)$]
 \item there is a special rational contraction $X\dasharrow Y$ onto a smooth $3$-fold with $\rho_X-\rho_Y=2$;  
 \item there is a  quasi-elementary contraction  $X\to S$ where either $S\cong\pr^2$ and $\rho_X=4$, or $S\cong\pr^1\times\pr^1$ or $\mathbb{F}_1$   and $\rho_X=5$.
   \end{enumerate}
\end{thm}  
A contraction $X\to S$ is 
\emph{quasi-elementary} if for every fiber $F\subset X$ we have $\dim\N(F,X)=\rho_X-\rho_S$; we refer the reader to \cite[\S 3]{fanos} for more details.

The bound $\rho_X\leq 6$ improves the previous bound $\rho_X\leq 12$, see \cite[Th.~2.12 and references therein]{blowup}, and is sharp, by the example $(\Bl_{2\pts}\pr^2)^2$ (see Rem.~\ref{deltaproduct}). On the other hand we are not aware of other examples of Fano $4$-folds with $\delta_X=2$ and $\rho_X=6$, while for  $\rho_X=4, 5$ we provide several (known and new) examples in \S\ref{ex_delta2} and \S\ref{other}.   Fano $4$-folds with $\delta_X=2$ and $\rho_X=3$ are classified and studied in \cite{saverio}, there are 28 families.
\begin{proof}[Proof of Th.~\ref{delta2detail}]
  By \cite[Th.~5.2 and its proof]{codimtwo}
one of the following holds:
  \begin{enumerate}[$(1)$]
  \item there exist 
    a SQM $X\dasharrow \w{X}$, a special, $K$-negative contraction $f\colon\w{X}\to Y$,  and
a  prime divisor $D\subset \w{X}$, such that $Y$ is smooth,
 $\dim Y=3$, $\rho_X-\rho_Y=2$,  $\codim\N(D,\w{X})=2$,
 and $f({D})=Y$;
 \item there is a  quasi-elementary contraction   $\psi\colon X\to S$ with $\dim S=2$ and $\rho_X-\rho_S=3$.
   \end{enumerate}

   Assume first that we are in $(1)$, and suppose by contradiction that $\rho_X\geq 7$. Then
    the study made in Section \ref{relrho2} applies to $X\dasharrow Y$; let us
    consider the factorization of $f$ as in Lemma \ref{rays} and \ref{switch}: $$\xymatrix{  {\w{X}} \ar@/^1pc/[rr]^{f}\ar[r]_{\tilde\alpha}&{\w{W}}\ar[r]_{\pi}&{Y}
}$$
where $\tilde\alpha$ is an elementary contraction of type $(3,2)$ with
$E_1=\Exc(\tilde\alpha)\subset\w{X}$ and
$S=\tilde\alpha(E_1)\subset \w{W}$.

We show that $\N(S,\w{W})\subsetneq\N(\w{W})$, which contradicts
\ref{nodelta2}.

Since $f({D})=Y$, we have $f_*(\N({D},\w{X}))=\N(Y)$, but $\dim\N({D},\w{X})=\dim\N(Y)=\rho_X-2$, therefore $\N({D},\w{X})\cap\ker f_*=\{0\}$, in particular $\NE(\tilde\alpha)\not\subset\N({D},\w{X})$. This implies that $\tilde\alpha$ must be finite on ${D}$.

   If ${D}\cap E_1\neq\emptyset$, then ${D}\cdot\NE(\tilde\alpha)>0$, so that ${D}$ meets every non-trivial fiber of $\tilde\alpha$, and the prime divisor
    $\tilde\alpha({D})$ contains $S$. We have
$\N(\tilde\alpha({D}),\w{W})=\tilde\alpha_*(\N({D},\w{X}))$, thus $\dim\N(\tilde\alpha({D}),\w{W})\leq \dim\N({D},\w{X})=\rho_X-2=\rho_W-1$, and  $\N(S,\w{W})\subseteq \N(\tilde\alpha({D}),\w{W})\subsetneq\N(\w{W})$. If instead ${D}\cap E_1=\emptyset$, then $\tilde\alpha({D})\cap S=\emptyset$, thus  $\N(S,\w{W})\subseteq\tilde\alpha({D})^{\perp}\subsetneq\N(\w{W})$ (see Rem.~\ref{disjoint}).

  We conclude that in case $(1)$ we have $\rho_X\leq 6$, therefore we get $(i)$.

\medskip
  
Assume now that we are in $(2)$.
  By \cite[Th.~1.1]{fanos}  $S$ is a smooth del Pezzo surface, so if $\rho_S\leq 2$ we have $(ii)$. If instead 
  $\rho_S\geq 3$, then again by \cite[Th.~1.1]{fanos} we have $X\cong S\times F$ where $F=\Bl_{2\pts}\pr^2$, and $\delta_X=\max\{\rho_S-1,\rho_F-1\}$ (see Rem.~\ref{deltaproduct}). Hence
  $\rho_S-1\leq\delta_X=2$, which gives $\rho_S=3$ and $S\cong F$, and we have again $(i)$ with the contraction $S\times S\to \pr^1\times S$.
\end{proof}
\begin{corollary}\label{poznan}
Let $X$ be a smooth Fano $4$-fold that is not isomorphic to a product of surfaces. If $\delta_X\geq 2$,
then $\rho_X\leq 6$.
\end{corollary}
\begin{proof}
 Since $X$ is not a product of surfaces,  by
 Th.~\ref{deltageq4} and \ref{delta=3} we have $\delta_X\leq 3$, and if $\delta_X=3$, then $\rho_X\leq 6$. Moreover if $\delta_X=2$, then again $\rho_X\leq 6$ by  Th.~\ref{delta2}.
 \end{proof}
 
In the case of an elementary rational contraction onto a $3$-fold, we show the following.
\begin{thm}\label{elem3fold}
Let $X$ be a smooth Fano $4$-fold that is not isomorphic to a product of surfaces, and having an elementary rational  contraction
$X\dasharrow Y$  with $\dim Y=3$. 
Then $\rho_X\leq 9$.
\end{thm}
The bound is sharp, as shown by the Fano model of $\Bl_{8\pts}\pr^4$, see Ex.~\ref{Fanomodel}.
\begin{proof}[Proof of Th.~\ref{elem3fold}]
 Since $X$ is not a product of surfaces,  by
 Th.~\ref{deltageq4} and \ref{delta=3} we have $\delta_X\leq 3$, and if $\delta_X=3$, then $\rho_X\leq 6$. Therefore we can assume that $X$ has Lefschetz defect $\delta_X\leq 2$.

  We follow \cite[\S 4, in particular the proof of Th.~4.4]{eff}; the strategy is the same as the beginning of the proof of Th.~\ref{B}.
We know that $Y$ has at most isolated, locally factorial, canonical singularities. 
  Moreover $Y$ is log Fano, $-K_Y$ is big, and if $g\colon Y\to Y_0$ is an elementary contraction of fiber type, then $-K_Y\cdot\NE(g)>0$. By \cite[Lemma 4.5]{eff}, if  $g\colon Y\to Y_0$ is a small elementary contraction, then $K_Y\cdot\NE(g)=0$.

\bigskip

Assume now that $\rho_X\geq 6$.
Then by \cite[Lemma 4.6]{eff}, if  $g\colon Y\to Y_0$ is a divisorial elementary contraction, then $-K_Y\cdot\NE(g)>0$ and $g$ is the blow-up of a smooth point (note that in \emph{loc.\ cit.} the map $X\dasharrow Y$ is assumed to be non-regular, but this is used only to deduce that $\delta_X\leq 2$, which we already know).
In particular $Y$ is weak Fano, and as in \ref{properties} we show that if $Y$ is not Fano, then its anticanonical map is small.

As in \cite[p.~622]{eff} we consider all divisorial extremal rays of $\NE(Y)$ and get a map
$$k\colon Y\la Y_0$$
which is the blow-up of $r$ distinct smooth points.  Moreover $Y_0$ is weak Fano, has the same singularities as $Y$, $\rho_Y=\rho_{Y_0}+r$, and $(-K_Y)^3=(-K_{Y_0})^3-8r$.

As in the proof of Lemma \ref{blowups}
we show that, up to increasing the number $r$ of blown-up points, and up to replacing $Y$ and $Y_0$ with SQM's, we can reduce to the case where $\rho_{Y_0}\leq 2$ and, if $\rho_{Y_0}=2$, then $Y_0$ has two distinct elementary rational contractions of fiber type. Moreover as in Lemma \ref{milano} we show that, if $Y_0$ is not Fano, then its anticanonical map is small, with exceptional locus contained in $(Y_0)_{\reg}$.

If $\rho_{Y_0}=1$, then $Y_0$ is Fano, and Lemma \ref{bound} yields 
$-K_{Y_0}^3\leq 64$. Thus
$$0<-K_Y^3=-K_{Y_0}^3-8r\leq 64-8r$$
which gives $r\leq 7$, $\rho_Y=r+\rho_{Y_0}\leq 8$, and $\rho_X=\rho_Y+1\leq 9$.
If instead $\rho_{Y_0}=2$, then by Lemma \ref{ou} and Prop.~\ref{sing} we have either $Y_0\cong\pr^2\times\pr^1$, or $Y_0\cong\pr_{\pr^1}(\ol\oplus\ol(1)^{\oplus 2})$, or 
$-K_{Y_0}^3\leq 48$.  The blow-up of $\pr^2\times\pr^1$ or $\pr_{\pr^1}(\ol\oplus\ol(1)^{\oplus 2})$ at a point (not lying on the curve of anticanonical degree zero) has a divisorial elementary contraction of type $(2,1)$, which is excluded as in Lemma \ref{milano}. Therefore $-K_{Y_0}^3\leq 48$,
and as above we get $r\leq 5$, $\rho_Y\leq 7$, and $\rho_X\leq 8$.
  \end{proof}
  As in the proof of Lemma \ref{table}, we also get the following.
  \begin{corollary}\label{tableelem}
Let $X$ be a smooth Fano $4$-fold  and
$X\dasharrow Y$ an elementary rational  contraction with $\dim Y=3$. Assume that $\rho_X\geq 6$,
$\delta_X\leq 2$, and that $Y$ is smooth.

Then $Y$ is weak Fano and, up to flops, $Y\cong\Bl_{r\,\pts}Y_0$ where $Y_0$ is one of the following.
{\em   $$\begin{array}{||c|c|c|c|c||}
\hline\hline
Y_0   & -K_{Y_0}^3&  \rho_{Y_0} & & \\
\hline\hline

       \pr^3 & 64 &  1 &  \text{Fano}  & \rho_X\leq 9, r=\rho_X-2    \\

       \hline

       \pr_{\pr^2}(T_{\pr^2}) & 48 & 2 & \text{Fano}  & \rho_X\leq 8, r=\rho_X-3  \\

  \hline

       \text{\em\cite[Th.~3.6(1)]{jahnkepeternell}} & 40 & 2 &  \text{weak Fano}  & \rho_X\leq 7, r=\rho_X-3 \\

        \hline

       \text{linear section of }G(2,5) & 40 & 1 & \text{Fano}   & \rho_X=6, r=4   \\

 \hline

           \text{\cite[Th.~3.5(3)]{jahnkepeternell}} & 32 & 2 & \text{weak Fano}    & \rho_X=6, r=3   \\

       \hline

       \text{divisor of degree $(1,2)$ in }\pr^2\times\pr^2 & 30 & 2 & \text{Fano} &  \rho_X=6, r=3  \\ 
       
\hline\hline
\end{array}$$}
    \end{corollary}

\begin{proof}[Proof of Th.~\ref{general}]
Since $X$ is not a product of surfaces,  by
Cor.~\ref{poznan}
if $\delta_X\geq 2$ then $\rho_X\leq 6$. Therefore we can assume that $X$ has Lefschetz defect $\delta_X\leq 1$.

  By Prop.~\ref{factor} we can assume that the rational 
  contraction $X\dasharrow Y$ is special.
Then  $\rho_X-\rho_Y\in\{1,2\}$ by Lemma \ref{rhodelta}, because
  $\delta_X\leq 1$. Thus the statement follows from Th.~\ref{elem3fold} and \ref{B}.
\end{proof}
\begin{proof}[Proof of Th.~\ref{summary}]
Since $\rho_X\geq 7$ and $X$ is not a product of surfaces,  Cor.~\ref{poznan}  implies that $\delta_X\leq 1$. As in the proof above, we see that there is a special  rational 
  contraction $X\dasharrow Y$ with $\dim Y=3$ and  $\rho_X-\rho_Y\in\{1,2\}$. Then $Y$ is weak Fano with at most isolated, locally factorial, and canonical singularities, by the proofs of Th.~\ref{elem3fold} and \ref{B} (in particular \ref{logFano} and \ref{weakFano}). 
\end{proof}     
\section{New families and examples}\label{examples}
\noindent In this section we construct several new families of Fano $4$-folds, with Picard number between $2$ and $7$, all having a rational contraction onto a $3$-fold.
Our strategy is inspired by the classification result in Th.~\ref{Bintro} and more generally by the study in Section \ref{relrho2} of Fano $4$-folds with a special rational contraction $X\dasharrow Y$ with $\rho_X-\rho_Y=2$.

In \S\ref{newa} and \S\ref{newb} respectively, we show that cases $(i)$ and $(ii)$ of Th.~\ref{Bintro} do happen for $\rho_X\in\{3,\dotsc,7\}$. 
This proves Prop.~\ref{newexintro}, and leaves open case $(iii)$, as follows.
\begin{question}\label{Q1}
Let $r\in\{0,\dotsc,6\}$ and let $W$ be the Fano model of $\Bl_{q_0,\dotsc,q_r}\pr^4$, for $q_i$ general points (see Ex.~\ref{Fanomodel}).
Let $A\subset\pr^4$ be  a general cone over a twisted cubic containing $q_0,\dotsc,q_r$.
Let $S\subset W$ be the transform of $A$, and $X\to W$ the blow-up of $S$.
Is $X$ is a smooth Fano $4$-fold?
 \end{question}  
 \noindent We also leave open the existence of case $(i)$ of Th.~\ref{Bintro} for $\rho_X=8,9$, see Question \ref{Q2}.
 
 We also give other constructions.
 In \S\ref{ex_delta2} we adapt  case $(a)$ from Section \ref{relrho2}  to get examples of Fano $4$-folds with Lefschetz defect $2$ and $\rho\in\{3,4,5\}$. Then in \S\ref{newel} we show that the examples from \S\ref{newa} have a different blow-down to a smooth Fano $4$-fold $Z$ with $\rho_Z\in\{2,\dotsc,6\}$, and we give an explicit description of $Z$ thanks to a result in \cite{langer}. Finally in \S\ref{other} we recall some other known examples.
 
 For all the new families we compute the main numerical invariants,  using Lemma \ref{tautological} and \cite[Lemma 6.25]{vb}, see Tables \ref{t2}, \ref{t3}, \ref{t4}, \ref{t5}.
 
In most cases these new families are obtained as blow-ups of the Fano model of $\Bl_{\pts}\pr^4$ along a surface, so let us start by recalling this example.
\begin{example}[the Fano model of $\Bl_{\pts}\pr^4$]\label{Fanomodel}
  Let $\sigma\colon \wi{W}\to\pr^4$  be the blow-up at $r+1$ general points $q_0,\dotsc,q_r$, with $r\in\{0,\dotsc,7\}$. For $r\geq 1$ the $4$-fold $\wi{W}$ is not Fano, but there is a SQM $\wi{W}\dasharrow W$ such that $W$ is smooth and Fano, with $\rho_W=r+2\in\{2,\dotsc,9\}$; we refer to $W$ as the Fano model of $\Bl_{r+1\,\pts}\pr^4$.

  The blow-up $\wi{W}$ contains exceptional lines given by the transforms of the lines $\overline{q_iq_j}$, and (for $r=6,7$) of the rational normal quartics through $7$ points among $q_0,\dotsc,q_r$; these curves are the indeterminacy locus of the map $\wi{W}\dasharrow W$ (see Lemma \ref{SQMFano}$(a)$).

 We note that $W$ is toric if and only if $r\leq 4$; in the classification of toric Fano $4$-folds  in \cite{bat2}, these are $B_3$, $D_9$, $M_1$, 3.5.8(iii), and 3.5.8(ii). We refer the reader to \cite{fanomodel} for more details on the case $r=6$, and to \cite{vb} for the case $r=7$. 
 For $r\geq 8$ it has been shown by Mukai \cite{mukaiXIV,mukaiTshaped} that $\wi{W}$ is not a Mori dream space, and this is a necessary condition to have a SQM that is smooth and Fano.

 Consider the projection $\pr^4\dasharrow \pr^3$ from  $q_0$ (an analogous construction can be made for the other points $q_j$), let $p_i\in\pr^3$ be the image of $q_i$ for $i=1,\dotsc,r$, and
 let $k\colon Y\to\pr^3$ be the blow-up of $p_1,\dotsc,p_r$.
 Then the composition $W\dasharrow Y$ is an elementary rational contraction.
 More precisely, there is an intermediate SQM $\wi{W}\dasharrow\w{W}$, that flips the transforms of the lines $\overline{q_0q_i}$ for $i=1,\dotsc,r$, such that $\pi\colon\w{W}\to Y$ is a $\pr^1$-bundle (see for instance \cite[Rem.~6.12]{vb}).
 
 For $r=7$ we have $\rho_W=9$, thus this example shows that the bounds on $\rho$ in Th.~\ref{general} and \ref{elem3fold} are sharp.

 We report in Table \ref{t1} the main numerical invariants of $W$.
 \begin{table}[h]
 $\begin{array}{||c|c|c|c|c|c|c|c||}
   \hline\hline
   
r   & \rho_W & K_{W}^4&  K_W^2\cdot c_2(W) & b_4(W)=h^{2,2}(W) & b_3(W)& h^0(W,-K_W) & \chi(T_W)\\
      \hline\hline

         0 & 2 & 544 & 232 & 2 & 0 & 111  & 20\\

        \hline

       1 & 3 & 464 & 212 & 4 &  0&96 &16 \\

       \hline

     2 & 4 & 385 & 190 & 7 & 0&81  &12\\

  \hline

      3 & 5 & 307 & 166 & 11 &0 &66 &  8\\

        \hline

      4  & 6 & 230 & 140 & 16 & 0&51 & 4\\

        \hline

      5  & 7 & 154& 112& 22 & 0& 36& 0\\

        \hline

      6  & 8 & 80&80  &30  & 0&21 & -4\\

        \hline

         7  & 9 & 13 &  34& 45 & 0& 6 & -8\\
      
\hline\hline
  \end{array}$

\bigskip
  
  \caption{Numerical invariants of the Fano model $W$ of $\Bl_{r+1\,\pts}\pr^4$}\label{t1}
\end{table}
  \end{example}
  \subsection{Set-up for blow-ups of $W$}\label{setup}
\noindent  In \S\ref{newa}, \ref{newb}, and \ref{ex_delta2} we are going to  construct families of Fano $4$-folds $X$ as  blow-ups of $W$ along a smooth surface; we set up here a common notation.

  Let  $W$ be the Fano model of $\wi{W}=\Bl_{q_0,\dotsc,q_r}\pr^4$;  we keep the same notation as in Ex.~\ref{Fanomodel}. Let $A\subset\pr^4$ be an irreducible surface and $S\subset\wi{W}$ its transform. In all our examples $S$ will be smooth and  contained in the open subset where the map $\wi{W}\dasharrow W$ is an isomorphism; we still denote by $S\subset W$ its transform.

  Let $\alpha\colon X\to W$ be the blow-up of $S$, and $E\subset X$ the exceptional divisor.

  We denote by $H$ the transform of $\sigma^*\ol_{\pr^4}(1)$ in $W$, and by $H_X$ its pullback in $X$.

  As in  Ex.~\ref{Fanomodel} consider the projection $\pr^4\dasharrow\pr^3$ from $q_0$ and the associated 
  $\pr^1$-bundle $\pi\colon \w{W}\to Y=\Bl_{p_1,\dotsc,p_r}\pr^3$.
  Let $k\colon Y\to\pr^3$ be the blow-up map, $G_i\subset Y$ the exceptional divisor over $p_i$, and $H_Y:=k^*\ol_{\pr^3}(1)$.
 We assume that $A$ is not a cone with vertex $q_0$; 
 let $B_0\subset \pr^3$ be the projection of $A$ (so that $\dim B_0=2$), and $B\subset Y$ its transform.
We still denote by $S\subset \w{W}$ the transform of $A$; then $B=\pi(S)$.

Let $\tilde\alpha\colon\w{X}\to\w{W}$ be the blow-up of $S$, and $f:=\pi\circ\tilde{\alpha}\colon\w{X}\to Y$. The induced map $X\dasharrow \w{X}$ is a SQM.
{\small
  \stepcounter{thm}
\begin{equation}\label{diagrama}
\xymatrix{X\ar[d]_{\alpha}\ar@{-->}[r]&{\w{X}}\ar@/^1pc/[dd]^<<<<<{f}\ar[d]_{\tilde\alpha}&&\\
W  \ar@{-->}[r]&  {\w{W}}\ar@{-->}[r]\ar[d]_{\pi}&{\wi{W}}\ar[r]^{\sigma}&{\pr^4}\ar@{-->}[d]\\
&Y\ar[rr]^k&&{\pr^3}
}\end{equation}}
\begin{remark}\label{GAEL}
  The composition $X\dasharrow Y$ is a special rational contraction with $\rho_X-\rho_Y=2$.
  In particular we deduce from Th.~\ref{B} that  $X$ is not Fano when $\rho_X> 9$ and $r >6$.
  \end{remark}

Let $\wi{D}_i\subset\wi{W}$ be  the exceptional divisor of the blow-up $\sigma\colon\wi{W}\to\pr^4$ over $q_i$, and $D_i\subset{W}$, $\w{D}_i\subset \w{W}$ its transforms.

Let $T:=\pi^{-1}(B)\subset\w{W}$ and $\w{T}\subset \w{X}$ its transform;
since $S$ is a smooth surface contained in $T$, we have $\w{T}\cong T$. Note that $T$ is the transform of the cone in $\pr^4$ over $A$ with vertex $q_0$.
In all our examples, 
 $T$ will be  contained in the open subset where 
the birational map $\w{W}\dasharrow W$ is an isomorphism. Similarly, 
$\w{T}$ is contained in the open subset where the map $\w{X}\dasharrow X$ is an isomorphism; we still denote by $T\subset W$ and $\w{T}\subset X$ the transforms.

Suppose now that $r\leq 4$; then $W$ is toric and $\w{W}=\pr_Y(\ma{E})$ with $$\ma{E}=\ol_Y\oplus\ol_Y\bigl(H_Y-\sum_{i=1}^rG_i\bigr).$$ 
Let moreover $J\subset\w{W}$ be the transform of a hyperplane $J_0\subset\pr^4$ through $q_1,\dotsc,q_r$; then $J$ is the section of $\pi$ corresponding to the projection  $\ma{E}\twoheadrightarrow\ol_Y(H_Y-\sum_{i=1}^rG_i)$, and
$\eta=\ol_{\w{W}}(J)=\ol_{\w{W}}(\w{H}-\sum_{i=1}^r\w{D}_i)$
 is the tautological class, where $\w{H}$ is the transform of $\sigma^*\ol_{\pr^4}(1)$; moreover $\eta_{|J}\cong \ol_Y(H_Y-\sum_{i=1}^rG_i)$.

We also have $T=\pr_B(\ma{E}_{|B})$, with tautological class 
  $\eta_{|T}=\ol_T(J_{|T})$, and we set $\pi_T:=\pi_{|T}\colon T\to B$.

\subsection{New families from case $(a)$}\label{newa}
\noindent We keep the notation as in \S\ref{setup}, and assume that $r\leq 6$ (see Rem.~\ref{GAEL}).
Let $A\subset\pr^4$ be a general cubic rational normal scroll containing
$q_0,\dotsc,q_r$ (such a scroll exists, see for instance \cite[Ex.~A]{izzet});
 recall that $A\cong\mathbb{F}_1$. 
 Then $S\subset\wi{W}$ is a del Pezzo surface with $\rho_S=r+3$, and it is disjoint from the transforms of the lines $\overline{q_iq_j}$ and, for $r=6$, of the rational normal quartic through $q_0,\dotsc,q_6$. 
 We will show the following.
   \begin{proposition}\label{a}
  For $r\in\{0,\dotsc,4\}$, $X$ is Fano with $\rho_X=r+3\in\{3,\dotsc,7\}$.  
\end{proposition}
\begin{question}\label{Q2}
  Is $X$ Fano for $r=5,6$?
\end{question}
The surface $B_0\subset \pr^3$ is a smooth quadric surface
  containing $p_1,\dotsc,p_r$,   $B$ is a smooth del Pezzo surface with $\rho_B=r+2$, and $-K_Y=2B$. Moreover
$\pi_{|S}\colon S\to B$ is the blow-up of a point $y\in B$, with exceptional curve the transform $\Gamma$ of the line in $A$ containing $q_0$. \label{fiberb} Every fiber of $f$ over $Y\smallsetminus\{y\}$ is one-dimensional, while $f^{-1}(y)=\tilde{\alpha}^{-1}(\Gamma)\cong\mathbb{F}_1$.

Note that  $T\subset\w{W}$ is the transform of the quadric cone $Q\subset\pr^4$, containing $A$, with vertex  $q_0$. Then $T$ is disjoint from the transforms of the lines $\overline{q_iq_j}$ for $1\leq i<j\leq r$,  thus it is contained in the open subset where the map $\w{W}\dasharrow W$ is an isomorphism.
\begin{lemma}\label{Ta}
For $r\in\{0,\dotsc,4\}$, $-K_{X|\w{T}}$ is ample.
\end{lemma}
\begin{proof}
  We  treat the case $r>0$, the case $r=0$ being similar and simpler.
  We have  $-K_{X|\w{T}}\cong -K_{\w{X}|\w{T}}$, thus it is enough to show that $-K_{\w{X}|\w{T}}$ is ample. 
  
 Let us fix some notation in $B$. We have $B=\Bl_{p_1,\dotsc,p_r}B_0\cong\Bl_{r+1\,\pts}\pr^2$. We denote by $e_2,\dotsc,e_r\subset B$ the exceptional curves over $p_2,\dotsc,p_r\in B_0$, so that $e_i$ are lines in $G_i\cong\pr^2$. Then we denote by $e_0,e_1\subset B$ the transforms of the two lines through $p_1$ in the quadric $B_0$, and by $C_{G_1}\subset B$ the exceptional curve over $p_1$, again a line in $G_1$. There is a birational map $B\to\pr^2$ contracting $e_0,e_1,\dotsc,e_r$, let $h\in\Pic(B)$ be the pullback of $\ol_{\pr^2}(1)$. Then $C_{G_1}\sim h-e_0-e_1$.

 We have $\ol_{B_0}(1)=\ol_{B_0}(k(e_0)+k(e_1))$,
  $H_{Y|B}=(k_{|B})^*(k(e_0)+k(e_1))=e_0+e_1+2C_{G_1}\sim 2h-e_0-e_1$, and
  $$\Bigl(H_Y-\sum_{i=1}^rG_i\Bigr)_{|B}\sim(2h-e_0-e_1)-(h-e_0-e_1)-e_2-\cdots-e_r=h-e_2-\cdots-e_r.$$
  Thus $\w{T}\cong T=\pr_B(\ma{E}_{|B})$ with $\ma{E}_{|B}=\ol_B\oplus\ol_B(h-e_2-\cdots-e_r)$.
  
  In $T$ we have $S\sim\eta_{|T}+\pi_T^*(M)$ for some $M\in\Pic(B)$.
  Moreover  $-K_T=\pi_T^*(-K_B-h+e_2+\cdots+e_r)+2\eta_{|T}$, and $\tilde\alpha^*(T)=\w{T}+\w{E}$ where $\w{E}=\Exc(\tilde\alpha)\subset\w{X}$. Finally $T=\pi^*(B)$ and $-K_Y=2B$, thus $\ol_Y(B)_{|B}=-K_B$ and $\ol_{\w{W}}(T)_{|T}=\ol_T(\pi_T^*(-K_B))$.
Then
  \begin{align*} \ol_{X}(\w{T})_{|\w{T}}=\ &  \ol_{X}\bigl(\alpha^*(T)-\w{E}\bigr)_{|\w{T}}\cong\ol_{\w{W}}(T)_{|T}\otimes\ol_T(-S),\text{ and hence:}\\
    \ol_{\w{X}}(-K_{\w{X}})_{|\w{T}}=\ & \ol_{\w{T}}(-K_{\w{T}})\otimes \ol_{\w{X}}(\w{T})_{|\w{T}}\cong\ol_T(-K_T-S)\otimes\ol_{\w{W}}(T)_{|T}\\
    \cong\ & \ol_T\bigl(\pi_T^*(-2K_B-h+e_2+\cdots+e_r)+2\eta_{|T}-S\bigr)\\
    =\ & \ol_T\bigl(\pi_T^*(-K_B+2h-e_0-e_1-M)+\eta_{|T}\bigr).
  \end{align*}
  
 Therefore $-K_{\w{X}|\w{T}}$ is isomorphic to the tautological class for
  $$T=\pr_B\bigl(\ol_B(-K_B+2h-e_0-e_1-M)\oplus\ol_B(-2K_B-M)\bigr),$$ and we have to show that both linear summands are ample on $B$.
 
  We compute $M$ by restricting to $T\cap J\cong B$ the relation $S\sim\eta_{|T}+\pi_T^*(M)$ of divisors in $T$. We have
  $$\eta_{|T\cap J}\cong(\eta_{|J})_{|T\cap J}\cong \Bigl(H_Y-\sum_{i=1}^rG_i\Bigr)_{|B}=h-e_2-\cdots-e_r.$$

  Now we need to compute the class of $S\cap J$ in $T\cap J$.
 Since $Q\subset\pr^4$ is a quadric cone with vertex $q_0$, $Q\cap J_0$ is a quadric surface containing $q_1,\dotsc,q_r$, isomorphic to $B_0$ via the projection from $q_0$.
 Moreover $C:=A\cap J_0$ is a twisted cubic in $J_0\cong\pr^3$, and as a curve in $Q\cap J_0\cong\pr^1\times\pr^1$ it has degree $(2,1)$  (see for instance \cite[Ex.~2.16]{harris}). The map $\sigma_{|T\cap J}\colon T\cap J\to Q\cap J_0$ blows-up  $q_1,\dotsc,q_r$. The transform of $C$ in $T\cap J\cong B$ is $S\cap J$, and has class:
\begin{align*}
   & \sigma_{|T\cap J}^*\ol(2,1)-C_{G_1}-e_2-\cdots-e_r\\
   & \sim 2(e_0+C_{G_1}) + (e_1+C_{G_1})-C_{G_1}-e_2-\cdots-e_r=\\
   & =2C_{G_1}+2e_0+e_1-e_2-\cdots-e_r\sim 2h-e_1-\cdots-e_r.
\end{align*}
We conclude that $M\sim  2h-e_1-\cdots-e_r-(h-e_2-\cdots-e_r)=h-e_1$.
Finally
$$-K_B+2h-e_0-e_1-M\sim -K_B+h-e_0\quad\text{and}\quad 
-2K_B-M\sim -K_B+2h -e_0-e_2-\cdots-e_r.$$
Now $h-e_0$ and $2h-e_0-e_2-\cdots-e_r$ are nef (since $r\leq 4$) and $-K_B$ is ample, so both summands are ample.
\end{proof}
\begin{lemma}\label{Ea}
 For $r\in\{0,\dotsc,4\}$, $-K_{X|E}$ is ample.
\end{lemma}
\begin{proof}
We show that $\ma{N}_{S/W}\otimes\ol_S(-K_S)$ is ample; this implies the statement by Lemma \ref{tautological}.
Since $S$ is contained in the open subset where the map $W\dasharrow\w{W}$ is an isomorphism, it is equivalent to show that  $\ma{N}_{S/\w{W}}\otimes\ol_S(-K_S)$ is ample.

Let us consider  $S\subset T\subset \w{W}$ and the associated normal bundle sequence:
$$0\la \ma{N}_{S/T}\otimes\ol_S(-K_S)\la\ma{N}_{S/\w{W}}\otimes\ol_S(-K_S)\la
\ma{N}_{T/\w{W}|S}\otimes\ol_S(-K_S)\la 0.$$
Recall from the proof of Lemma \ref{Ta}, whose notation we keep, that
$\ol_{\w{W}}(T)_{|T}=\ol_T(\pi_T^*(-K_B))$. Then
we have $\ma{N}_{T/\w{W}|S}=\ol_{\w{W}}(T)_{|S}=\ol_S(\pi_{|S}^*(-K_B))$ nef in $S$, thus $\ma{N}_{T/\w{W}|S}\otimes\ol_S(-K_S)$
is ample.

Still from the proof of Lemma \ref{Ta} we have that $S\sim\eta_{|T}+\pi_T^*(M)$ is a tautological divisor for
$$T=\pr_B(\ma{E}_{|B}\otimes M)=\pr_B\bigl(\ol_B(h-e_1)\oplus\ol_B(2h-e_1-\cdots-e_r)\bigr).$$
Both $h-e_1$ and $2h-e_1-\cdots-e_r$ are nef in $B$ (since $r\leq 4$), thus  
$S$ is nef in $T$, and $\ma{N}_{S/T}\otimes\ol_S(-K_S)=\ol_T(S)_{|S}\otimes\ol_S(-K_S)$ is ample.

From the exact sequence above we conclude that $\ma{N}_{S/\w{W}}\otimes\ol_S(-K_S)$ is ample.
\end{proof}
\begin{proof}[Proof of Prop.~\ref{a}]
  For $i=0,\dotsc,r$  let us consider the transform $T_i\subset {W}$ of the quadric cone in $\pr^4$ containing $A$ and with vertex $q_i$.
  
  In ${W}$ we have
  $T_i\sim 2H-\sum_{j=0}^rD_j-D_i$, $\sum_{i=0}^rT_i\sim  2(r+1)H-(r+2)\sum_{j=0}^r D_j$,
$-K_W=5H-3\sum_{i=0}^rD_i$,
  and
  $(r+2)(-K_W)=(4-r)H+3\sum_{i=0}^rT_i$.
 We also have $\alpha^*(T_i)= \w{T}_i+E$, where $\w{T}_i\subset X$ is again the transform of $T_i$, and $\alpha^*(-K_W)=-K_X+E$, which gives
 \stepcounter{thm}
 \begin{equation}\label{-K}
   (r+2)(-K_X)=(4-r)H_X+3\sum_{i=0}^r\w{T}_i+(2r+1)E.\end{equation}
 
 We have $-K_{X|E}$ ample by Lemma \ref{Ea}, and in the notation of Lemma \ref{Ta} we have $\w{T}=\w{T}_0$ and $-K_{X|\w{T}_0}$ ample. By considering the projection from $q_i$ instead of $q_0$, again by Lemma \ref{Ta} we get $-K_{X|\w{T}_i}$ ample for every $i=0,\dotsc,r$.

Since $\sigma^*\ol_{\pr^4}(1)$ is nef in $\wi{W}$, for its transform $H$ in $W$ we have $H\cdot\Gamma\geq 0$ for every irreducible curve $\Gamma$ not contained in the indeterminacy locus of $W\dasharrow\wi{W}$.
 
 Let us consider an irreducible curve $C\subset X$. If $C\subset E$, or if $C\subset \w{T}_i$ for some $i$, then $-K_X\cdot C>0$ by what precedes.
 Assume that $C\not\subset E$ and $C\not\subset\w{T}_i$ for every $i$, and let $C_W\subset W$ be the transform of $C$. If $C\cap E=\emptyset$, then $-K_X\cdot C=-K_W\cdot C_W>0$. If $C\cap E\neq\emptyset$, then $C_W\cap S\neq\emptyset$, therefore $C_W$ is not contained in the indeterminacy locus of the map $W\dasharrow\wi{W}$. Then $H_X\cdot C=H\cdot C_W\geq 0$, $\w{T}_i\cdot C\geq 0$ for every $i$, $E\cdot C>0$, and finally $-K_X\cdot C>0$ by \eqref{-K}.

This shows that $-K_X$ is strictly nef. Moreover 
  one can check directly that
 $K_X^4>0$ (see Table \ref{t2}), so that $-K_X$ is big, and it is ample by the base point free theorem. 
\end{proof}
 \begin{table}[h]
 $\begin{array}{||c|c|c|c|c|c|c|c||}
\hline\hline
r   & \rho_X & K_{X}^4&  K_X^2\cdot c_2(X) & b_4(X)=h^{2,2}(X) & b_3(X)&h^0(X,-K_X) &\chi(T_X)\\
    \hline\hline

0 & 3 & 303 & 174 & 5 &0 & 66&4\\
    
      \hline

       1 & 4 & 256 & 160 & 8 &  0 &57 & 2 \\

       \hline

     2 & 5 & 210 & 144 &  12 &  0 &48 & 0 \\

  \hline

      3 & 6 & 165 & 126 &  17 &  0 & 39 &  -2  \\

        \hline

      4  & 7 & 121 & 106 &  23  &  0 & 30 &  -4\\
      
\hline\hline
   \end{array}$

   \bigskip

   \caption{Numerical invariants of the Fano $4$-folds from \S\ref{newa}}\label{t2}
 \end{table}
 \begin{remark}\label{relation}
   For $r=0$, $X_0$ is the blow-up of $W_0:=\Bl_{q_0}{\pr^4}$ along the transform $S_0$ of $A$. Let
   $\w{F}_i\subset X_0$  be  the transform of the line $\overline{q_0q_i}\subset\pr^4$, for $i=1,\dotsc,r$, and let $\wi{X}\to X_0$ be the blow-up of the $r$ curves $\w{F}_1,\dotsc,\w{F}_r$.
   Then $X$ is a SQM of $\wi{X}$, compare diagram \eqref{diagramW}. 
 \end{remark}
   \subsection{New families as in case $(b)$}\label{newb}
\noindent We keep the notation as in \S\ref{setup}.
Let $Q\subset\pr^4$ be a general  quadric hypersurface containing $q_0,\dotsc,q_r$. Moreover
let $M\subset\pr^4$ be a general cubic hypersurface with double points at $q_0,\dotsc,q_r$ and
such that, if 
$A:=Q\cap M$, then $A$ is a reduced and irreducible surface with $\Sing(A)=\{q_0,\dotsc,q_r\}$,
having rational double points of type $A_1$ or $A_2$ in $q_i$ for every $i$; $A$ is a sextic (singular) K3 surface.

Let  $\wi{Q}$ and  $\wi{M}$, $Q_W$ and $M_W$ be the transforms of $Q$ and $M$ in $\wi{W}$ and $W$ respectively; 
  then $\wi{Q}$ is  contained in the open subset where the map $\wi{W}\dasharrow W$ is an isomorphism.

  The transform $S\subset \wi{W}$  of $A$ is a smooth K3 surface, and since $S\subset\wi{Q}$, $S$ is contained in the open subset where the map $\wi{W}\dasharrow W$ is an isomorphism.
  We will show the following.
  \begin{proposition}\label{fanob}
     The $4$-fold $X$ is Fano if and only if $r\in \{0,\dotsc,4\}$; in these cases we have
      $\rho_X=r+3\in\{3,\dotsc,7\}$.  
\end{proposition}
\begin{remark}
For $r=0$ and $\rho_X=3$, this Fano $4$-fold is the same as \cite[K3-50]{BFMT}.
 \end{remark} 
  
 The surface $B_0$ is a quartic surface with double points at $p_1,\dotsc,p_r$, and $B$ is a smooth K3 surface with $B\in |-K_Y|$.  We note that, by generality of $Q$ and $M$, $A$ does not contain lines through $q_0,\dotsc,q_r$, the projection from $q_0$ is finite on $A$, and
$\pi_{|S}\colon S\to B$ is an isomorphism. This also implies that every fiber of $f\colon\w{X}\to Y$ is one-dimensional.

  Set $h:=(\sigma^*\ol_{\pr^4}(1))_{|S}$ nef and big on $S$,  and set $C_i:=\wi{D}_{i|S}$. Note  that $C_i$ is a conic in $\wi{D}_i\cong\pr^3$, smooth if $q_i$ is of type $A_1$ for $A$, reducible if of type $A_2$. Moreover  $C_i$ is a $(-2)$-curve in $S$  if $q_i$ is of type $A_1$, the union of two $(-2)$-curves  if of type $A_2$.
\begin{lemma}\label{normal}
 Let $r\in \{0,\dotsc,4\}$.   We have $\ma{N}_{S/W}\cong \ol_{S}(2h-\sum_{i=0}^rC_i)\oplus\ol_{S}
( 3h-2\sum_{i=0}^rC_i)$, and  $\ma{N}_{S/W}$ is ample.
  \end{lemma}
  \begin{proof}
   Since $\wi{Q}$ is contained in the open subset where the map $\wi{W}\dasharrow W$ is an isomorphism, we have $S=Q_W\cap M_W$ in $W$, and $\ma{N}_{S/W}\cong\ol_W(Q_W)_{|S}\oplus\ol_W(M_W)_{|S}\cong\ol_{\wi{W}}(\wi{Q})_{|S}\oplus\ol_{\wi{W}}(\wi{M})_{|S}$.

    We note that $M$ has multiplicity $2$ at each $q_i$, therefore
$\wi{Q}_{|S}\sim 2h-\sum_{i=0}^rC_i$ and $\wi{M}_{|S}\sim 3h-2\sum_{i=0}^rC_i$,
which yields the first statement.

We show ampleness of both linear summands of $\ma{N}_{S/W}$.
We have $-K_{\wi{W}}=\sigma^*\ol_{\pr^4}(5)-3\sum_{i=0}^r \wi{D}_i$,  thus
$$(-K_W)_{|S}\cong (-K_{\wi{W}})_{|S}=5h-3\sum_{i=0}^rC_i$$
is ample in $S$, because $W$ is Fano. Hence
$$3\wi{Q}_{|S}\sim 3\Bigl( 2h-\sum_{i=0}^rC_i\Bigr)=h+\Bigl(5h-3\sum_{i=0}^r
C_i\Bigr)$$
is ample in $S$.

For the second summand, we treat the case $r=4$ maximal, the other cases being simpler. 
We recall that $W$ is a toric Fano $4$-fold, thus it has an explicit combinatorial description, see  \cite[3.5.8(iii)]{bat2}. In particular
the cone $\NE(W)$ has 20 extremal rays, all small. The loci of 10 of these rays $R_i$ are the exceptional planes in the indeterminacy locus of the map ${W}\dasharrow \wi{W}$, corresponding to the exceptional lines $\ell_{ij}\subset\wi{W}$ given by the transforms of the lines $\overline{q_iq_j}\subset\pr^4$. The loci $L_{abc}$ of the remaining 10 extremal rays $R_i'$ are exceptional planes given by the transforms of the planes $P_{abc}$ spanned by $q_a,q_b,q_c$ in $\pr^4$. The birational map $L_{abc}\dasharrow P_{abc}$ is a standard Cremona map, and a line in $L_{abc}$ corresponds to a conic in $P_{abc}$ through $q_a,q_b,q_c$. 

Let us consider now $\ol_W(M_W)\in\Pic(W)$.  We have $\wi{M}\cdot\ell_{ij}=-1$ in $\wi{W}$, thus $M_W\cdot R_i>0$.
On the other hand the previous description of $L_{abc}$ implies that $M_W\cdot R_i'=0$, therefore 
$M_W$ if nef in $W$, and it defines a contraction $\beta\colon W\to W_0$
such that $\NE(\beta)$ is generated by the 10 extremal rays $R_i'$, and
$\Exc(\beta)$ contains all the $L_{abc}$. 
We claim that $\beta$ is birational and small, with exceptional locus equal to the union of all the $L_{abc}$. Indeed one can check that $\sum_{j=0}^4 D_j\cdot R_i'=-2$ for every $i$, therefore if $\Gamma\subset W$ is an irreducible curve contracted by $\beta$, we must have $\sum_{j=0}^4 D_j\cdot \Gamma<0$ and hence $\Gamma\subset D_j$ for some $j$. We have ${D}_j\cong\Bl_{4\pts}\pr^3$, and $\beta_{|D_j}$ is precisely the blow-up of $4$ points in $\pr^3$, so that $\Gamma\subset L_{abc}$ with $a,b,c\neq j$.

 Now we show that $S_W\cap L_{abc}=\emptyset$ for every $a,b,c$.  
 Set for simplicity $P:=P_{abc}\subset \pr^4$. Then $P\cap Q$ is a conic  in $P$ through $q_a,q_b,q_c$, which must be smooth because the points are not aligned and, being general, $Q$ does not contain any line $\overline{q_iq_j}$. 
On the other hand $P\cap M=\overline{q_aq_b}\cup\overline{q_aq_c}\cup\overline{q_bq_c}$; in particular set-theoretically $P\cap A=\{q_a,q_b,q_c\}$.

Let $\wi{P}\subset\wi{W}$ be the transform of $P$, and consider the exceptional divisor $\wi{D}_a$. Then $\wi{P}\cap \wi{D}_a=\Gamma$ line in $\wi{D}_a\cong\pr^3$, and $\wi{M}\cap \Gamma=\{x_b,x_c\}$ where $x_i:=\wi{D}_a\cap\ell_{ai}$ for $i=b,c$. Moreover $\wi{Q}\cap\Gamma$  is a point, corresponding to the tangent direction to the conic $P\cap Q$ at $q_a$. For $i=b,c$ this conic contains $q_i$, thus $\overline{q_aq_i}$ is not tangent to $P\cap Q$, and $\wi{Q}\cap\Gamma\neq x_i$. We conclude that $S_{\wi{W}}\cap\Gamma= \wi{Q}\cap\wi{M}\cap \Gamma=\emptyset$, namely $S_{\wi{W}}\cap\wi{P}\cap \wi{D}_a=\emptyset$, and finally that 
 $S_{\wi{W}}\cap \wi{P}=\emptyset$. Since $S$ is contained in the open subset where the map $\wi{W}\dasharrow W$ is an isomorphism, we still have $S_W\cap L_{abc}=\emptyset$ in $W$. We get $S_W\cap\Exc(\beta)=\emptyset$, therefore $(M_W)_{|S_W}$ is ample.
\end{proof}

We note that  $T\subset\w{W}$ is the transform of the cone over $A$ in $\pr^4$ with vertex $q_0$, and it is disjoint from the transforms of the lines $\overline{q_iq_j}$ for $1\leq i<j\leq r$, thus it is contained in the open subset where 
the map $\w{W}\dasharrow W$ is an isomorphism.
\begin{lemma}\label{Tb}
For $r\in \{0,\dotsc,4\}$, $-K_{X|\w{T}}$ is ample.
\end{lemma}
\begin{proof}
  We have  $-K_{X|\w{T}}\cong -K_{\w{X}|\w{T}}$, thus it is enough to show that $-K_{\w{X}|\w{T}}$ is ample.
 
  Since $\pi_{|S}\colon S\to B$ is an isomorphism, with a slight abuse of notation we still denote by $C_i$ the image of $C_i\subset S$ in $B$; for $i=1,\dotsc,r$ we have $C_i=G_{i|B}$.

  The K3 surface $S\cong B$ has a map to $\pr^4$ with image the sextic $A$, and we denote by $h$ the pullback of $\ol_{\pr^4}(1)$, and a map to $\pr^3$ with image the quartic $B_0$, and we denote by $h_0$ the pullback of $\ol_{\pr^3}(1)$; we have
  $h_0=h-C_0=H_{Y|B}$.

  We have  $\eta_{|{S}}\cong\ol_{\w{W}}(\w{H}-\sum_{i=1}^r\w{D}_i)_{|S}\cong\ol_{S}(h-\sum_{i=1}^rC_i)$,
  and if $L:=\ol_{B}(h-\sum_{i=1}^rC_i)\in\Pic(B)$, the section $S$ of $\pi_T\colon T\to B$ corresponds to a surjection $\phi\colon\ma{E}_{|B}\twoheadrightarrow L$. This implies that:
  $$\ol_T(S)_{|S}\cong\ma{N}_{{S}/{T}}\cong\ker(\phi)^{-1}\otimes L=\det(\ma{E}_{|B})^{-1}\otimes L^{\otimes 2},$$
  and also that in ${T}$ we have ${S}\sim\eta_{|T}+\pi_T^*(L-\det\ma{E}_{|B})$.
  
  Similarly to the proof of Lemma \ref{Ta}, we also have:
  \begin{align*} 
    B_{|B}&\sim (-K_Y)_{|B}=4h_0-2\sum_{i=1}^rC_i=4h-4C_0-2\sum_{i=1}^rC_i,\\
{\w{T}}_{|\w{T}}&\cong\pi_T^*(B)_{|B}-S=-\eta_{|T}+\pi_T^*(B_{|B}+\det\ma{E}_{|B}-L),\\
 -K_{{T}}&=\pi_T^*(-\det\ma{E}_{|B})+2\eta_{|T},\\ -K_{\w{X}|\w{T}}&=-K_{\w{T}}+\w{T}_{|\w{T}}\cong\eta_{|T}+\pi_T^*(B_{|B}-L)=\eta_{|T}+\pi_T^*\Bigl(3h-3C_0-\sum_{i=0}^rC_i\Bigr),\text{ and }\\                      \ma{E}_{|B}&\cong\ol_B\oplus\ol_B\bigl(h_0-\sum_{i=1}^rC_i\bigr)=\ol_B\oplus\ol_B\bigl(h-\sum_{i=0}^rC_i\bigr).
    \end{align*}
  We conclude that $-K_{X|\w{T}}$ is isomorphic to the tautological class for
  \stepcounter{thm}
  \begin{equation}\label{treno}
    {T}\cong\pr_B\Bigl(\ol\bigl(3h-\sum_{i=0}^r C_i-3C_0\bigr)\oplus\ol\bigl(4h-2\sum_{i=0}^rC_i-3C_0\bigr)\Bigr).
    \end{equation}

  In $B$ we have $h\cdot C_0=0$ and $C_0^2=-2$, thus $h_0\cdot C_0=(h-C_0)\cdot C_0=2$, and $C_0\subset B$ is the transform of a conic $\Gamma_0\subset B_0\subset\pr^3$ not containing any $p_i$; we have $\Gamma_0\cong C_0$, thus $\Gamma_0$ is reduced, either smooth or reducible.
  Let $\Pi\subset\pr^3$ be the plane containing $\Gamma_0$, so that $\Pi_{|B_0}=\Gamma_0+\Gamma_0'$ where $\Gamma_0'$ is another conic in $\Pi$, and $(\Gamma_0')^2=-2$ in $B_0$, in particular $\Gamma_0'$ cannot be a double line.
  Moreover $p_i\not\in \Gamma_0'$ for every $i$, otherwise we would not have $K_B\sim 0$. Let $C_0'\subset B$ be the transform of $\Gamma_0'$.

  Set $h':=h-C_0+C_0'=h_0+C_0'$ in $B$.
We have $h_0\sim C_0+C_0'\sim h'-C_0'$, $C_0\sim h'-2C_0'$, and
  $$3h-\sum_{i=0}^r C_i-3C_0\sim 2h'-C_0'-\sum_{i=1}^rC_i,\quad
  4h-2\sum_{i=0}^rC_i-3C_0\sim 3h'-2C_0'-2\sum_{i=1}^rC_i.$$
  Therefore  by \eqref{treno} $-K_{X|\w{T}}$ is isomorphic to the tautological class for:
  \stepcounter{thm}
  \begin{equation}\label{roma} {T}\cong\pr_B\Bigl(\ol\bigl(2h'-C_0'-\sum_{i=1}^rC_i)\oplus\ol\bigl(3h'-2C_0'-2\sum_{i=1}^rC_i\bigr)\Bigr).
    \end{equation}

  We have $h'\cdot C_0'=0$ and $(h')^2=6$, thus $h'$ is nef and big, and Riemann Roch yields $h^0(B, h')=5$.
Since $h_0$ is base point free, the linear system $|h'|$ can have base points only along $C_0'$. On the other hand $H^1(B,h_0)=0$ by Kawamata-Viehweg vanishing, thus the restriction $H^0(B,h')\to H^0(C_0',\ol_{C_0'})=\C$ is surjective, and no point on $C_0'$ is a base point.

We conclude that $h'$ is base-point-free; moreover for every irreducible curve $\Gamma\subset B$ we have $h'\cdot \Gamma=(h_0+C_0')\cdot \Gamma=0$ if and only if $\Gamma$ is among $C_0',C_1,\dotsc,C_r$. Finally, since the linear system $|h_0|=|h'-C_0'|$ defines a birational map, $|h'|$ 
defines a birational map $\sigma_B'\colon B\to\pr^4$ with exceptional locus  $C_0',C_1,\dotsc,C_r$, whose image is a sextic surface $A'$ with isolated singularities at $q_0':=\sigma_B'(C_0')$ and $q_i':=
\sigma_B'(C_i)$ for $i=1,\dotsc,r$.
Thus $\sigma'_B$ factors through $k_{|B}$:
$$\xymatrix{
B\ar[r]_{k_{|B}}\ar@/^1pc/[rr]^{\sigma'_B}&{B_0}\ar[r]&{A'}
}$$
where the map $B_0\to A'$ contracts $\Gamma_0'$ to $q_0'$, and its inverse $A'\dasharrow B_0$
is the projection from $q_0'$. We conclude that $A'$ does not contain lines through $q_0'$,
and that  $q_0',\dotsc,q_r'$ are in general linear position in $\pr^4$, because
$p_1,\dotsc,p_r$ are in general linear position in $\pr^3$.
We also note that by generality of $A$, $\overline{p_ip_j}\not\subset B_0$, thus 
$\overline{q_i'q_j'}\not\subset A'$ for every $i,j$ with $1\leq i<j\leq r$.
 
Let us consider now $\wi{W}':=\Bl_{q_0',\dotsc,q_r'}\pr^4$ and its Fano model $W'$. Note that since $r\leq 4$, these are toric varieties, and $\wi{W}'\cong\wi{W}$, $W'\cong W$. The transform $S'\subset \wi{W}'$ is isomorphic to $B$, and it is contained in the open subset where the map $\wi{W}'\dasharrow W'$ is an isomorphism; we still denote by $S'\subset W'$ its transform.
Then as in the proof of Lemma \ref{normal} we see that $\ol_B(2h'-C_0'-\sum_{i=1}^rC_i)$ and $\ol_B(3h'-2C_0'-2\sum_{i=1}^rC_i)$ are both ample on $B$,
and we get the statement by \eqref{roma}.
\end{proof} 
\begin{proof}[Proof of Prop.~\ref{fanob}]
If $X$ is Fano, then $\rho_X\leq 7$ by Rem.~\ref{GAEL} and Th.~\ref{Bintro}$(ii)$.

Conversely, let us assume that $r\leq 4$, and we show that $X$ is Fano.
  We proceed similarly to the proof of Prop.~\ref{a}.
Let $T_i\subset{W}$ be the transform of the cone in $\pr^4$ over $A$ with vertex $q_i$, for $i=0,\dotsc,r$.

  In ${W}$ we have
  $T_i\sim 2(2H-\sum_{j=0}^rD_j-D_i)$, $\sum_{i=0}^rT_i\sim  4(r+1)H-2(r+2)\sum_{j=0}^r D_j$,
$-K_W=5H-3\sum_{i=0}^rD_i$,
  and
  $2(r+2)(-K_W)=2(4-r)H+3\sum_{i=0}^rT_i$.

  Finally let $\w{T}_i\subset X$ be the transform of $T_i$; then $\alpha^*(T_i)= \w{T}_i+E$, and we get
  \stepcounter{thm}
  \begin{equation}\label{-Kb}
  2(r+2)(-K_X)=2(4-r)H_X+3\sum_{i=0}^r\w{T}_i+(r-1)E.\end{equation}
  
  Since $S$ is a K3 surface and $K_S\sim 0$, by Lemmas \ref{normal} and \ref{tautological} we have that $-K_{X|E}$ is ample.
Moreover in the notation of Lemma \ref{Tb} we have $\w{T}=\w{T}_0$ and $-K_{X|\w{T}_0}$ ample. By considering the projection from $q_i$ instead of $q_0$, again by Lemma \ref{Tb} we get that $-K_{X|\w{T}_i}$ is ample for every $i=0,\dotsc,r$.

Let $r\in\{1,\dotsc,4\}$. Then as in the proof of Prop.~\ref{a}, using \eqref{-Kb} we show that 
 $-K_X$ is strictly nef. Moreover 
 $K_X^4>0$ (see Table \ref{t3}), which gives the statement. 
 The case $r=0$ is simpler and is left to the reader. 
\end{proof}
 \vspace{-10pt}
 \begin{table}[h]
   $\begin{array}{||c|c|c|c|c|c|c|c|c||}
\hline\hline
r   & \rho_X & K_{X}^4&  K_X^2\cdot c_2(X) & h^{2,2}(X) & h^{1,3}(X)&b_3(X)&h^0(X,-K_X) &\chi(T_X)\\
\hline\hline

       0 & 3 & 180 & 144 &  22& 1 & 0 & 43 & -18\\

       \hline

         1 & 4 & 150 & 132 & 24 & 1 &0 & 37  & -17\\

       \hline

     2 & 5 & 121 & 118 & 27 & 1 &0 & 31  & -16\\

  \hline

      3 & 6 & 93 & 102 & 31 & 1 & 0 & 25 & -15  \\

        \hline

      4  & 7 & 66 & 84 &  36 & 1 & 0 & 19 & -14 \\
      
\hline\hline
     \end{array}$

     \bigskip

   \caption{Numerical invariants of the Fano $4$-folds from \S\ref{newb}}\label{t3}
 \end{table}
     \subsection{New families with $\delta_X=2$ from case $(a)$}\label{ex_delta2}
\noindent  We keep the notation as in \S\ref{setup}.
Let $A\subset\pr^4$ be a general quadric surface containing
$q_1,\dotsc,q_r$ and not $q_0$, in particular $A$ is contained in a hyperplane not passing through $q_0$;
 recall that $A\cong\pr^1 \times \pr^1$.
 Then $S\subset\wi{W}$ is a del Pezzo surface with $\rho_S=r+2$, and it is disjoint from the transforms of the lines $\overline{q_iq_j}$, thus $S$ is contained in the open subset where the map $\wi{W}\dasharrow W$ is an isomorphism.
  We will show the following.
   \begin{proposition}\label{a_delta2}
    The $4$-fold $X$ is Fano if and only if $r\in \{0,1,2\}$. In these cases we have
     $\rho_X=r+3\in\{3,4,5\}$ and $\delta_X=2$.
\end{proposition}
\begin{remark}\label{esempiosaverio}
For $r=0$ and $\rho_X=3$, this Fano $4$-fold is the same as \cite[$X^7_{1,2}$]{saverio}.
 \end{remark} 
 
The surface $B_0\cong A$ is a smooth quadric surface through $p_1,\dotsc,p_r$, and $B$ is a smooth del Pezzo surface with $\rho_B=r+2$ and $-K_Y=2B$ (this is the same as in \S\ref{newa}); moreover
$S\cong B$.

We note that  $T\subset\w{W}$ is the transform of the quadric cone $Q\subset\pr^4$, containing $A$, with vertex  $q_0$. Then $T$ is disjoint from the transforms of the lines $\overline{q_iq_j}$ for $1\leq i<j\leq r$, thus it is contained in the open subset where 
the birational map $\w{W}\dasharrow W$ is an isomorphism. 
\begin{lemma}\label{Ta_delta2}
$-K_{X|\w{T}}$ is ample.
\end{lemma}
\begin{proof}
  We  treat the case $r>0$, the case $r=0$ being already in \cite{saverio}, see Rem.~\ref{esempiosaverio}.
We have  $-K_{X|\w{T}}\cong -K_{\w{X}|\w{T}}$, thus it is enough to show that $-K_{\w{X}|\w{T}}$ is ample.
  
 We follow the same notation for $B$ as in the proof of Lemma \ref{Ta}, that is we have $B=\Bl_{p_1,\dotsc,p_r}B_0\cong\Bl_{r+1\,\pts}\pr^2$ with exceptional divisors $e_0,e_1,\dotsc,e_r$ over $\pr^2$. Let $h\in\Pic(B)$ be the pullback of $\ol_{\pr^2}(1)$; then we have
  $\w{T}\cong T=\pr_B(\ma{E}_{|B})$ with $\ma{E}_{|B}=\ol_B\oplus\ol_B(h-e_2-\cdots-e_r)$.

We can choose the hyperplane $J_0\subset\pr^4$ containing $q_1,\dotsc,q_r$ as the hyperplane containing $A$. Then $S\subset\w{W}$ is the complete intersection of $J$ and $T$, hence
 $\eta_{|T} \cong\ol_T(J_{|T})=\ol_T(S)$. 
By construction and adjunction we have
$${-K_{\w X}}_{|\w T} \sim {-K_{\w W}}_{|T} - S\cong \pi_{|T}^*(-2K_B - (h-e_2-\cdots-e_r)) + \eta_{|T}$$
and $-2K_B - (h-e_2-\cdots-e_r)=-K_B+2h-e_0-e_1$,
thus ${-K_{\w X}}_{|\w T}$ is a tautological divisor of $\w T\cong \pr_B(\ol(-K_B+2h-e_0-e_1)\oplus\ol(-2K_B))$. 
We have that $-K_B$ and $2h-e_0-e_1$ are respectively an ample and a nef divisor on $B$, thus their sum is ample on $B$, and this concludes the proof.
\end{proof}

\begin{lemma}\label{Ea_delta2}
$-K_{X|E}$ is ample if and only if $r\in\{0,1,2\}$.
\end{lemma}
\begin{proof}
 Again we  treat the case $r>0$, the case $r=0$ being already in \cite{saverio}, see Rem.~\ref{esempiosaverio}.
  We show that $\ma{N}_{S/W}^{\vee}\otimes\ol_{W}(-K_{W})_{|S}$ is ample if and only if $r\in\{0,1,2\}$; this implies the statement by Lemma \ref{tautological}.
Since $S$ is contained in the open subset where the map $W\dasharrow\w{W}$ is an isomorphism, it is equivalent to work with $\ma{N}_{S/{\w W}}^{\vee}\otimes\ol_{\w W}(-K_{\w W})_{|S}$. 

We keep the same notation as in the proof of Lemma \ref{Ta_delta2}. To compute $\ma{N}_{S/{\w W}}$, we note that $S$ is the complete intersection of $J$ and $T$, so that
$$\ma{N}_{S/{\w W}} = \ol_S(J_{|S})\oplus \ol_S(T_{|S}) \cong \ol_B(h-e_2-\cdots-e_r)\oplus \ol_B(-K_B).$$
By adjunction ${-K_{\w W}}_{|S}= -K_S + \det \ma{N}_{S/{\w W}} \cong -2K_B+h-e_2-\cdots-e_r$, so that 
$$\ma{N}_{S/{\w W}}^{\vee}\otimes\ol_{\w W}(-K_{\w W})_{|S} \cong \ol_B(-2K_B)\oplus \ol_B(-K_B+h-e_2-\cdots-e_r).$$

Now $-K_B$ is ample, and both $h$ and $h-e_2$ are nef, thus $\ma{N}_{S/{\w W}}^{\vee}\otimes\ol_{\w W}(-K_{\w W})_{|S}$ is ample
if $r\in\{1,2\}$. 
If instead $r\geq 3$, let $\Gamma\subset B$ be the transform of the line through the images of $e_2$ and $e_3$ in $\pr^2$; then $\Gamma\sim h-e_2-e_3$ and
$(-K_B+h-e_2-\cdots-e_r)\cdot\Gamma=0$, thus $\ma{N}_{S/{\w W}}^{\vee}\otimes\ol_{\w W}(-K_{\w W})_{|S}$ is not ample.
\end{proof}

\begin{proof}[Proof of Prop.~\ref{a_delta2}]
If $X$ is Fano, then $r \leq 2$ by Lemma \ref{Ea_delta2}. 

For the
converse, we set $r=2$, and  show that $X$ is Fano. 
The proof of the ampleness of $-K_X$ for $r \leq 1$ is similar and easier, and we omit it.

We note that $\w{D}_0\subset\w{W}$  is the section of $\pi$ corresponding to the projection of $\ol_Y\oplus\ol_Y(H_Y-G_1-G_2)$ onto $\ol_Y$: then $\w{D}_0\cap J=\emptyset$ and $-K_{\w W} \sim J + \w{D}_0 + 2T$.

Let $\ell_{i,j}\subset \wi W$ be the transform of the line $\overline{q_iq_j}\subset\pr^4$. The map $\wi W \dasharrow \w W$ is the flip of $\ell_{0,1}$ and $\ell_{0,2}$, while the map $\w W \dasharrow W$ is the flip (of the transform) of $\ell_{1,2}$, which is an exceptional line contained in $J$ and disjoint from $S$. Thus $\w{D}_0$ is contained in the open subset where the birational map $\w{W}\dasharrow W$ is an isomorphism. 
We denote by $J'\subset W$ the transform of $J\subset\w{W}$, so that
$J\dasharrow J'$ is the flop of $\ell_{1,2}$, and
$-K_W \sim J' + D_0+ 2T$.

Lastly, $S \subset W$ is the complete intersection of $J'$ and $T$, and so 
\begin{equation}\label{-K_delta2}
-K_X\sim \w{J}'+D'_0+2\w{T} + 2E,
\end{equation}
where $\w{J}',D_0'\subset X$ are the transforms of $J',D_0\subset W$.

We have $-K_{X|\w{T}}$ ample by Lemma \ref{Ta_delta2} and $-K_{X|E}$ ample by Lemma \ref{Ea_delta2}, and we show here below that $-K_{X|D_0'}$ and $-K_{X|\w{J}'}$ are ample as well. Then as in the proof of Prop.~\ref{a}, using \eqref{-K_delta2} and that $W$ is Fano, we show that 
 $-K_X$ is strictly nef. Finally, one can check directly that $K_X^4>0$ (see Table \ref{t4}), which gives the statement.
\begin{prg}
$-K_{X|D_0'}$ is ample.

In fact $D_0$ is disjoint from $J$, thus from $S$, therefore $-K_{X|D_0'} \cong -K_{\w{W}|\w{D}_0}\cong 3H_Y-G_1-G_2$ ample on $Y=\Bl_{p_1,p_2}\pr^3$ (this can be checked directly, as $Y$ is toric, and $\NE(Y)$ is generated by the classes of the lines $C_{G_i}\subset G_i$ for $i=1,2$, and of $\Gamma_{1,2}$ transform of the line $\overline{p_1p_2}$).
\end{prg}
\begin{prg}
$-K_{X|\w{J'}}$ is ample.

We have $\w{J}'\cong J'\cong Y'$ where $Y\dasharrow Y'$ is the flop of $\Gamma_{1,2}$. Again $Y'$ is toric, and $\NE(Y')$ is generated by the classes of the flopping curve $\Gamma'$ and of the transform $F_i$ of a line through $p_i$ in $\pr^3$, for $i=1,2$.

Moreover
$-K_{X|\w{J'}} \cong -K_{W|J'}-S$ is the transform of $-K_{\w W|J}-S\cong 5H_Y-3(G_1+G_2)-B\sim 3H_Y-2(G_1+G_2)$. Since $(3H_Y-2(G_1+G_2))\cdot \Gamma_{1,2}=-1$, this gives $-K_X\cdot F_i=1$ for $i=1,2$
and $-K_X\cdot\Gamma'>0$, hence
$-K_{X|\w{J'}}$ is ample.
\end{prg}
 Lastly, we show that $\delta_X=2$. It is not difficult to see that $\codim\N(\w{J_0},X) =2$, thus $\delta_X\geq 2$.
Moreover $X$ cannot be a product of surfaces, because $W$ is not (see for instance \cite[Lemma 2.10]{eleonora}), thus $\delta_X\leq 3$ by  Th.~\ref{deltageq4}.
If $\delta_X=3$, then $\rho_X \geq 5$ by Th.~\ref{delta=3}, therefore $\rho_X=5$ and $r=2$. By classification  we see that there are no correspondences between the numerical invariants of $X$  and those of Fano $4$-folds with $\delta=3$ and $\rho=5$ (see Table \ref{t4} and \cite[Table 3.4]{delta3_4folds}). Thus $\delta_X= 2$ and this concludes the proof.
\end{proof}
 \begin{table}[h]
$\begin{array}{||c|c|c|c|c|c|c|c||}
\hline\hline
r   & \rho_X & K_{X}^4&  K_X^2\cdot c_2(X) & b_4(X)=h^{2,2}(X) & b_3(X)&h^0(X,-K_X) &\chi(T_X)\\
    \hline\hline

0 & 3 & 350 & 188 & 4 & 0 & 75 & 7 \\
    
      \hline

       1 & 4 & 303 & 174 & 7 & 0 & 66 & 5 \\

       \hline

     2 & 5 & 257 & 158 & 11 & 0 & 57 & 3 \\
      
\hline\hline
   \end{array}$

   \bigskip

   \caption{Numerical invariants of the Fano $4$-folds from \S\ref{ex_delta2}}\label{t4}
 \end{table}
     \subsection{New families with an elementary rational contraction onto a $3$-fold}\label{newel}
     \noindent
Let $Z_0\subset\pr^2\times\pr^3$ be a general divisor of type $(1,1)$, and $\pi_0\colon Z_0\to\pr^3$ the projection. Then $Z_0$ is a Fano $4$-fold with $\rho_{Z_0}=2$, and there is a point $y_0\in\pr^3$ such that $\pi_0$ is smooth with fiber $\pr^1$ over $\pr^3\smallsetminus\{y_0\}$ and $\pi_0^{-1}(y_0)\cong\pr^2$, see \cite[Ex.~11.1]{kachi}.

Let $r\in\{0,\dotsc,7\}$ and let $\w{Z}\to Z_0$ be the blow-up of $r$ general fibers of $\pi_0$; note that $\pi_0$ induces an elementary contraction $\pi\colon \w{Z}\to Y:=\Bl_{r\,\pts}\pr^3$ (in fact $\w Z\subset \pr^2\times Y$), and that there is $y\in Y$ such that  $\pi$ is smooth  with fiber $\pr^1$ over $Y\smallsetminus\{y\}$ and
$\pi^{-1}(y)\cong\pr^2$.
\begin{proposition}\label{fanoel}
  For  $r\in\{0,\dotsc,4\}$ there is a SQM $\w{Z}\dasharrow Z$ such that $Z$ is a smooth Fano $4$-fold with $\rho_Z=r+2\in\{2,\dotsc,6\}$, and $Z\dasharrow Y$ is an elementary contraction.
\end{proposition}
\begin{question}\label{Q3} Does Prop.~\ref{fanoel} hold for $r=5,6,7$ too?
\end{question}
We note that  Prop.~\ref{fanoel} cannot hold for $r\geq 8$ by Th.~\ref{general}.
\begin{remark}
For $r=2$ and $\rho_Z=4$, the Fano $4$-fold $Z$ has the same numerical invariants as \cite[Fano 4-7]{FTT} (see Table \ref{t5}). However we do not know whether the two $4$-folds belong to the same family.
\end{remark}  
\begin{proof}[Proof of Prop.~\ref{fanoel}]
  Let $X$ be the Fano $4$-fold introduced in \S \ref{newa}, with $r\in\{0,\dotsc,4\}$ and $\rho_{X}=r+3\in\{3,\dotsc,7\}$; we keep the same notation as in \S\ref{newa}. The divisor $\w{T}\subset X$ is
  a $\pr^1$-bundle over $B\cong\Bl_{p_1,\dotsc,p_r}B_0$, and $B_0\cong\pr^1\times\pr^1$.

  Set $r=4$, so that $\rho_X=7$. Let us consider $f\colon\w{X}\to Y$  and its two factorizations in elementary contractions (see diagrams \eqref{diagramrays} and \eqref{diagrama}):
$$\xymatrix{X\ar@{-->}[r]\ar[d]_{\alpha'}&{\w{X}}\ar[r]^{\tilde\alpha}\ar[d]_{\tilde\alpha'}\ar[dr]^{f}&{\w{W}}\ar[d]^{\pi}\\
 {W'}  \ar@{-->}[r]& {\w{W}'}\ar[r]_{\pi'}&{Y}
}$$ 
By Lemma \ref{rays}, $\alpha'$ is an elementary contraction of type $(3,2)$ with $\Exc(\alpha')=
\w{T}$, and $W'$ is Fano. Recall from p.~\pageref{fiberb} that every fiber of $f$ over $Y\smallsetminus\{y\}$ is one-dimensional, while $f^{-1}(y)\cong\mathbb{F}_1$, and $\tilde{\alpha}_{|f^{-1}(y)}$ is a $\pr^1$-bundle. Therefore $\tilde{\alpha}'(f^{-1}(y))\cong\pr^2$, every non-trivial fiber of $\tilde\alpha'$  has dimension one, and the same holds for $\alpha'$. By  Th.~\ref{32}, $W'$ is a smooth Fano $4$-fold with $\rho_{W'}=6$, and $\alpha'$ 
is the blow-up of of a smooth irreducible surface $S'\subset W'$. We have $S'\cong B$, so that $S'$ is a del Pezzo surface with $\rho_{S'}=6$. We also note that $\pi'\colon\w{W}'\to Y$ is smooth with fiber $\pr^1$ over $Y\smallsetminus\{y\}$, and $(\pi')^{-1}(y)\cong\pr^2$. We are going to show that $W'\cong Z$ (for $r=4$).
       
We have a diagram like \eqref{diagramW} (see also Rem.~\ref{relation}):
{\small
$$ \xymatrix{X\ar@{-->}[r]\ar[d]_{\alpha'}&{\w{X}}\ar@/_1pc/[dd]_>>>>>{f}\ar[d]^{\tilde{\alpha}'}\ar@{-->}[r]&{\wi{X}}\ar[r]^{\sigma}&{X_0}\ar[d]_{\alpha_0'}\ar@/^1pc/[dd]^{f_0}\ar[dr]^{\alpha_0}&\\
 {W'}\ar@{-->}[r] &{\w{W}'=\wi{W}'}\ar[d]^{\pi'}\ar[rr]^{\sigma_{W'}}&&{W'_0}\ar[d]_{\pi_0'}&{W_0=\Bl_{q_0}\pr^4}\ar[dl]^{\pi_0}\ar[r]&{\pr^4}\\
  &Y\ar[rr]^k&&{\pr^3}&&
}$$}
where $\sigma_{W'}$ blows-up the $4$ smooth fibers $F_i:=\pi_0^{-1}(p_i)$, $i=1,\dotsc,4$, $\alpha_0'$ is the blow-up of a surface $S_0'\cong B_0\cong\pr^1\times\pr^1$, and $\sigma$ blows-up the transforms of $F_i$, $i=1,\dotsc,4$ (see Lemma \ref{bigdiagram}; we are in case $(a)$, with $E_1,E_2$ being $E,\w{T}$).

We note that $W_0'$ is smooth, $\rho_{W_0'}=2$, and $W_0'$ is Fano by Lemma \ref{imageexc}. Moreover $\pi_0'\colon W_0'\to\pr^3$ is a scroll in the sense of adjunction theory, namely there exists $L\in\Pic(W_0')$ such that $L\cdot F=1$ for a general fiber $F$ of $\pi_0'$.
Indeed, let us consider the transform $D_0'\subset W_0'$ of the exceptional divisor $D_0\subset W_0$ of the blow-up $W_0\to\pr^4$. Since the general fiber of $\pi_0$ is contained in the open subset where both $\alpha_0$ and $\alpha_0'$ are isomorphism, we still have $D_0'\cdot F=1$ in $W_0'$, and we set $L:=\ol_{W_0'}(D_0')$.

Smooth Fano $4$-folds with $\rho=2$ and with a scroll structure over $\pr^3$ have been classified in \cite{langer},  there are 18 possibilities. To identify $W_0'$, we compute its numerical invariants, and we get $K_{W_0'}^4=432$ (see Table \ref{t5}). A simple computation
of the anticanonical degrees of the 18 Fano $4$-folds from  \cite{langer} shows that there is only one case with $K^4=432$, and it is precisely $Z_0\subset\pr^2\times\pr^3$ a general divisor of type $(1,1)$ (see \cite[\S 9.1.2, case 5]{langer}). We conclude that $\w{Z}=\w{W}'$, so it has a SQM to a smooth Fano $4$-fold $Z=W'$.
This shows the statement for $r=4$ (and $r=0$).

For smaller $r$ the statement follows from the case $r=4$. Indeed set $Z=Z_r$. For any $r\in\{1,2,3\}$ there is a birational map $Z_4\dasharrow Z_r$ given by a SQM followed by the blow-up of $4-r$ smooth curves (transforms of fibers of $\pi_0$), where $Z_r$ is a smooth Fano $4$-fold which is a SQM of $\w{Z}_r=\w{W}'_r$.

We conclude that, for any $r\in\{0,\dotsc,4\}$,  $\w{T}\subset X$ is
  the exceptional divisor of the blow-up $\alpha'\colon X\to Z$ of a smooth irreducible surface $S'\subset Z$, where  $S'$ is a  del Pezzo surface with $\rho_{S'}=r+2$, and   $S'\cong \pr^1\times\pr^1$ for $r=0$.

We use $\alpha'$ to
 compute the invariants of $Z$ from those of $X$ (note that for $r=0$ we have $Z=Z_0=W_0'$). By the proof of Lemma \ref{Ta} we have that $$\ol_{X}(-\w{T})_{|\w{T}}\cong\ol_{\w{X}}(-\w{T})_{|\w{T}}\cong\ol_T\bigl(\eta_T+\pi_T^*(K_B+M)\bigr)$$
  is the tautological class for
  $$\w{T}=\pr_B\bigl(\ol_B(-2h+e_0+e_2+\cdots+e_r)\oplus\ol_B(-h+e_0)\bigr).$$
  We conclude that $\ma{N}_{S'/Z}\cong\ol_B(2h-e_0-e_2-\cdots-e_r)\oplus\ol_B(h-e_0)$, from which we compute $c_2(\ma{N}_{S'/Z})=1$, $(K_{Z|S'})^2=30-4r$, $K_{S'}\cdot K_{Z|S'}=16-2r$, and finally using Lemma \ref{tautological} and \cite[Lemma 6.25]{vb} the invariants in the table below.
\end{proof}
\begin{table}[h]
     $$\begin{array}{||c|c|c|c|c|c|c|c||}
\hline\hline
r   & \rho_{Z} & K_{Z}^4&  K_{Z}^2\cdot c_2(Z) & b_4(Z)=h^{2,2}(Z) & b_3(Z)&h^0(Z,-K_{Z}) &\chi(T_{Z})\\
    \hline\hline

0 & 2 & 432 & 204 &3  &0 & 90 & 12\\
    
      \hline

       1 & 3 & 368 & 188 & 5 & 0 &78 & 9 \\

       \hline

     2 & 4 & 305 & 170 & 8  &  0 & 66 &6 \\
       
  \hline

      3 & 5 & 243 & 150 & 12 &  0 & 54 &    3\\

        \hline

      4  & 6 & 182 & 128 &  17 &  0 & 42 &0 \\
      
\hline\hline
   \end{array}$$

 \bigskip

   \caption{Numerical invariants of the Fano $4$-folds from \S\ref{newel}}\label{t5}
 \end{table}
\subsection{Other examples}\label{other}
\begin{example}[Toric Fano $4$-folds]
We recall that toric Fano $4$-folds are classified, see \cite{bat2,sato}. In the range $\rho\leq 5$, there are several examples with $\delta=2$, or with a  rational contraction onto a $3$-fold.  For $\rho=6$ the only example with $\delta=2$ is $S\times S$, $S=\Bl_{2\pts}\pr^2$.
\end{example}
\begin{example}[Products with $\pr^1$]
There are 13 Fano $3$-folds $Y$ with $\rho_Y=4$ (see \cite[\S 12.5]{fanoEMS} and \cite{morimukaierratum}), among which 4 are toric; they have Lefschetz defect $\delta_Y=2$ (see \cite[Lemma 13]{rendiconti}). For each such $Y$, $X=\pr^1\times Y$ is a Fano $4$-fold with $\rho_X=5$ and $\delta_X=2$ (see Rem.~\ref{deltaproduct}).
\end{example}
\begin{example}
  Some other examples of Fano $4$-folds with $\delta=2$ and $\rho\in\{4,5\}$ can be obtained as follows.
  Consider the classification of
Fano $4$-folds with $\delta=3$ (Th.~\ref{delta=3}, \cite[Prop.~1.5]{delta3}). Excluding the toric ones, and products, there are four families, with $\rho\in\{5,6\}$. These families are obtained by blowing-up a $\pr^2$-bundle $Z\to T$, where $T$ is $\pr^2$, $\pr^1\times\pr^1$, or $\mathbb{F}_1$, along three pairwise disjoint smooth irreducible surfaces $S_i$ for $i=1,2,3$, where $S_2$ and $S_3$ are sections of the $\pr^2$-bundle, and $S_1$ is a double cover of the base $T$. By blowing-up $Z$ only along $S_1$ and $S_i$, $i\in\{2,3\}$, we get a  (non-toric) Fano $4$-fold $X$ with $\rho_X=\rho_T+3\in\{4,5\}$ and $\delta_X=2$.
\end{example}

\small

\noindent{\bf Acknowledgements.} C.C.\ has been partially supported by PRIN 2022L34E7W ``Moduli spaces and birational geometry'', and both authors are members of GNSAGA, INdAM.

\providecommand{\noop}[1]{}
\providecommand{\bysame}{\leavevmode\hbox to3em{\hrulefill}\thinspace}
\providecommand{\MR}{\relax\ifhmode\unskip\space\fi MR }
\providecommand{\MRhref}[2]{%
  \href{http://www.ams.org/mathscinet-getitem?mr=#1}{#2}
}
\providecommand{\href}[2]{#2}

\end{document}